\documentclass[9pt,technote]{IEEEtran}

\usepackage{siunitx}
\DeclareMathAlphabet{\mathpzc}{OT1}{pzc}{m}{it}
\usepackage{titling}
\newcommand{\subtitle}[1]{%
  \posttitle{%
    \par\end{center}
    \begin{center}\LARGE#1\end{center}
    \vskip0.5em}%
}
\usepackage{indentfirst}
\usepackage{amsmath,bbm}

\usepackage{xfrac}
\usepackage{nicefrac} 
\usepackage[super]{nth}

\usepackage[customcolors,norndcorners]{hf-tikz}
\usepackage{amsthm,mathtools,amssymb,amscd,enumerate,fancybox,framed,multirow,accents,xcolor}
\usepackage[english]{babel}

\usepackage{scalerel}
\usepackage{graphicx}
\usepackage{subfig}%

\usepackage{varwidth} 
\usepackage{graphbox}   
\usepackage{stackengine}
\usetikzlibrary{matrix,chains,positioning,decorations.pathreplacing,arrows,shadings,shadows,shapes.arrows,arrows.meta,backgrounds,tikzmark,intersections}
\usepgflibrary{shapes.multipart}
\usetikzlibrary{shapes}
\tikzset{every picture/.style={remember picture}}

\makeatletter
\let\pgfmathModX=\pgfmathMod@
\usepackage{pgfplots}%
\let\pgfmathMod@=\pgfmathModX
\makeatother

\pgfplotsset{compat=newest} 
\pgfplotsset{plot coordinates/math parser=false} 
\newlength\figureheight 
\newlength\figurewidth 

\usetikzlibrary{calc}

\usepackage{fancyhdr,cancel,array,booktabs,setspace,fancyhdr,color}
\usetikzlibrary{patterns}

\pgfdeclarelayer{background}
\pgfdeclarelayer{foreground}
\pgfsetlayers{background,main,foreground}   

\newcommand{\convexpath}[2]{
[   
    create hullnodes/.code={
        \global\edef\namelist{#1}
        \foreach [count=\counter] \nodename in \namelist {
            \global\edef\numberofnodes{\counter}
            \node at (\nodename) [draw=none,name=hullnode\counter] {};
        }
        \node at (hullnode\numberofnodes) [name=hullnode0,draw=none] {};
        \pgfmathtruncatemacro\lastnumber{\numberofnodes+1}
        \node at (hullnode1) [name=hullnode\lastnumber,draw=none] {};
    },
    create hullnodes
]
($(hullnode1)!#2!-90:(hullnode0)$)
\foreach [
    evaluate=\currentnode as \previousnode using \currentnode-1,
    evaluate=\currentnode as \nextnode using \currentnode+1
    ] \currentnode in {1,...,\numberofnodes} {
  let
    \p1 = ($(hullnode\currentnode)!#2!-90:(hullnode\previousnode)$),
    \p2 = ($(hullnode\currentnode)!#2!90:(hullnode\nextnode)$),
    \p3 = ($(\p1) - (hullnode\currentnode)$),
    \n1 = {atan2(\y3,\x3)},
    \p4 = ($(\p2) - (hullnode\currentnode)$),
    \n2 = {atan2(\y4,\x4)},
    \n{delta} = {-Mod(\n1-\n2,360)}
  in 
    {-- (\p1) arc[start angle=\n1, delta angle=\n{delta}, radius=#2] -- (\p2)}
}
-- cycle
}

\makeatletter
\renewcommand*\env@matrix[1][c]{\hskip -\arraycolsep
  \let\@ifnextchar\new@ifnextchar
  \array{*\c@MaxMatrixCols #1}}
\makeatother

\usetikzlibrary{decorations.text}

\newcommand{\E}{\mathbb{E}}

\makeatletter
\DeclareRobustCommand{\rvdots}{%
  \vbox{
    \baselineskip4\p@\lineskiplimit\z@
    \kern-\p@
    \hbox{.}\hbox{.}\hbox{.}
  }}
\makeatother

\usepackage{geometry}
\usepackage{verbatim}
\usetikzlibrary{shapes.geometric , arrows}
\usetikzlibrary{calc,positioning, fit}
\usepackage{array}

\usepackage[skip=0pt, margin=0cm]{caption}
\usepackage{floatrow}

\usepackage{tabularx}
\usepackage{adjustbox}
\usepackage[ruled, lined, longend, linesnumbered]{algorithm2e}
\PassOptionsToPackage{noend}{algpseudocode}
\usepackage{algpseudocode}
\usepackage{authblk}
\renewcommand{\Re}{\operatorname{Re}}

\renewcommand{\psi}{\Psi}

\newcommand{\supp}{\mathrm{supp\,}}

\DeclareMathOperator*{\argmax}{arg\,max\,}
\DeclareMathOperator*{\argmin}{arg\,min\,}
\DeclareMathOperator*{\esssup}{ess\,sup\,}
\DeclareMathOperator*{\essinf}{ess\,inf\,}
\newcommand{\diag}{\mathop{\mathrm{diag}}}
\newtheorem{theorem}{Theorem}[section]
\newtheorem{corollary}{Corollary}[theorem]
\newtheorem{lemma}[theorem]{Lemma}
\newtheorem{proposition}[theorem]{Proposition}
\theoremstyle{definition}

\theoremstyle{remark}
\newtheorem{remark}{\bf Remark}[section]
\newtheorem{example}{\bf Example}[section]

\DeclarePairedDelimiterX{\infdivx}[2]{(}{)}{%
  \kern1pt #1\delimsize\|\,#2%
}
\newcommand{\infdiv}{D_{\kern-0.5pt K\kern-1pt L\kern-0.5pt}\infdivx}

\makeatletter
\renewcommand*{\thesection}{\arabic{section}}
\makeatother

\let\subparagraph\paragraph
\usepackage[compact]{titlesec}
\let\paragraph\subsection
\titleformat{\subsection}[runin]{\normalfont\normalsize\scshape}{}{1em}{\thesection.\arabic{subsection}.~}

\usepackage{hyperref}

\begin{document}

\title{\textbf{Probabilistic Neural Network: Frequency and Moment Learnings} }

\author[1]{Kyung Soo Rim\thanks{This work was supported by a grant from the National Research Foundation of Korea, Grant No. NRF-2017R1E1A1A03070307.}}
\author[2]{U Jin Choi}
\affil[1]{\rm Department of Mathematics, Sogang University, Seoul, Korea \protect \\
{\small E-mail: \href{mailto:ksrim@sogang.ac.kr}{ksrim@sogang.ac.kr}}}
\affil[2]{\rm Department of Mathematical Sciences, Korea Advanced Institute of Science and Technology, Daejeon, Korea \protect \\
{\small E-mail: \href{mailto:ujchoi@kaist.ac.kr}{ujchoi@kaist.ac.kr}}}
\date{April 16, 2020}                     
\setcounter{Maxaffil}{0}
\renewcommand\Affilfont{\itshape\small
}

\maketitle

\begin{abstract}
We introduce probabilistic neural networks that describe unsupervised synchronous learning on an atomic Hardy space and space of bounded real analytic functions, respectively. For a stationary ergodic vector process, we prove that  the probabilistic neural network yields a unique collection of neurons in global optimization without initialization and back-propagation. 
During learning, we show that all neurons communicate with each other, in the sense of linear combinations, until the learning is finished.
Also, we give convergence results for the stability of neurons, estimation methods, and topological statistics 
to appreciate unsupervised estimation
of a probabilistic neural network.
As application, we attach numerical experiments on samples
drawn by a standing wave.
\end{abstract}


\section{Introduction}

Nowadays, learning algorithms and architectures  are currently being developed dramatically for deep neural networks using the back-propagation of the gradient descent method. 
In the mathematical theory of artificial neural networks, the universal approximation theorems, 
which were proved by G. Cybenko in 1989  (\cite{cybenko}) and by K. Hornik showed in 1991 (\cite{hornik}),
state
that a feed-forward network with hidden layers containing a finite number of neurons can approximate continuous functions on compact subsets of $\mathbb{R}^n$, with an activation function. 
However, it does not concern the algorithmic learnability of those parameters.

The most common form of machine learning including recurrent neural networks, deep or not, is supervised learning. Wearing the back-propagation process in training, it has not provided  global optimization,  
and has been dependent on initial conditions.
There are many splendent results for the neural networks as main books such as \cite[Bishop]{bishop}, \cite[Goodfellow, Bengio, Courville]{goodfellow-bengio-courville}, \cite[Haykin]{haykin},  \cite[Murphy]{murphy}  are well rewritten mathematically, to provide examples.
 As the next generation of deep neural networks, many scientists mention unsupervised learning (\cite[Bengio, Courville, Vincent]{bengio-courville-vincent}) thanks to advances in their architecture and ways of training them. Also, they expect unsupervised learning to become far more important in the long term, because there widely exist unsupervised signals, e.g., that are originated from human and animal learning.

The aim of this article is to find a network that gives the globally optimal unsupervised synchronous learning without any initialization and back-propagation. Using the probabilistic method and  theory of dynamics,  we define two types of probabilistic neural network and  derive unique collections of neurons such that their network probabilities  are the global solution for the observed samples.

 The learning process of the probabilistic neural network is not carried out  sequentially by hierarchical layers but is transmitted to all neurons at the same time as input. 
Once sample data is presented, all neurons in the probabilistic neural network interact simultaneously until the learning is complete. Also, there is no back-propagation. Meanwhile, the better the data (e.g. independent and identical distributed or stationary ergodic data) for learning get, 
the better the probabilistic neural network predicts.
In addition, the probabilistic neural network gives a certain criterion of learning rate, from which we can control the amount of the observed samples.
Thus, the probabilistic neural network is closer to the biological human brain (\cite[Watson]{watson}, \cite[Geirhos, Janssen, Sch\"utt, Rauber, Bethge, Wichmann]{geirhos-janssen-schutt-rauber-bethge-wichmann}).

This article is organized as follows.
In the next section, we define energy and partition functions for a network probability which contains hidden parameters as unknown neurons. In  section 3 we consider the Kullback-Leibler divergence between two probability densities and discuss that the energy function of the network probability is expanded in a space of functions 
as an infinite sum. 
The Fr\'echet derivative  of a cross entropy is obtained to identify hidden parameters which are to be a solution of Fr\'echet  partial differential equations. 
In section 4 we discuss 
the atomic Hardy space and space of bounded real analytic functions, on which we derive the exact forms of hidden parameters which are determined uniquely. 
The functions of the two spaces play a role of energy functions for frequency  and moment learnings, respectively.
As corollaries, we have simpler forms of parameters. Although an energy function may be an infinite series, we prove that the network probability equipped with the partial sums of the energy function converges to the limiting distribution in $L^1$-norm.  
In section 5 we prove that communication emerges between neurons during learning, until the learning is complete. 
Moreover, the learning rates of probabilistic neural networks are defined and explored  in the section of application.
In section 6 we derive a dynamical system 
from a cumulative  distribution function for a time series of random vectors. If a sequence of samples is the stationary ergodic  processes, then we can generate plenty of samples using the Koopman mode decomposition (KMD) of the induced dynamical system, linear or nonlinear (\cite{rowley-mezic-bagheri-schlatter-henningson}). In section 7 we prove that the empirical distribution function derived from a stationary ergodic process converges to the limiting distribution in $L^1$-norm. This guarantees the stability of convergence for empirical distribution functions.  For probabilistic neural networks, we define the active path for a signal, and compute a likelihood of the signal to exist on the network in section 8.  In addition, the physical interpretation is introduced by suitable topology on the active path. 
Furthermore, more elaborate version of examples will be examined in section 9 including topological statistics for estimations.

Throughout this paper we use the following general notations: 
For a random vector $X$ with its value $x$, $X_r$ means a random vector at time $t$ with its value $x_r$ and  $X_{r,k}$ is the $k$th random vector of $X_{r}$ with its value $x_{r,k}$. Moreover, $(X)$ and $(x)$ denote sequences of $X$ and its value $x$, respectively.   
The notation of $\E(X)_p$ means the expectation of $X$ with a probability distribution of $p$. 
For quantities $A$ and $B$,  we write $A\lesssim_{n} B$ if there is a constant $C_n$ which depends only on $n$ such that $A\le C_n B$, where 
possibly depending on some other variables as well, we append them to $n$. 
Also, 
$A\approx B$ means $A\lesssim B$, $B\lesssim A$, and write $A\equiv B$ when $A$ is defined as $B$. 
If an operations appear between multi-indexes, e.g., an inequality, combinatorial notation, partial derivative, etc.,  
it follows the rule of multi-index operations.  
Especially, $\dot{x}$ is the derivative of $x$ with respect to $t$ when $t$ is regarded as time, and
$\mathbb{Z}^\ast$ the set of all non-negative integers.
%
Finally, the notation of $|\cdot|$ denotes the absolute value of a scalar or multi-index, or the euclidean norm of a vector. 
Sometimes one can meet  `$\cdot$' just like $|\cdot|$ without any concrete variable. To avid abusing notations, we omit  any variable if we do not need to.

\smallskip

\section{Energy and partition functions}

Let $X=(X_1,\ldots,X_n)$ be a  random vector  from $P_0$ unknown, namely, a governing probability. 
Assume that there is a probability $P$ of $X$, namely, a network probability, with $Y$ a collection of parameters such that 
\begin{equation} \label{network probability}
\begin{split}
	&P_0(X_1=x_1,\ldots,X_n=x_n) \\
	 &\qquad= P(X_1=x_1,\ldots,X_n=x_n\mid Y=y),
\end{split}	 
\end{equation}
where $y=(y_\alpha)$ is a countable collection of complex numbers.
For simplicity of expression, we also call
$p_0(x_1,\ldots,x_n)$ and  $p(x_1,\ldots,x_n\mid y)$ a governing probability and network probability for an observed sample vector $x=(x_1,\ldots,x_n)$ with a collection of parameters $y=(y_\alpha)$, respectively, which play a role of neurons. 


Throughout the article, assume that probability distributions defined on a Borel $\sigma$-algebra assign a positive probability to a nonempty open set, 
the entropy of $p_0$ is finite, 
and  for each $x$, $p(x\mid y)$ is continuously differentiable as a function of $y$. 
Since $p_0$ does not contain a neuron, it alone cannot describe a neural network.
From  information of $x$ fixed, 
we devote to find $y$ at which  $p(x\mid y)$ equals  $p_0(x)$.

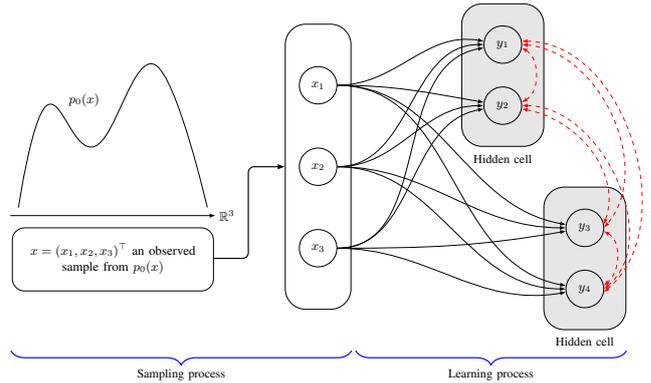
\begin{figure}[tp]
\begin{center}  
%
\tikz \node [scale=0.54, inner sep=0,] {
\begin{tikzpicture}
	\tikzstyle{place}=[circle, text centered, draw=black, text width={width("iMAGi")}]
	\tikzstyle{pretty}=[rectangle, text centered, draw=black, minimum width=0.65cm, minimum height=0.65cm]
	\tikzstyle{ground}=[fill,pattern=north east lines,draw=none,minimum width=0.3,minimum height=0.6]
	
	   \draw[fill=black!0, rounded corners=5mm] (-0.3, 0.5) rectangle (1.3, -6.5) {};
	   \draw[fill=black!10, rounded corners=5mm] (4, 1) rectangle (6, -2.5) {};
	   \draw[fill=black!10, rounded corners=5mm] (6., -3.5) rectangle (8., -7) {};

	\draw node at (0.5, -1) [place] (first_1) {$x_1$}; 
	\draw node at (0.5, -3) [place] (first_2) {$x_2$};  
	\draw node at (0.5, -5) [place] (first_3) {$x_{3}$};	
	
	\node at (5, 0) [place] (1_1)
		{$y_{1}$};  
	\node at (5, -1.5) [place] (1_3)
		{$y_{2}$};

	\node at (7., -4.5) [place] (2_1) {$y_{3}$}; 
	\node at (7., -6) [place] (2_2) {$y_{4}$};

			\draw [->,black,>=latex] (first_1.0) to [out=0,in=170] (1_1.170);
			\draw [->,black,>=latex] (first_1.0) to [out=0,in=170] (1_3.170);
			
			\draw [->,black,>=latex] (first_2.0) to [out=0,in=180] (1_1.180);
			\draw [->,black,>=latex] (first_2.0) to [out=0,in=180] (1_3.180);
			
			\draw [->,black,>=latex] (first_3.0) to [out=0,in=190] (1_1.190);
			\draw [->,black,>=latex] (first_3.0) to [out=0,in=190] (1_3.190);

			\draw [->,black,>=latex] (first_1.0) to [out=0,in=170] (2_1.170);
			\draw [->,black,>=latex] (first_1.0) to [out=0,in=170] (2_2.170);

			\draw [->,black,>=latex] (first_2.0) to [out=0,in=180] (2_1.180);
			\draw [->,black,>=latex] (first_2.0) to [out=0,in=180] (2_2.180);

			\draw [->,black,>=latex] (first_3.0) to [out=0,in=190] (2_1.190);			
			\draw [->,black,>=latex] (first_3.0) to [out=0,in=190] (2_2.190);
			
			\draw [dashed,<->,red,>=latex] (1_1.-10) to [out=-30,in=30] (1_3.10);
			\draw [dashed,<->,red,>=latex] (1_1.0) to [out=-5,in=50] (2_1.0);
			\draw [dashed,<->,red,>=latex] (1_1.10) to [out=0,in=40] (2_2.-10);
			\draw [dashed,<->,red,>=latex] (1_3.-10) to [out=-10,in=70] (2_1.10);
			\draw [dashed,<->,red,>=latex] (1_3.0) to [out=0,in=40] (2_2.0);
			\draw [dashed,<->,red,>=latex] (2_1.-10) to [out=-30,in=40] (2_2.10);

 	\draw [thick, blue,decorate,decoration={brace,amplitude=10pt,mirror},xshift=0.4pt,yshift=-0.4pt](-7, -7.5) -- (1.3, -7.5) node[black,midway,yshift=-0.6cm] {\small Sampling process}; %
	\node [text centered, text width=3.4cm] at (5,-2.8)   {\small Hidden cell} ;
	\node [text centered, text width=3cm] at (7.,-7.3)  {\small Hidden cell} ;
 	\draw [thick, blue,decorate,decoration={brace,amplitude=10pt,mirror},xshift=0.4pt,yshift=-0.4pt](1.4, -7.5) -- (8, -7.5) node[black,midway,yshift=-0.6cm] {\small Learning process}; %
\begin{pgfonlayer}{foreground}
\end{pgfonlayer}

\draw[->,>=latex] (-7,-4.2) -- (-2,-4.2) node[right] {\;$\mathbb{R}^3$};
     \draw  plot[smooth, tension=.7] coordinates{(-6.8,-4) (-6.1176,-1.5) (-4.941,-2.5)  (-3.47,-0.5) (-2.2,-4)};
    \node at (-5.2,-1.35) {$p_0(x)$};    
    \node[align=center] at (-4.5,-5.25) {$x=(x_1,x_2,x_3)^{\top}$ an observed \\ sample from $p_0(x)$};    
   \draw[fill=black!0, rounded corners=3mm, fill opacity=0] (-6.95, -4.5) rectangle (-2.05, -6.05) {};
\draw [thick,->,>=latex,rounded corners] (-2.05,-5.25) -| (-1.2,-3) -- (-0.3,-3);
\end{tikzpicture}
};
\caption{This is an example of signal-flow graph for $p(x\mid y)$. A sample is drawn from $p_0(x)$ from which $(y_\alpha)$ are identified (learned) that satisfy $p_0(x)=p(x\mid y)$. Fo estimation, $y_\alpha$ are classified as several groups, namely, cells. The black arrows mean  flows which start 
from input signals and the red dotted arrows denote communications between parameters with certain relations.} 
\label{fig:M1}
\end{center}
\end{figure}

The fact that a sample $x$ is observed, implies that a certain network probability 
causes $x$, and thus, the network probability  does not vanish identically.
We rewrite $p(x\mid y)$ as a quotient of two positive functions,
\begin{equation} \label{rep pf np}
     p(x\mid y) = \frac{f(x,y)}{g(y)},
\end{equation}
where $f(x,y)$ is an integrable function for $x$ combined with $y$ and
$g(y) = \int_{S}f(x,y)dx$ is a partition function of $y$ with respect to $f$. 
By (\ref{network probability}), it follows that
\begin{equation*} 
    p_0(x) \propto f(x,y)
\end{equation*}
for each $y$. This means that the governing probability is proportional to $f$. 
From $f(x,y)>0$,  (\ref{rep pf np}) is written as  
\begin{equation} \label{rep pf np2}
     p(x\mid y) = \frac{e^{-E(x;y)}}{Z(y)},
\end{equation}
where $E(x;y)=-\ln f(x,y)$ is called an energy function for the network
and rewrite $g(y)$ as $Z(y)$ conventionally.
If the components  $X_1,\ldots,X_n$ of $X$ are \textsc{i.i.d.}, then by (\ref{rep pf np})
$f(x_1,\ldots,x_n,y) = \prod_k f(x_k,y)$, consequently,
\begin{equation} \label{iid}
 	E(x_1,\ldots,x_n;y) = E(x_1;y) + \cdots +  E(x_n;y),
\end{equation}
where $E(x_k;y)=-\ln f(x_k,y)$.

Sometimes, a flow graph is useful to understand a random process.
A signal-flow graph is a network of directed links that are interconnected at certain points called nodes $y_\alpha$. 
A probabilistic neural network  is also a signal-flow graph which consists of an observed sample and parameters of a network probability that satisfies (\ref{network probability}).
A hidden node $y_\alpha$ has associated every input signal $x_k$.
The probabilistic neural network 
 is represented by means of  Figure \ref{fig:M1}
which consists of two parts of the sampling and learning processes. The former provides samples from the data-driven method if we need it, while the latter approximates the values of parameters.  

\smallskip
\section{The Kullback–Leibler divergence} \label{sa}

According to (\ref{network probability}) and (\ref{rep pf np2}), the goal  is to identify $y$ such  that 
\begin{equation} \label{goal}
	p_0(x)=\frac{e^{-E(x;y)}}{Z(y)}.
\end{equation}
We  call a component $y_\alpha$ of $y$ in (\ref{goal}) a neuron of the network probability or governing probability.
To solve the equation (\ref{goal}), information of $E(x;y)$ is very important.
In this study we are
devoted to analyzing it by a linearization which is expressed in a suitable infinite dimensional space.

For  probability distributions $p_1$ and $p_2$ defined on the same probability space, 
the Kullback–Leibler divergence between $p_1$ and $p_2$ is defined by
\begin{equation*}
	\infdiv{p_1}{p_2} = \mathbb{E}\kern -2pt \left(\ln\frac{p_1}{p_2}\right)_{\kern -2pt p_1}, 
\end{equation*}
which is defined only for $x$, where $p_2(x)=0$ implies $p_1(x)=0$. 
Although the Kullback–Leibler divergence is not a distance, it satisfies the following three conditions;
\begin{enumerate}[$(a)$]
\item $\infdiv{p_1}{p_2}\ge0$ \quad (Gibbs' inequality).
\item $\infdiv{p_1}{p_2}=0$\, if and only if \,$p_1=p_2$\; a.e. \quad (identity of indiscernibles).
\item $\infdiv{p_1}{p_2}\ne \infdiv{p_2}{p_1}$ \quad (asymmetricity).
\end{enumerate}
For non-negative measurable functions $f_1$ and $f_2$, by the equality condition of Jensen's inequality, $(b)$ is extended to
$\infdiv{f_1}{f_2}=0$  if and only if 
\begin{equation*}
	f_1=cf_2\;\mbox{ a.e.}
\end{equation*}
for some constant $c$.  


The Kullback-Leibler divergence with $p_0$ and $p$ instead of $p_1$ and $p_2$ yields that 
\begin{equation} \label{cross entropy}
\begin{aligned}
    \infdiv{p_0}{p}
        &=\kern-0.5em\underbrace{-\mathbb{E}\left(\ln p(\;\cdot\mid y)\right)_{p_0}}_{{\mbox{\footnotesize cross entropy of }\,p_0\; \mbox{\footnotesize and}\;p}}\kern-0.5em
         - \underbrace{\big(-\mathbb{E}\left(\ln p_0\right)_{p_0}\big)}_{{\mbox{\footnotesize entropy of }\,p_0}} \\ 
        &\equiv H(p_0,p) - H(p_0)\ge0.
\end{aligned}
\end{equation} 
So, for all $y$, $H(p_0,p)\ge H(p_0)$ and the Kullback-Leibler divergence can be written as the cross entropy of $p_0$ and $p$,  minus the entropy of $p_0$.
To see (\ref{goal}), by the identity of indiscernibles of the Kullback-Leibler divergence, we have to only find $y$ such that $\infdiv{p_0}{p}=0$.
Two quantities of  $\infdiv{p_0}{p}$ and $H(p_0,p)$ are identical by the constant $H(p_0)$ difference. 
The first step toward figuring out the most efficient solution 
is to determine $y$ such that 
\begin{equation} \label{entropy form}
     \argmin_{y} \infdiv{p_0}{p(\,\cdot\mid y)} = \argmin_{y} H(p_0,p(\,\cdot\mid y)).
\end{equation}
Unfortunately, it is  difficult to clarify (\ref{entropy form}) directly because of nonlinearity of $H(p_0,p(\,\cdot\mid y))$. 
To overcome that issue, we will expand the nonlinear energy function in suitable Banach spaces.

Let $r$, $r'$ be a pair of conjugate exponents with $1\le r\le\infty$.
For $1\le r' <\infty$,
let $E(x;y)\in \ell^r(\mathbb{Z}^n,\mathcal{F}(S)) \otimes \ell^{r'}(\mathbb{Z}^n,\mathbb{C})\equiv\ell^r(\mathcal{F}(S))\otimes\ell^{r'}$
be an operator such that
$(\phi_\alpha)\in\ell^r(\mathcal{F}(S))$ and $(y_\alpha)\in\ell^{r'}$, 
where $\mathcal{F}(S)$ is a proper function space on a compact sample space $S$ and $\otimes$  a tensor product. 
If $r'=\infty$, then $(y_\alpha)$ is chosen in $c_0(\mathbb{Z}^n,\mathbb{C})\equiv c_0$ as a subspace of $\ell^{\infty}(\mathbb{Z}^n,\mathbb{C})$. 

Suppose that $E$ has the form of
\begin{equation} \label{linearization}
	E(x;y)=\sum_\alpha\phi_\alpha(x)y_\alpha
\end{equation}
such that for $1\le r<\infty$, 
\begin{equation*}
	\|(\phi_\alpha)\|_{\ell^r} 
\end{equation*}	
is uniformly bounded on $S$,
and for $r=\infty$,
\begin{equation*}
	\sup_\alpha |\phi_\alpha|
\end{equation*}	
is uniformly bounded on $S$.
Note that $E$ converges at every pair of $x$ and $y$ by H\"older's inequality. 
For the partition function of $E$, furthermore, the integrability of $e^{-E(x;y)}$ for any $y$, is always assumed.

\begin{lemma} \label{convexity}
If $E$ satisfies $(\ref{linearization})$, then 
$\infdiv{p_0}{p}$ is well defined and 
its Fr\'echet derivative is induced by 
\begin{equation} \label{frechet}
	\partial \infdiv{p_0}{p}(h)= \sum_\alpha \partial_{y_\alpha} H(p_0,p) h_\alpha
\end{equation}
for $h\in\ell^{r'}$ if $1\le r' <\infty$ and $h\in c_0$ if $r'=\infty$.
\end{lemma}

\begin{proof} 
From (\ref{cross entropy}), we only prove the lemma with $H(p_0,p)$ instead of $\infdiv{p_0}{p}$, since $H(p_0)$ is a positive number.
We first show the summability of $H(p_0,p)$.
By the definition of the cross entropy in (\ref{cross entropy}), 
\begin{equation} \label{expansion of cross entropy}
\begin{aligned}
	H(p_0,p) 
	&=\E(E\big(\,\cdot\,;y) + \ln Z(y)\big)_{\kern-0.1em p_0} \\
	&=\E\Big(
		\sum_\alpha \phi_\alpha(\,\cdot\,)y_\alpha\!\Big)_{\kern-2pt p_0} 
	\!+ \ln\! \int_{S} \!e^{-\kern-1pt \sum_\alpha\kern-2pt  \phi_\alpha\kern-1pt (x)y_\alpha} dx \\
	&\equiv I_1+ I_2.	
\end{aligned}
\end{equation}

By assumption, $I_2$ is readily finite. We show the convergence of $I_1$.
Fix $y$. If $r=\infty$, then
by the triangle inequality, $I_1$ is bounded by
\begin{equation*}
	\E\Big(\!\sup_\alpha|\phi_\alpha|\Big)_{\kern-0.1em p_0}\! \|y\|_{\ell^1}
		\le \big\|\sup_\alpha|\phi_\alpha|\big\|_{L^\infty(S)} \|y\|_{\ell^{r'}}<\infty.
\end{equation*}

For $1\le r<\infty$, by the triangle inequality and 
by H\"older's inequality again, $I_1$ is less than or equal to
\begin{equation}  \label{final bound of I_E}
\begin{aligned}
	&\E\!\left(\!\Big( \sum_{\alpha} |\phi_\alpha|^r \Big)^{1/r} \right)_{\kern-0.2em p_0} \!\! \|y\|_{\ell^{r'}} \\
	&\hspace{1.0cm}\le \big\|\|(\phi_\alpha)\|_{\ell^r} \big\|_{L^\infty(S)}\|y\|_{\ell^{r'}} \\
	&\hspace{1.0cm}<\infty.
\end{aligned}		
\end{equation}

We will find the Fr\'echet derivative for $H(p_0,p)$:
By interchangeability of integral and limit signs, 
the ordinary partial derivative of
$H(p_0,p(\,\cdot \mid (\ldots, y_\alpha,\ldots))$ with respect to $y_\alpha$ exists. Indeed,
\begin{equation} \label{f-derivative}
\begin{aligned} 
	&\partial_{y_\alpha} H(p_0,p(\,\cdot\mid y)) \\
	        &=\E\kern-0pt \left(\partial_{y_\alpha} E(\,\cdot\,;y)\right)_{\kern-0pt p_0}
            		+ \E\kern-0pt \left(\frac{\partial_{y_\alpha} Z(y)}{Z(y)} \right)_{\kern-0pt p_0}\\
	        &=\E\kern-0pt \left( \partial_{y_\alpha} E(\,\cdot\,;y) \right)_{\kern-0pt p_0}
            		+ \frac{\partial_{y_\alpha} Z(y)}{Z(y)} \\
	        &=\E\kern-0pt \left(\partial_{y_\alpha} E(\,\cdot\,;y) \right)_{\kern-0pt p_0}
			- \E\kern-0pt \left(\partial_{y_\alpha} E(\,\cdot\,;y) \right)_{\kern-0pt p} \\
	        &=\E\kern-0pt \left(\phi_\alpha \right)_{\kern-0pt p_0}
			- \E\kern-0pt \left(\phi_\alpha \right)_{\kern-0pt p},
\end{aligned}	 
\end{equation}	
where the last term is finite from the uniform boundedness of $\|(\phi_\alpha)\|_{\ell^r}$ $(1\le r\le\infty)$.

Next, we will prove that  the Fr\'echet derivative  of $H(p_0,p)$
is written as
\begin{equation*}
	\partial H(p_0,p)(h) = \sum_\alpha \partial_{y_\alpha} H(p_0,p)h_\alpha
\end{equation*} 
for $h\in\ell^{r'}$ if $1\le r' <\infty$ and $h\in c_0$ if $r'=\infty$.
%
Actually, 
let $h$ be clipped out such that $|\supp h|<\infty$.
By the chain rule, 
\begin{equation*} 
\begin{aligned}
	&H(p_0,p(\,\cdot\mid y+h))\kern-1pt - \kern-1pt H(p_0,p(\,\cdot\mid y))\kern-1pt 
			-\kern-1pt \sum_\alpha\kern-1pt  \partial_{y_\alpha} \kern-1pt H(p_0,p)  h_\alpha \\
		&= \sum_\alpha \kern-2pt \left(
			 	\int_0^1 \kern-2pt\partial_{y_\alpha} \kern-1pt H(p_0,p(\,\cdot\mid y+th)) \,dt 
				- \partial_{y_\alpha} \kern-1pt H(p_0,p) \!\right) \kern-2pt h_\alpha \\
		&\equiv I\kern-1pt I,	
\end{aligned}
\end{equation*}
where the summation runs over only finite number $\alpha$.

For $r=\infty$, by (\ref{f-derivative}) and by the triangle inequality,
\begin{equation*}
	|I\kern-1pt I| \le \|h\|_{\ell^1}\kern-3pt\int_0^1\!\!\! \int_S\! \sup_\alpha|\phi_\alpha(x)| \big|p(x\mid y+th)\kern-1pt  - \kern-1pt p(x\mid y)\big|dx\kern1pt  dt.
\end{equation*}
By the Lebesgue dominated convergence theorem and by the continuity of $p(x\mid \cdot\,)$,
\begin{equation*}
	\frac{|I\kern-1ptI|}{\|h\|_{\ell^1}} \longrightarrow0
\end{equation*}
as $\|h\|_{\ell^1}\to0$.

If $1\le r<\infty$, then by (\ref{f-derivative}),  by the triangle inequality and by Fubini's theorem,
\begin{equation*} \label{limit1}
\begin{aligned}
	|I\kern-1ptI| &\le \int_0^1\!\!\!\int_S \sum_\alpha|\phi_\alpha h_\alpha|\big|p(x\mid y+th)\kern-1pt -\kern-1pt p(x\mid y)\big|dx\kern1pt dt \\
		&\le \|h\|_{\ell^{r'}}\!\! \int_0^1\!\!\!\int_S \|(\phi_\alpha)\|_{\ell^r} \big|p(x\mid y+th)\kern-1pt -\kern-1pt p(x\mid y)\big|dx\kern1pt dt,
\end{aligned}		
\end{equation*}
where the second inequality comes from H\"older's inequality.
Similarly, 
\begin{equation*} \label{F-deriv-estimate}
	\frac{|I\kern-1ptI|}{\|h\|_{\ell^{r'}}} \longrightarrow0
\end{equation*}
as $\|h\|_{\ell^{r'}}\to0$, .

Thus, 
\begin{equation} \label{functional}
\begin{aligned}
	\partial \infdiv{p_0}{p}(h)
		&= \partial H(p_0,p) \\
		&=\sum_\alpha \partial_{y_\alpha} H(p_0,p) h_{\alpha}
\end{aligned}		
\end{equation}
for $h$ such that $|\supp h| <\infty$. 
Since clipped sequences are dense in $\ell^{r'}$ if $1\le r'<\infty$ and in $c_0$
if $r'=\infty$,
it is sufficient to show that $\partial \infdiv{p_0}{p}$ is a bounded linear functional.

For $1\le r<\infty$, by (\ref{f-derivative}) and by Minkowski inequality,
\begin{equation*}
\begin{aligned}
	\left( \sum_\alpha \big| \partial_{y_\alpha}H(p_0,p)\big|^r\right)^{1/r}\!\! 
			&= \Big(\sum_\alpha \big|\E(\phi_\alpha)_{p_0}-\E(\phi_\alpha)_p\big|^r\Big)^{1/r} \\
			&\le \int_S \|(\phi_\alpha)\|_{\ell^r}(p_0+p)dx \\
			&\le2\big\| \|(\phi_\alpha)\|_{\ell^r}\big\|_{L^\infty(S)} <\infty.
\end{aligned}
\end{equation*}
If $r=\infty$, then by (\ref{f-derivative}) and by the triangle inequality,
\begin{equation*}
	\big|\partial_{y_\alpha}H(p_0,p)\big|
		\le 2\big\|\sup_\alpha |\phi_\alpha| \|_{L^\infty(S)}<\infty.
\end{equation*}
From the form of (\ref{functional}), those give the boundedness of $\infdiv{p_0}{p}$.
\end{proof}

In Theorem \ref{convexity}, we call $\partial_{y_\alpha} H(p_0,p)$ a Fr\'echet partial derivative of $H(p_0,p)$ to distinguish it from ordinary derivatives.
Now 
we are ready to solve (\ref{goal}) partially except for uniqueness.
\begin{theorem} \label{basic equations}
If $E$ satisfies the condition of $(\ref{linearization})$ and $\infdiv{p_0}{p}$ has the minimum at $y$, then 
\begin{equation} \label{epdes1}
	\E\kern-0pt (\phi_\alpha )_{\kern-0pt p_0}
		=\E\kern-0pt(\phi_\alpha)_{\kern-0pt p}
\end{equation}
holds at $y$. Moreover, 
a network probability $p$ for $p_0$ is given by
\begin{equation*}
	p(x\mid y) \equiv \frac{e^{-E(x;y)-s}}{Z}=p_0(x),
\end{equation*}
where  $s=\mathbb{E}\big(\ln \frac1{p_0e^{E}}\big)_{\kern-0.1em p_0}$ and $Z$ is 
the partition function with respect to the energy function $E+s$.
\end{theorem}

\begin{proof}
Let $y$ be a minimum point of $\infdiv{p_0}{p}$.
Combining the boundedness of  the Fr\'echet derivative only with finite sequence $h$ 
with Lemma \ref{convexity}, we have 
\begin{equation*}
\begin{aligned}
	\frac{d}{dt}\infdiv{p_0}{p(\,\cdot\mid y+th)}\Big|_{t=0}
		&= \partial \infdiv{p_0}{p(\,\cdot\mid y)} \\
		&=0.
\end{aligned}		
\end{equation*}
Equivalently, %
from (\ref{f-derivative}),  
\begin{equation} \label{frechet pd}
		\E\kern-0pt (\phi_\alpha)_{\kern-0pt p_0(x)} 
			- \E\kern-0pt (\phi_\alpha)_{\kern-0pt p(x \mid y)}=0
\end{equation} 
for all $\alpha$. 
The obtained energy function $E$ will only differ from it by some added constant.

Let $\tilde p = e^{-\tilde E}$ be a non-negative function, 
where 
$\tilde E=E+s$ with $s=\mathbb{E}\big(\ln\frac1{p_0e^{E}}\big)_{\kern-0.1em p_0}$ evaluated by the solution $y$. 
Note the shift $s$ is finite from (\ref{expansion of cross entropy}) and the finiteness of $H(p_0)$.
It follows that
\begin{equation*}
\begin{aligned}
	\infdiv{p_0}{\tilde p}
		&= \mathbb{E}\left(\ln\frac{p_0}{\tilde p}\right)_{p_0} \\ 
		&= \mathbb{E}(\ln p_0)_{p_0}  -  \mathbb{E}(\ln\tilde p)_{p_0} \\ 
		&= \mathbb{E}(\ln p_0)_{p_0}  + \mathbb{E}(E)_{p_0} +s \\ 
		&= 0,
\end{aligned}
\end{equation*}
where $s$ the smallest number, since $\infdiv{p_0}{p}$ has the minimum at $y$. 
By identity of indiscernibles, $\tilde p=cp_0$ for a positive constant $c$.
The fact of 
\begin{equation*}
	c=\int_S cp_0dx=\int_S\tilde pdx
\end{equation*}
produces that $c$ must be the partition function $Z$ of $\tilde p$. Thus,
we have the desired network probability,
\begin{equation*}
	p(x\mid y)\equiv\frac{e^{-\tilde E(x;y)}}{Z}=p_0(x)
\end{equation*}
equipped with $y$. Therefore, the proof is complete.
\end{proof}

As a special case, if $\infdiv{p_0}{p}$ has the minimum at $y=0$, then
 $E=0$ and $p=p_0$ is a uniform probability distribution on $S$.
It is not simple to apply Theorem \ref{basic equations} 
if the number of of first-order Fr\'echet partial differential equations is not finite.
Nevertheless,
(\ref{epdes1}) is a useful necessary condition to decide extreme points. A suitable solution of (\ref{epdes1}) is a candidate for the minimum point of $\infdiv{p_0}{p}$.

\begin{example} \label{memoryless}
Let $p_0(x)=e^{-2x}/\sinh(2)$ be a memoryless distribution on $[-1,1]$. 
We assume that the governing probability model is given. (In the next section, we introduce sample driven methods without any information on  governing probabilities.)
We may put $f(x)=e^{-E(x;y)}$, where $E(x;y)=y_0+y_1x$. 

If $y_1=0$, then $Z(y_1=0)=2e^{-y_0}$ and $p(x\mid y)=\nicefrac12$ which is a uniform probability distribution on $[-1,1]$. 
A random sample from $p_0$, does not follow the uniform distribution with probability $1$. Thus, we assume $y_1\ne0$, 
the partition function $Z(y_1\ne0)=\frac{2}{y_1}e^{-y_0}\sinh(y_1)$, and
\begin{equation*}
	p(x\mid y_1)=\frac{y_1e^{-y_1 x}}{2\sinh(y_1)}.
\end{equation*}

From $\partial_{y_0} E = \phi_0=1$ and $\partial_{y_1} E = \phi_1=x$, 
(\ref{epdes1}) is calculated at
\begin{equation*}
        \E\left(1\right)_{p_0} = \E\left(1\right)_{p}\,\mbox{ and }\,
        \E\left(X\right)_{p_0} = \E\left(X\right)_{p}.
\end{equation*}
We obtain
$y_1 = 2$ and a network probability is derived as
\begin{equation*}
	p(x\mid y) = p(x\mid y_1=2)=p_0(x).
\end{equation*}
In fact, the collection of $y_1=2$ and $Z(y_1=2)=\sinh(2)$ is determined uniquely by Theorem \ref{PNN by moment} with
\begin{equation*}
	y_1=\frac{d}{dx}\ln \frac1{p_0}\Big|_{x=0}=2, \quad
	\ln Z(y) = \ln\frac1{p_0}\Big|_{x=0}=\ln\sinh(2),
\end{equation*}  
and $y_k=\frac{d^k}{k!dx^k}\ln \frac1{p_0}\Big|_{x=0}=0$ for $k\ge2$.
\end{example}

If $X_1,\ldots,X_n$ are \textsc{i.i.d.}, then Example \ref{memoryless}  can be extended to  
\begin{equation*}
	p(x_1,\ldots,x_n\mid y) = \frac{e^{-2\sum_{k=1}^n x_k}}{\sinh^n(2)}.
\end{equation*}

\smallskip

\section{Probabilistic Neural Networks} \label{pnn}

In this section, we introduce two function spaces: an atomic Hardy space and space of real analytic functions, in which  energy functions will be taken. 
On the spaces, Theorems \ref{PNN by frequency} and \ref{PNN by moment} 
characterize the neurons  using Theorem \ref{basic equations}.
The proofs of two theorems adopt theory of functions to bypass  a large amount 
of calculations of (\ref{epdes1}) and to give the uniqueness.

We say that a distribution $f$ belongs to the $H^1$-Hardy space if for some Schwartz function $\phi$ on $\mathbb{R}^n$ with 
$\int_{\mathbb{R}^n} \phi(x)\,dx \ne0$, the maximal function 
\begin{equation*}
	M_\phi f(x) = \sup_{t>0} |(f\ast \phi_t)(x)|
\end{equation*}
is integrable, where $\phi_t(x)=\phi(x/t)/t^n$. Then the $H^1$-Hardy space is a Banach space with norm $\int_{\mathbb{R}^n}|M_\phi f| dx$.
If we define a function $a$, namely, an $H^1$-atom, such that 
\begin{enumerate}[$(a)$]
	\item $a$ is supported in a cube $Q$,
	\item $|a|\le |Q|^{-1}$ almost everywhere,
	\item $\int_{\mathbb{R}^n} a(x)\,dx=0$,
\end{enumerate}
then by the atomic decomposition theorem, $f$ can be 
written as an infinite linear combination of atoms $a_k$ of
$f=\sum_{k=}^\infty y_k a_k$ whose norm is equivalent to $\sum_{k=1}^\infty|y_k|$,
where $\sum_{k=1}^\infty|y_k|<\infty$ for complex numbers $y_k$
(\cite[Stein]{stein}, \cite[Grafakos]{grafakos}).

Let $\mathbb{T}^n=[-1,\,1)^n$ be a torus.
We define an atomic Hardy space $H_a^1(\mathbb{T}^n) \equiv H_a^1$ by 
\begin{equation*} 
	H_a^1= \left\{ \sum_{\alpha\in\mathbb{Z}^n\setminus\{0\}} y_\alpha\omega_\alpha(x) 
						\middle\vert
						\sum_{\alpha\in\mathbb{Z}^n\setminus\{0\}} |y_\alpha|<\infty\right\}
\end{equation*}
with the norm $\|f\|_{H_a^1}=\sum_{\alpha\ne0} |y_\alpha|$,
where $\omega_\alpha(x)=e^{\pi i\alpha\cdot x}$.
Note that $\omega_\alpha$ satisfies $(a)$, $(b)$, and $(c)$ of $H^1$-atom on $\mathbb{T}^n$
instead of $\mathbb{R}^n$.

We denote the subspace of $L^1(\mathbb{T}^n)$ whose element has zero integral  (i.e., direct current (DC) free)  by $L_0^1(\mathbb{T}^n)\equiv L_0^1$.

\begin{proposition} \label{dense}
The space $H_a^1$ is dense in $L^1_0$.
\end{proposition}

\begin{proof}
Let $\sum_{\alpha\in\mathbb{Z}^n\setminus\{0\}} y_\alpha\omega_\alpha(x) \in H_a^1$.
The series converges uniformly and absolutely. 
So, the limit  is continuous and by interchanging sum with integral, its integral vanishes. Thus, $H_a^1\subset L_0^1$.

For a bounded function in $L_0^1$, 
taking a circular convolution of the function and a sequence of mollifiers, 
by Fubini's theorem 
we have 
an $L_0^1$-convergent sequence of smooth functions.
By the reproducing property of mollifiers, the limit of convolutions recovers the bounded function in the $L^1$-norm. 
Since bounded functions are dense in $ L_0^1$, the set of smooth functions whose integrals vanish, is also dense in $L_0^1$.

Let $f$ be an $m$ times continuously differentiable function in $L_0^1$. If 
$m>1+n/2$,  then by the Fourier series representation for smooth functions,
$f=\sum_{\alpha\ne0} \hat{f}(\alpha)\omega_\alpha$, where $(\hat f(\alpha))\in\ell^1$ and
\begin{equation*}
	\hat f(\alpha) = \frac1{2^n}\int_{\mathbb{T}^n} f(x)\bar\omega_\alpha(x)\,dx
\end{equation*}
is  the $\alpha$th Fourier coefficient of $f$.
The fact of vanishing integral of $f$ implies $f \in H_a^1$. Hence, $H_a^1$ is dense in $L_0^1$ and the proof is complete.
\end{proof}

By normalization, we suppose that a bounded random vector belongs to $\mathbb{T}^n$. 
As one of the main results, the next theorem gives the unique solution of  (\ref{goal}) in the concept of frequency-analysis for $E(x\,;y)\in H_0^1$, which is embedded in $\ell^\infty(\mathcal{F}(\mathbb{T}^n))\otimes\ell^1$.
However, $H_a^1$ is still a candidate space for energy functions until the integrability of $e^{-E}$ is guaranteed. We will prove it in Theorem \ref{decaying}.

\begin{theorem} \label{PNN by frequency}
Let $X=(X_1,\ldots,X_n)\in\mathbb{T}^n$ be a random vector from $p_0$.  
If $E\in H_a^1$, then  $y$ is the unique solution of $p_0(x)=p(x\mid y)$,  determined by
\begin{equation*} 
 	y_\alpha = \widehat{\ln \frac{1}{p_0}}(\alpha),\qquad 
 	\ln Z =  \widehat{\ln \frac{1}{p_0}}(0), 
\end{equation*}
i.e.,
\begin{equation} \label{fn}
 	p_0(x)  =  e^{-\ln Z(y) - \sum_{\alpha\ne0 } y_\alpha \omega_\alpha(x)}.
\end{equation}
\end{theorem}

In Theorem \ref{PNN by frequency},  we denote $\ln Z(y)$ by $D_c$ and call it the direct current (DC) of the network. 
Neurons of the network consists of  $(y_\alpha)_{\alpha\ne0}$ and $D_c$, where
$y_\alpha$
is the $\alpha$th order Fourier coefficient of $\ln 1/p_0$. 
For this reason,
we call the process of Theorem \ref{PNN by frequency} the frequency learning for $X$. In Figure \ref{fig:M4}, we draw the architecture of Theorem \ref{PNN by frequency}.%

\begin{proof}[Proof of Theorem \ref{PNN by frequency}] 
Let $x$ be  the realization vector of $X$.
Clearly, $E(x;y)$ satisfies the condition of (\ref{linearization}). By Lemma \ref{convexity}, 
the Fr\'echet derivative of $E$ with respect to $y_\alpha$ is as follows. For $\alpha\ne0$,
\begin{equation*}
     \partial_{y_\alpha} E = \omega_\alpha(x), 
\end{equation*}
and so,  (\ref{epdes1}) is equivalent to 
\begin{equation*}
    \int_{\mathbb{T}^n} \big( p(x\mid y) -  p_0(x)\big)\omega_\alpha(x)\,dx = 0.
\end{equation*}
By the uniqueness of the Fourier series for the integrable function $p(x\mid y)-p_0(x)$, 
it is constant and must be $0$ from the vanishing of integral.
By continuity,
\begin{equation*} 
    \frac{e^{-E(x;y)}}{Z(y)} = p_0(x),
\end{equation*}
i.e., 
\begin{equation*}
    E(x;y) = -\ln Z(y) - \ln p_0(x).  
\end{equation*}
Also, 
by
the uniqueness of the Fourier series for $E\in H_a^1$, we have
\begin{equation*}
    E(x;y) = -\ln Z(y) - \sum_{\alpha\ne0} \widehat{\ln p_0}(\alpha)\omega_\alpha(x)
\end{equation*}
and
\begin{equation} \label{freq sol}
\begin{aligned}
 	y_\alpha &= \frac1{2^n}\int_{\mathbb{T}^n}\bar\omega_\alpha(x)\ln \frac1{p_0(x)}\,dx,\\
	 \ln Z &= \frac1{2^n}\int_{\mathbb{T}^n}\ln \frac1{p_0(x)}\,dx = D_c.
\end{aligned}     
\end{equation}

The collection of (\ref{freq sol}) is determined uniquely, which 
satisfies $\infdiv{p_0}{p}=0$.
Thus, we get $p(x\mid y)=p_0(x)$ uniquely, and 
therefore, the proof is complete.
\end{proof}

Theorem \ref{PNN by frequency} enables us to design an artificial neural network (Figure \ref{fig:M4}) aiming at the frequency decomposition of energy functions in $H_a^1$.
If  all components  of $X$ are \textsc{i.i.d.}, then  $p_0(x)=\prod_1^np_0(x_k)$ and $E(x;y)=\sum_1^n E(x_k;y)$, 
and  $y_\alpha=0$ except for $\alpha$ to be an integer-lattice on axes.  

\begin{corollary} \label{cor of PNN}
If  all components $X_1, \ldots,X_n$ of $X$ are \textsc{i.i.d.}  from $p_0$, then for $\alpha\ne0$,
\begin{equation*}
\begin{aligned}
	y_{\alpha}
		&= 
		\left\{
 		\begin{array}{cl}  
		 \widehat{\ln \frac{1}{p_0(x_k)}}(\alpha_k)&\mbox{if }\;\alpha_k=|\alpha|  \\
		 0 &\mbox{elsewhere},
 		\end{array}
		\right.\\
	\ln Z(y) &= \sum_{k=1}^n\widehat{\ln \frac{1}{p_0(x_k)}}(0),
\end{aligned}     
\end{equation*}
where $\widehat{\,\cdot\,}$ is a Fourier coefficient  on $\mathbb{T}$.
\end{corollary}

\begin{proof} Let $\alpha\ne0$. From the  \textsc{i.i.d.} property, 
\begin{equation*}
	\ln p_0(x) = \sum_{k=1}^n \ln p_0(x_k).
\end{equation*}
By Theorem \ref{PNN by frequency},  
\begin{align}
	y_\alpha &= \frac1{2^n} \sum_{k=1}^n \int_{\mathbb{T}^n}\bar\omega_\alpha(x)\ln \frac{1}{p_0(x_k)}\,dx
		\label{iid fourier coefficient} \\
	\ln Z &=\frac1{2^n}\sum_{k=1}^n\int_{\mathbb{T}^n}\ln \frac1{p_0(x_k)}\,dx. \notag
\end{align}	 
In the integrand of (\ref{iid fourier coefficient}), $\ln 1/p_0(x_k)$ depends only on $x_k$. Thus, (\ref{iid fourier coefficient}) vanishes if $\alpha$ has a non-trivial component whose index is different from $k$.
\end{proof}

In the process of estimation it is advantageous to classify and analyze neurons into bundles. 
We divide these into disjoint collections. 
For $n$ the number of features, $y_\alpha:\mathbb{Z}^n\to \mathbb{C}$  is a complex-valued function by $\alpha\mapsto y_\alpha$.
For a positive integer $k$, we define the $k$-cell $C_\omega(k)$
of $y_\alpha$ whose frequency order size is $k$, more precisely, 
\begin{equation*}
	C_\omega(k) = \big\{y_\alpha\mid |\alpha|=|\alpha_1|+\cdots+|\alpha_n|=k \big\}.
\end{equation*}
For $k=0$, set $C_\omega(0)=\{D_c\}$ a singleton of zero index.

Sometimes  
it is  more efficient to restrict neurons within an appropriate finite subset. 
We denote 
a finite sub-collection of indexes
by $\mathpzc{D}$, namely,    
a dictionary,   
and its configuration  is freely selectable. 
In order to  appreciate the representation of estimation from neurons,
it is necessary to seize the point of drawn samples.

\makeatletter
\newcommand*\bigcdot{\mathpalette\bigcdot@{.5}}
\newcommand*\bigcdot@[2]{\mathbin{\vcenter{\hbox{\scalebox{#2}{$\m@th#1\bullet$}}}}}
\makeatother

\begin{figure}[htp]
\begin{center}  
\tikz \node [scale=0.53, inner sep=0,] {
\begin{tikzpicture}
	\tikzstyle{place}=[circle, text centered, draw=black, text width={width("iMAGi")}]
	\tikzstyle{pretty}=[rectangle, text centered, draw=black, minimum width=0.65cm, minimum height=0.65cm]
	\tikzstyle{ground}=[fill,pattern=north east lines,draw=none,minimum width=0.3,minimum height=0.6]
	
	   \draw[fill=black!0, rounded corners] (-1, 1) rectangle (1, -7) {};
	   \draw[fill=black!10, rounded corners] (4, 1) rectangle (6, -7) {};
	   \draw[fill=black!10, rounded corners] (7., 1) rectangle (9., -7) {};
	   \draw[fill=black!10, rounded corners] (11, 1) rectangle (13, -7) {};

	\draw node at (0, -1) [place] (first_1) {$x_1$}; 
	\draw node at (0, -3) [place] (first_2) {$x_2$};  
	\draw node at (0, -5) [place] (first_3) {$x_{3}$};	
	
	\node at (5, 0) [place] (1_1)
		{$y_{0\bar10}$};  
	\node at (2.5, -3) [pretty] (1_2)
		{$D_c$};  \node [text centered] at (5,-3)  {$\begin{matrix} {\bigcdot} \\ \bigcdot \\ \bigcdot \end{matrix}$};
	\node at (5, -6) [place] (1_3)
		{$y_{001}$};

	\node at (8., -1.7) [place] (2_1)
		{$y_{0\bar12}$}; 
	\node [text centered] at (8.,-3)  {$\begin{matrix} {\bigcdot} \\ \bigcdot \\ \bigcdot \end{matrix}$};
	\node at (8., -4.3) [place] (2_2)
		{$y_{002}$};
 	\node [text centered] at (10,-3)  {$\begin{matrix} {\bigcdot} & \bigcdot & \bigcdot \end{matrix}$};
 	\node [text centered] at (14.,-3)  {$\begin{matrix} {\bigcdot} & \bigcdot & \bigcdot \end{matrix}$};

	\node at (12, -1) [place] (3_1)
		{$y_{00\bar k}$}; \node [text centered] at (12,-3)  {$\begin{matrix} {\bigcdot} \\ \bigcdot \\ \bigcdot \end{matrix}$};
	\node at (12, -5) [place] (3_2)
		{$y_{00k}$}; 

			\draw [->,black,>=latex] (first_1.0) to [out=0,in=170] (1_1.180);
			\draw [->,black,>=latex] (first_1.0) to [out=0,in=140] (1_2.170);
			\draw [->,black,>=latex] (first_1.0) to [out=0,in=-240] (1_3.-200);

			\draw [->,black,>=latex] (first_2.0) to [out=15,in=190] (1_1.190);
			\draw [->,black,>=latex] (first_2.0) to [out=0,in=180] (1_2.180);
			\draw [->,black,>=latex] (first_2.0) to [out=-15,in=-190] (1_3.-190);

			\draw [->,black,>=latex] (first_3.0) to [out=0,in=240] (1_1.200);
			\draw [->,black,>=latex] (first_3.0) to [out=0,in=-140] (1_2.-170);
			\draw [->,black,>=latex] (first_3.0) to [out=0,in=-170] (1_3.-180);

			\draw [->,black,>=latex] (first_1.0) to [out=0,in=165] (2_1.170);
			\draw [->,black,>=latex] (first_1.0) to [out=0,in=160] (2_2.170);
			
			\draw [->,black,>=latex] (first_2.0) to [out=10,in=180] (2_1.180);
			\draw [->,black,>=latex] (first_2.0) to [out=-10,in=-180] (2_2.180);

			\draw [->,black,>=latex] (first_3.0) to [out=0,in=-160] (2_1.-170);
			\draw [->,black,>=latex] (first_3.0) to [out=0,in=-165] (2_2.-170);			

			\draw [->,black,>=latex] (first_1.0) to [out=0,in=165] (3_1.160);
			\draw [->,black,>=latex] (first_1.0) to [out=0,in=150] (3_2.-180);
			
			\draw [->,black,>=latex] (first_2.0) to [out=17,in=175] (3_1.170);
			\draw [->,black,>=latex] (first_2.0) to [out=-17,in=-175] (3_2.-170);
			
			\draw [->,black,>=latex] (first_3.0) to [out=0,in=-150] (3_1.180);
			\draw [->,black,>=latex] (first_3.0) to [out=0,in=-165] (3_2.-160);
			
			\draw [dashed,<->,red,>=latex] (1_1.-60) to [out=-60,in=25] (1_2.25);
			\draw [dashed,<->,red,>=latex] (1_1.-50) to [out=-50,in=50] (1_3.50);
			\draw [dashed,<->,red,>=latex] (1_1.-40) to [out=-40,in=120] (2_2.120);
			\draw [dashed,<->,red,>=latex] (1_1.-30) to [out=-30,in=120] (2_1.120);
			\draw [dashed,<->,red,>=latex] (1_1.-20) to [out=-10,in=120] (3_2.120);
			\draw [dashed,<->,red,>=latex] (1_1.-10) to [out=0,in=120] (3_1.120);
			\draw [dashed,<->,red,>=latex] (1_3.60) to [out=60,in=-25] (1_2.-25);
			\draw [dashed,<->,red,>=latex] (1_3.40) to [out=40,in=-120] (2_1.-120);
			\draw [dashed,<->,red,>=latex] (1_3.30) to [out=30,in=-120] (2_2.-120);
			\draw [dashed,<->,red,>=latex] (1_3.20) to [out=10,in=-120] (3_1.-120);
			\draw [dashed,<->,red,>=latex] (1_3.10) to [out=0,in=-120] (3_2.-120);
			\draw [dashed,<->,red,>=latex] (2_1.-130) to [out=-160,in=15] (1_2.15);
			\draw [dashed,<->,red,>=latex] (2_1.-50) to [out=-50,in=50] (2_2.50);
			\draw [dashed,<->,red,>=latex] (2_1.50) to [out=50,in=130] (3_1.130);
			\draw [dashed,<->,red,>=latex] (2_1.0) to [out=-20,in=130] (3_2.130);
			\draw [dashed,<->,red,>=latex] (2_2.130) to [out=160,in=-15] (1_2.-15);
			\draw [dashed,<->,red,>=latex] (2_2.-50) to [out=-50,in=-130] (3_2.-130);
			\draw [dashed,<->,red,>=latex] (2_2.0) to [out=20,in=-130] (3_1.-130);
			\draw [dashed,<->,red,>=latex] (3_1.-140) to [out=-140,in=5] (1_2.5);
			\draw [dashed,<->,red,>=latex] (3_1.-50) to [out=-50,in=50] (3_2.50);
			\draw [dashed,<->,red,>=latex] (3_2.140) to [out=140,in=-5] (1_2.-5);
 	\draw [thick, blue,decorate,decoration={brace,amplitude=10pt,mirror},xshift=0.4pt,yshift=-0.4pt](-1, -8.75) -- (1, -8.75) node[black,midway,yshift=-0.6cm] {\small Signal}; %
	\node [text centered, text width=3.4cm] at (2.5,-8)   {\small DC  \\[4pt] \boxed{0$-cell$}} ;
	\node [text centered, text width=3.4cm] at (5,-8)   {\small order size $1$  \\[4pt] \boxed{1$-cell$}} ;
	\node [text centered, text width=3cm] at (8.,-8)  {\small order size $2$  \\[4pt] \boxed{2$-cell$}} ;
	\node [text centered, text width=3cm] at (12.,-8)  {\small order size $k$  \\[4pt] \boxed{k$-cell$}} ;
 	\draw [thick, blue,decorate,decoration={brace,amplitude=10pt,mirror},xshift=0.4pt,yshift=-0.4pt](1.5, -8.75) -- (14.5, -8.75) node[black,midway,yshift=-0.6cm] {\small Neurons}; %
\begin{pgfonlayer}{foreground}
\fill[blue,opacity=0.3] \convexpath{1_2,1_1,1_3,1_1,2_1,1_1}{15pt};
\end{pgfonlayer}
\end{tikzpicture}
};
\caption{Architecture of  the frequency learning  for $(x_1,x_2,x_3)$. The blue  area represents the neurons of the dictionary $\mathpzc{D}=\{0,\,0\bar10,\ldots,001,0\bar10\}$, 
where  $0$ is the index of $D_c$ and $\bar r\equiv -r$ for a positive integer $r$. Any pair of neurons is connected  by a red dashed bidirectional arrow in the sense of Theorem \ref{full learning and bidirection for a frequency learning}. This explains that communication of neurons emerges in the process of frequency learning.} 
\label{fig:M4}
\end{center}
\vspace{-0.1cm}
\end{figure}
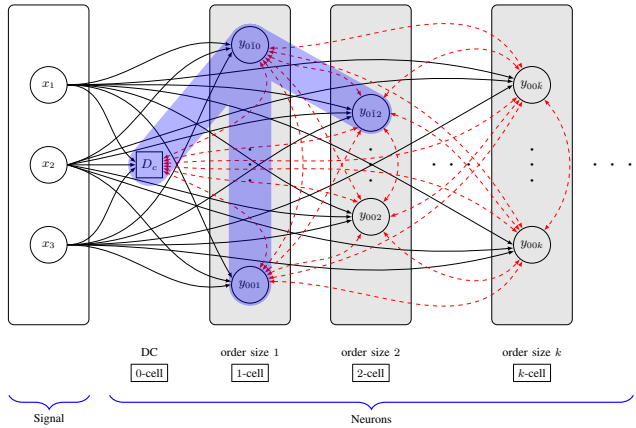

We are occasionally interested in geometric quantities of an energy function, e.g., 
a mean, variation, skewness, or  kurtosis, etc.
To calculate the moments of the energy, 
assume that $p_0$ is sufficiently smooth.
Let $A_{0}(\mathbb{T}^n)\equiv A_{0}$ be the space of all real analytic functions on the interior of  $\mathbb{T}^n$, which fix the origin,
with the supremum norm and write $f\in A_{0}$ as
\begin{equation} \label{series}
	f(x)=\sum_{\alpha\in\mathbb{Z}^{\ast n}\setminus\{0\}}y_\alpha x^\alpha,
\end{equation}
where $y_\alpha\in\mathbb{R}$.
From the definition of $A_0$, $(y_\alpha)\in\ell^1$.

The result below is one of the main theorems that gives the unique solution for  (\ref{goal}) on $A_0$, in the concept of moment-analysis, which is embedded in $\ell^\infty(\mathcal{F}(\mathbb{T}^n))\otimes\ell^1$. Let us note that the integrability of $e^{-E}$ is shown in Theorem \ref{decaying}.

\begin{theorem} \label{PNN by moment}
Let $X=(X_1,\ldots,X_n)\in\mathbb{T}^n$ be a random vector from $p_0$.
If $E\in A_{0}$, then  $y$ is the unique solution of $p_0(x)=p(x\mid y)$, determined by
\begin{equation*} 
 	y_\alpha =  \frac1{\alpha!} \partial_{x}^\alpha \ln \frac1{p_0}\Big\vert_{x=0},
	 \quad\;
 	\ln Z(y) =  \ln \frac1{p_0}\Big\vert_{x=0},
\end{equation*}
i.e.,
\begin{equation} \label{fn2}
 	p_0(x)  =  e^{-\ln Z(y) - \sum_{\alpha\ne0} y_\alpha x^\alpha}.
\end{equation}
\end{theorem}

We denote $\ln Z(y)$ by $D_c$ in Theorem \ref{PNN by moment} and call it the direct current of the network.
Neurons $y_\alpha$ are partial derivatives of $\ln 1/p_0$ at $0$.
By the reason of $y_\alpha$ representing quantitative measures of the shape of $E$,
we call the process of Theorem \ref{PNN by moment} a moment learning for $X$.
For a positive integer $k$, 
we collect $y_\alpha$ having a moment order $\alpha$ such that  
$|\alpha|=k$ and put it by $C_x(k)$, namely, the $k$th cell of $y$. 
In addition, write $C_x(0)=\{D_c\}$.
In Figure \ref{fig:M5} we draw the architecture  of Theorem \ref{PNN by moment}.

\begin{proof}[Proof of Theorem \ref{PNN by moment}]%
Let $x$ be  the realization vector of $X$.
According to $x^\alpha$ is uniformly bounded by 1 and $y\in\ell^1$,
Lemma \ref{convexity} produces that
the Fr\'echet derivative of $E$ with respect to $y_\alpha$ is 
\begin{equation*}
     \partial_{y_\alpha} E= x^\alpha, 
\end{equation*}
and so, for  $\alpha\ne0$,
\begin{equation*}
    \int_{\mathbb{T}^n} \big( p(x\mid y) -  p_0(x)\big)x^\alpha(x)\,dx = 0.
\end{equation*}

The polynomials are dense in $A_0$, that implies that  
the integral of $p(x\mid y)-p_0(x)$ is constant, and besides, the constant must be $0$, since both are probability distributions. 
Hence, there is the unique $y$ such that $p(x\mid y)=p_0(x)$. Analyticity of $E$ turns out to be
\begin{equation*}
	E(x;y)=\sum_{\alpha\ne0} y_\alpha x^\alpha= -\ln Z(y) - \ln p_0(x).
\end{equation*}

By the uniqueness of the power series of (\ref{series}), we have
\begin{equation*}
\begin{aligned}
 	y_\alpha =  \frac1{\alpha!} \partial_{x}^\alpha\ln \frac1{p_0(x)}\Big\vert_{x=0},
	 \quad\;
 	\ln Z(y) =  \ln \frac1{p_0(x)}\Big\vert_{x=0}
\end{aligned}     
\end{equation*}
for $\alpha\ne0$ and 
$\infdiv{p_0}{p}=0$ uniquely. Therefore, $p(x\mid y)=p_0(x)$ and the proof is complete.
\end{proof}

From \textsc{i.i.d.} features,  the non-trivial neurons are located only on axes.

\begin{corollary} \label{cor of moment PNN}
If all components $X_1,\ldots,X_n$ of $X$ are \textsc{i.i.d.}  from $p_0$, then for $\alpha\ne0$,
\begin{equation*}
\begin{aligned}
	y_{\alpha}
		&= 
		\left\{
 		\begin{array}{cl}  
		 \frac{1}{\alpha_k !} \frac{d^{\alpha_k}}{dx_k^{\alpha_k}}\ln \frac1{p_0(x_k)}\Big\vert_{x_k=0} 
		 	&\mbox{if }\;\alpha_k=|\alpha|  \\
		 0 
		 	&\mbox{elsewhere},
 		\end{array}
		\right.\\
	\ln Z(y) &= \sum_{k=1}^n\ln \frac1{p_0(x_k)}\Big\vert_{x_k=0}.
\end{aligned}     
\end{equation*}
\end{corollary}

\begin{proof} By  the  \textsc{i.i.d.} property, 
\begin{equation*}
	\ln p_0(x) = \sum_{k=1}^n \ln p_0(x_k).
\end{equation*}
By Theorem \ref{PNN by moment},
\begin{equation} \label{iid moment coefficient} 
	y_\alpha =  \frac1{\alpha!}\partial_{x}^\alpha \sum_{k=1}^n \ln \frac1{p_0(x_k)}\Big|_{x_k=0}.
\end{equation}
In the summand of (\ref{iid moment coefficient}), $\ln 1/p_0(x_k)$ depends only on $x_k$. Thus, (\ref{iid moment coefficient}) vanishes if at least two components of $\alpha$ are non-zero for $\alpha\ne0$. when $\alpha=0$, the result follows directly. 
%
\end{proof}

%
%
%

%

\begin{figure}[htp]
\begin{center}  
\tikz \node [scale=0.53, inner sep=0,] {
\begin{tikzpicture}
	\tikzstyle{place}=[circle, text centered, draw=black, text width={width("MAG")}]
	\tikzstyle{pretty}=[rectangle, text centered, draw=black, minimum width=0.65cm, minimum height=0.65cm]
	\tikzstyle{ground}=[fill,pattern=north east lines,draw=none,minimum width=0.3,minimum height=0.6]
	
	   \draw[fill=black!0, rounded corners] (-1, 1) rectangle (1, -7) {};
	   \draw[fill=black!10, rounded corners] (4, 1) rectangle (6, -7) {};
	   \draw[fill=black!10, rounded corners] (7., 1) rectangle (9., -7) {};
	   \draw[fill=black!10, rounded corners] (11, 1) rectangle (13, -7) {};

	\draw node at (0, -1) [place] (first_1) {$x_1$}; 
	\draw node at (0, -3) [place] (first_2) {$x_2$};  
	\draw node at (0, -5) [place] (first_3) {$x_{3}$};	
	
	\node at (5, 0) [place] (1_1)
		{$y_{100}$};  
	\node at (2.5, -3) [pretty] (1_2)
		{$D_c$};  \node [text centered] at (5,-3)  {$\begin{matrix} {\bigcdot} \\ \bigcdot \\ \bigcdot \end{matrix}$};
	\node at (5, -6) [place] (1_3)
		{$y_{001}$};

	\node at (8., -1.7) [place] (2_1)
		{$y_{200}$}; \node [text centered] at (8.,-3)  {$\begin{matrix} {\bigcdot} \\ \bigcdot \\ \bigcdot \end{matrix}$};
	\node at (8., -4.3) [place] (2_2)
		{$y_{002}$};
 	\node [text centered] at (10.,-3)  {$\begin{matrix} {\bigcdot} & \bigcdot & \bigcdot \end{matrix}$};
 
	\node at (12, -1) [place] (3_1)
		{$y_{k00}$}; \node [text centered] at (12,-3)  {$\begin{matrix} {\bigcdot} \\ \bigcdot \\ \bigcdot \end{matrix}$};	
	\node at (12, -5) [place] (3_2)
		{$y_{00k}$}; 
 	\node [text centered] at (14.,-3)  {$\begin{matrix} {\bigcdot} & \bigcdot & \bigcdot \end{matrix}$};

			\draw [->,black,>=latex] (first_1.0) to [out=0,in=170] (1_1.180);
			\draw [->,black,>=latex] (first_1.0) to [out=0,in=140] (1_2.170);
			\draw [->,black,>=latex] (first_1.0) to [out=0,in=-240] (1_3.-200);

			\draw [->,black,>=latex] (first_2.0) to [out=15,in=190] (1_1.190);
			\draw [->,black,>=latex] (first_2.0) to [out=0,in=180] (1_2.180);
			\draw [->,black,>=latex] (first_2.0) to [out=-15,in=-190] (1_3.-190);

			\draw [->,black,>=latex] (first_3.0) to [out=0,in=240] (1_1.200);
			\draw [->,black,>=latex] (first_3.0) to [out=0,in=-140] (1_2.-170);
			\draw [->,black,>=latex] (first_3.0) to [out=0,in=-170] (1_3.-180);

			\draw [->,black,>=latex] (first_1.0) to [out=0,in=165] (2_1.170);
			\draw [->,black,>=latex] (first_1.0) to [out=0,in=160] (2_2.170);
			
			\draw [->,black,>=latex] (first_2.0) to [out=10,in=180] (2_1.180);
			\draw [->,black,>=latex] (first_2.0) to [out=-10,in=-180] (2_2.180);

			\draw [->,black,>=latex] (first_3.0) to [out=0,in=-160] (2_1.-170);
			\draw [->,black,>=latex] (first_3.0) to [out=0,in=-165] (2_2.-170);			

			\draw [->,black,>=latex] (first_1.0) to [out=0,in=165] (3_1.160);
			\draw [->,black,>=latex] (first_1.0) to [out=0,in=150] (3_2.-180);
			
			\draw [->,black,>=latex] (first_2.0) to [out=17,in=175] (3_1.170);
			\draw [->,black,>=latex] (first_2.0) to [out=-17,in=-175] (3_2.-170);
			
			\draw [->,black,>=latex] (first_3.0) to [out=0,in=-150] (3_1.180);
			\draw [->,black,>=latex] (first_3.0) to [out=0,in=-165] (3_2.-160);

			\draw [dashed,<->,red,>=latex] (1_1.-60) to [out=-60,in=25] (1_2.25);
			\draw [dashed,<->,red,>=latex] (1_1.-50) to [out=-50,in=50] (1_3.50);
			\draw [dashed,<->,red,>=latex] (1_1.-40) to [out=-40,in=120] (2_2.120);
			\draw [dashed,<->,red,>=latex] (1_1.-30) to [out=-30,in=120] (2_1.120);
			\draw [dashed,<->,red,>=latex] (1_1.-20) to [out=-10,in=120] (3_2.120);
			\draw [dashed,<->,red,>=latex] (1_1.-10) to [out=0,in=120] (3_1.120);
			\draw [dashed,<->,red,>=latex] (1_3.60) to [out=60,in=-25] (1_2.-25);
			\draw [dashed,<->,red,>=latex] (1_3.40) to [out=40,in=-120] (2_1.-120);
			\draw [dashed,<->,red,>=latex] (1_3.30) to [out=30,in=-120] (2_2.-120);
			\draw [dashed,<->,red,>=latex] (1_3.20) to [out=10,in=-120] (3_1.-120);
			\draw [dashed,<->,red,>=latex] (1_3.10) to [out=0,in=-120] (3_2.-120);
			\draw [dashed,<->,red,>=latex] (2_1.-130) to [out=-160,in=15] (1_2.15);
			\draw [dashed,<->,red,>=latex] (2_1.-50) to [out=-50,in=50] (2_2.50);
			\draw [dashed,<->,red,>=latex] (2_1.50) to [out=50,in=130] (3_1.130);
			\draw [dashed,<->,red,>=latex] (2_1.0) to [out=-20,in=130] (3_2.130);
			\draw [dashed,<->,red,>=latex] (2_2.130) to [out=160,in=-15] (1_2.-15);
			\draw [dashed,<->,red,>=latex] (2_2.-50) to [out=-50,in=-130] (3_2.-130);
			\draw [dashed,<->,red,>=latex] (2_2.0) to [out=20,in=-130] (3_1.-130);
			\draw [dashed,<->,red,>=latex] (3_1.-140) to [out=-140,in=5] (1_2.5);
			\draw [dashed,<->,red,>=latex] (3_1.-50) to [out=-50,in=50] (3_2.50);
			\draw [dashed,<->,red,>=latex] (3_2.140) to [out=140,in=-5] (1_2.-5);
 	\draw [thick, blue,decorate,decoration={brace,amplitude=10pt,mirror},xshift=0.4pt,yshift=-0.4pt](-1, -9) -- (1, -9) node[black,midway,yshift=-0.6cm] {\small Signal}; %
	\node [text centered, text width=3.4cm] at (2.5,-8)   {\small DC    \\[4pt] \boxed{0$-cell$}};
	\node [text centered, text width=3.4cm] at (5,-8)   {\small order size $1$ \\[4pt] \boxed{1$-cell$}};
	\node [text centered, text width=3cm] at (8.,-8)  {\small order size $2$ \\[4pt] \boxed{2$-cell$}};
	\node [text centered, text width=3cm] at (12,-8)  {\small order size $k$ \\[4pt] \boxed{k$-cell$}};
 	\draw [thick, blue,decorate,decoration={brace,amplitude=10pt,mirror},xshift=0.4pt,yshift=-0.4pt](1.5, -9) -- (14.5, -9) node[black,midway,yshift=-0.6cm] {\small Neurons}; %
\begin{pgfonlayer}{foreground}
\fill[blue,opacity=0.3] \convexpath{1_2,1_1,1_3,1_1,2_1,1_1}{15pt};
\end{pgfonlayer}
\end{tikzpicture}
};
\caption{Architecture of  the moment learning  for $(x_1,x_2,x_3)$. The blue  area represents the neurons in the dictionary $\mathpzc{D}=\{0,100,010,001,200\}$, where $0$ is the index of $D_c$. Any pair of neurons is connected by a red dashed bidirectional arrow in the sense of Theorem \ref{full learning and bidirection for a moment learning}. This explains that communication of neurons emerges in the process of moment learning.} 
\label{fig:M5}
\end{center}
\vspace{-0.1cm}
\end{figure}
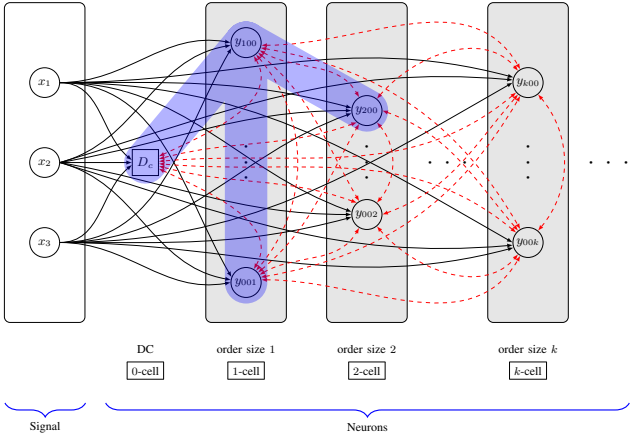

The corollary below shows 
that the lower and upper magnitudes of a governing probability control the size of an energy function
from $\ln 1/p_0(x) = E(x;y) + D_c$. 

\begin{corollary} \label{bound of energy1}
If $y$ is the the solution in Theorem \ref{PNN by frequency} or \ref{PNN by moment}, then
\begin{equation*}
	\essinf_{x}\ln\frac1{p_0(x)} - D_c 
		\le E(x;y) \le 
			\esssup_{x} \ln\frac1{p_0(x)} - D_c
\end{equation*}
for almost every $x$.
\end{corollary}

The next result ensures the $L^1$-norm convergence of approximated network probabilities equipped with a finite sum of $E$.

\begin{theorem} \label{decaying}
Let $E_N(x;y)$ be either $\sum_{0\ne|\alpha|\le N} y_\alpha \omega_\alpha(x)$ or $\sum_{0\ne|\alpha|\le N} y_\alpha x^\alpha$. Then
\begin{equation*}
		\lim_{N\to\infty}\int_{\mathbb{T}^n} \big|e^{-E(x;y)} - e^{-E_N(x;y)}\big|\,dx = 0.
\end{equation*}
\end{theorem}

\begin{proof}
By Taylor's series expansion,  by the triangle inequality, and  by the uniform convergence, 
\begin{equation*}
\begin{aligned}
	\int_{\mathbb{T}^n} & \big|e^{-E(x;y)} - e^{-E_N(x;y)}\big|dx \\
		&\le \sum_{k=1}^\infty\frac{1}{k!}\int_{\mathbb{T}^n} e^{-E(x;y)}|E(x;y)-E_N(x;y)|^k\,dx\\
		&\le \sum_{k=1}^\infty\frac{1}{k!}
			\left(\sum_{|\alpha|>N} |y_\alpha|\right)^k
			\int_{\mathbb{T}^n} e^{-E(x;y)}dx\\
		&= \sum_{k=1}^\infty\frac{1}{k!}
			\left(\sum_{|\alpha|>N} |y_\alpha|\right)^k Z(y) \\
		&= \left(e^{\sum_{|\alpha|>N} |y_\alpha|}-1 \right) Z(y).
\end{aligned}	
\end{equation*}
The last term goes to 0 as $N\to\infty$, since $y\in\ell^1$. Therefore, the proof is complete.
\end{proof}

\begin{remark} \label{description1}
\begin{enumerate}[$(i)$]
\item Theorems \ref{PNN by frequency} and \ref{PNN by moment} provide that the components of a random vector are unidirectionally connected to every neuron simultaneously.
\item One may generalize the energy space to a separable space
that satisfies summable conditions of (\ref{linearization}).
\item Architectures of Theorems \ref{PNN by frequency}, \ref{PNN by moment} say that
cells of higher order neurons are responsible for higher resolutions of frequencies and moments.
\item Theorem  \ref{decaying} guarantees the $L^1$-convergence of $e^{-E_N}$. 
This enables us to approximate likelihoods equipped with $E_N$ instead of $E$.
\end{enumerate}
\end{remark}

\smallskip

\section{Full learning and communication of neurons}
\setcounter{paragraph}{1}

With finite samples, the induced $p_0$ is an approximation.
Moreover, any calculated neuron $\tilde y_\alpha$ is also an approximation of $y_\alpha$  
In this section, 
we will settle a learning status up to the full learning of $\tilde y_\alpha = y_\alpha$.

Since $\mathbb{T}^n$ is  second countable in the standard topology, it is generated by
a countable basis $(B_k)$. 
If $x\in\mathbb{T}^n$ does not have a dense orbit $(x_k)$, then
the orbit is disjoint with some $B_k$. From the assumption in advance that a nonempty open set has a positive probability, we have 
\begin{equation*}
	0=\frac1m\sum_{k=0}^{m-1}\mathbbm{1}_{B_k}(x_k) \ne\mathbb{E}(\mathbbm{1}_{B_k}(X))_{p_0}>0
\end{equation*}
for any $m$, where $\mathbbm{1}_{B_k}$ is an indicator function of $B_k$. By Birkhoff ergodic theorem, the set of all $x$ whose orbit is not dense in $\mathbb{T}^n$ has measure zero. 
Hence, we conclude the following remark.
The Birkhoff ergodic theorem will be described in section  \ref{kmd} more precisely.

\begin{remark} \label{denseness}
For almost every $x$, its orbit is dense in $\mathbb{T}^n$.
\end{remark}

\paragraph{Frequency learning}
Let $(x)$ be stationary ergodic samples. Take a partition of $\mathbb{T}^n$ with non-overlapping rectangles such that every rectangle contains a single sample $x$, whose volume is written as $\Delta_x$.
On $H_a^1$,  from the approximation,  
\begin{equation} \label{app neuron1}
	\tilde y_\alpha  \equiv \frac1{2^n}\sum_{x} \bar\omega_\alpha(x)\ln \frac1{p_0(x)} \Delta_x,
\end{equation}
we have
\begin{equation}\label{connection1}
\begin{aligned}
	\tilde y_\alpha  \approx y_\alpha 
		&= \widehat{\ln \frac1{p_0}}(\alpha) 
		=\widehat{\omega_\beta\ln \frac1{p_0}}(\alpha+\beta) \\ 
		&= \sum_{\gamma\in\mathbb{Z}^n}  
			\widehat{\omega_\beta} (\alpha+\beta-\gamma)(x) 
			\widehat{\ln \frac1{p_0}}(\gamma) \\
		&\approx \frac1{2^n}\sum_{\gamma\in\mathbb{Z}^n}  
			\sum_{x}  \bar \omega_{\alpha-\gamma}(x) \Delta_x\,
			 \tilde y_\gamma,
\end{aligned}		
\end{equation}
where $\tilde y_0$ is an approximation of $D_c$. 
Every $\tilde y_\alpha$ is an infinite linear combination of $\tilde y_\gamma$ approximately, especially, $|y_\gamma|\to0$ as $\tilde\gamma|\to$ is negligible if $|\gamma|$ is sufficiently large according to Theorem \ref{PNN by frequency}.

\begin{theorem} \label{full learning and bidirection for a frequency learning} 
If $x_1,\ldots,x_m$ are realization of a stationary ergodic vector process of $X_1$, $\cdots$, $X_m$, then for any $\alpha$, 
\begin{equation*}
\begin{aligned}
	\tilde y_\alpha \stackrel{(\ref{app neuron1})}{=}
	& \frac1{2^n}\sum_{k=1}^m \bar\omega_\alpha(x_k)\ln  \frac1{p_0(x_k)} \Delta_{x_k}  \\
		& \to y_\alpha \leftarrow
	\frac1{2^n}\sum_{\gamma\in\mathbb{Z}^n}
		\sum_{k=1}^m  
		\bar \omega_{\alpha-\gamma}(x_k) \Delta_{x_k}\tilde y_\gamma \stackrel{(\ref{connection1})}{\approx} \tilde y_\alpha
\end{aligned}		
\end{equation*}
as $m\to\infty$, where $\tilde y_0$ is an approximation of $D_c$.
\end{theorem}

\begin{proof}
From the estimate of $|\ln 1/p_0(x_k)|=|E(x_k;y)|\le\sum_\alpha |y_\alpha|\approx_n \|E\|_{H_a^1}$, 
Remark \ref{denseness} gives the convergence of 
\begin{equation} \label{app neuron proof}
		\tilde y_\alpha\to \frac1{2^n}\int_{\mathbb{T}^n} \bar\omega_\alpha(x)\ln  \frac1{p_0(x)} \,dx = y_\alpha 
\end{equation}
as $m\to\infty$. 
In addition, by Remark \ref{denseness}, by (\ref{app neuron proof}), and by the boundedness of $|\bar\omega_{\alpha}|=1$,    
\begin{equation*}
\begin{aligned}
	 \frac1{2^n}\sum_{\gamma\in\mathbb{Z}^n} \sum_{k=1}^m  
		\bar \omega_{\alpha-\gamma}(x_k) \Delta_{x_k}\,\tilde y_\gamma
		&\to 
	\sum_{\gamma\in\mathbb{Z}^n}	y_\gamma\frac1{2^n}
		\int_{\mathbb{T}^n} \bar\omega_{\alpha-\gamma}(x) \,dx \\
		&=\sum_{\gamma\in\mathbb{Z}^n}y_\gamma \delta_{\alpha}(\gamma)\\
		&=y_\alpha
\end{aligned}		
\end{equation*}
as $m\to\infty$,
where $\delta_{\alpha}(\gamma)$ is a Dirac delta.
Therefore, the proof is complete.
\end{proof}

Theorem \ref{full learning and bidirection for a frequency learning} says 
bidirectional communication between neurons during learning.
Red dashed arrows of Figures \ref{fig:M4} denote connections between $y_\alpha$ and $y_\gamma$,  the
connection  is defined by
\begin{equation} \label{lr for frequency}
	\Upsilon_\alpha(\gamma) \equiv \frac1{2^n}\sum_{k=1}^m  \bar \omega_{\alpha-\gamma}(x_k) \Delta_{x_k}.
\end{equation}

We define a learning rate for frequency by the inverse proportion to the maximum variance of $\Upsilon_\alpha$ (for example, see Figure \ref{frequency LR}), i.e.,
\begin{equation} \label{learning-rate1}
	\frac1{\max_\alpha \mathrm{Var}(\Upsilon_\alpha)}.
\end{equation}
By the reason of $\Upsilon_\alpha\to\delta_\alpha$ as $m\to\infty$, 
the learning rate grows to $\infty$ as $m\to\infty$.

\paragraph{Moment learning.}
As an approximation of $\partial_x^\alpha f(x)$, 
the $\alpha$th order central difference quotient $\tilde \partial_x^\alpha f(x)$ with a sequence of spacing vectors $(h_k)$
is defined by 
\begin{equation} \label{central difference}
\begin{aligned}
	&\tilde \partial_x^\alpha f(x) 
	\equiv \sum_{\|b\|_\infty=0,1}(-1)^b\frac1{2^\alpha h_k^\alpha} \\
	&\times\! 
		f\Big(x_1\!+\!\sum\limits_{k=1}^{\alpha_1}\!(-1)^{b_1}h_{k,1},
		\ldots,x_n\!+\!\sum\limits_{k=1}^{\alpha_n}\!(-1)^{b_n}h_{k,n}\Big),
\end{aligned}	
\end{equation}
where $h_k=(h_{k,1},\ldots,h_{k,n})$.
As $\max_k|h_k|\to0$, (\ref{central difference}) goes to $\partial_x^\alpha f(x)$ by L'Hospital's theorem if it exists.

Suppose that $(x_k)$ is ordered sample vectors of size $m$.
For notational simplicity, we may put $x_{1,j}$ is the smallest upper bound of $0$, i.e.,
$x_{0,j}<0\le x_{1,j}$  for every $j$,
in the ordering of $x_{k,j}<x_{k+1,j}$ $(k=-m,-m+1,\ldots,m-1)$.

\begin{theorem} \label{full learning and bidirection for a moment learning}
If $x_1,\ldots,x_m$ are ordered realization of a stationary ergodic vector process of $X_1$, $\cdots$, $X_m$, then for each $\alpha$ 
with $|\alpha|\le m$, 
\begin{equation*} 
\begin{aligned}
	\tilde y_\alpha 
	&= \frac{1}{\alpha!}\sum_{\|b\|_\infty=0,1}(-1)^b \frac1{2^\alpha h_k^\alpha} \\
	&\quad\times\!\ln 1/p_0\Big(x_{\sum\limits_{k=1}^{\alpha_1}\!(-1)^{b_1},1},
		\ldots,x_{\sum\limits_{k=1}^{\alpha_n}\!(-1)^{b_n},n}\Big)
		\to y_\alpha
\end{aligned}	
\end{equation*}
as $m\to\infty$,  where 
$h_{k,j}\equiv x_{k+1,j}-x_{k,j}$  is the $j$th component of $h_k$ and
$\tilde y_0$ is an approximation of $D_c$.
\end{theorem}

\begin{proof} 
Let $\alpha\ne0$. By Remark \ref{denseness}, $\max_k|h_k|\to0$ as $m\to\infty$. Since $E$ is sufficiently smooth, 
(\ref{central difference}) yields $\tilde y_\alpha\to y_\alpha$ as $m\to\infty$.
For $\alpha=0$, by Remark \ref{denseness}, the smallest upper bound $x_{1,j}$ of $0$ decrease to $0$ as $m\to\infty$. Hence, 
\begin{equation*}
	\tilde D_c = \ln \frac1{p_0(x_{1,1},\ldots,x_{1,n})} \to \ln \frac1{p_0(0,\ldots,0)}
\end{equation*}
as $m\to\infty$. 
This completes the proof.
\end{proof}

From Theorem \ref{full learning and bidirection for a moment learning},
every $\tilde y_\alpha$ is a linear combination of $(\ln 1/p_0(x_k))$.  This means every $\tilde y_\alpha$ is bidirectionally connected to each other.
Figures \ref{fig:M5} explains communication between neurons
marked with red dashed arrows.

Also, Theorem \ref{full learning and bidirection for a moment learning} says that
the accuracy of $\tilde y_\alpha$ depends on the size of $h_k$. 
Hence, 
the connection  is defined by 
\begin{equation*}
	\Upsilon(k)=h_k,
\end{equation*}
similarly, a learning rate for moment is defined by the inverse proportion to the maximum of differences between samples, i.e., 
\begin{equation} \label{learning-rate2}
	\frac1{\max_{k} |\Upsilon(k)|}.
\end{equation}

\begin{remark}
\begin{enumerate}[$(i)$]
\item The ergodic property in Theorems \ref{full learning and bidirection for a frequency learning} and  \ref{full learning and bidirection for a moment learning} can be replaced with the \textsc{i.i.d.} property.
\item
At the ultimate time, the learning process is complete and every neuron will be independent of each other.
\item
To derive a governing probability, 
its limiting distribution have to be differentiable. Especially, in Theorem \ref{full learning and bidirection for a moment learning}, $p_0$ must be analytic. 
 Otherwise,   
a suitable analytic approximation could be replaced with $p_0$. We present an example in section 9.
\end{enumerate}
\end{remark}

Because all samples come from $p_0$, the lack of samples may cause inaccuracy to calculate $y_\alpha$ of Theorems \ref{full learning and bidirection for a frequency learning} and \ref{full learning and bidirection for a moment learning}.
Although there are several sampling methods, e.g.,
Metropolis-Hastings sampling, Gibbs sampling, both depend on prior information and have accumulated error.  
In the following section, we introduce another sampling method which is very effective,
whose convergence is also proven mathematically.

\smallskip

\section{Dynamical systems} \label{kmd}

In this section we introduce the Koopman mode decomposition (KMD) for a dynamical system as a preprocessing method for samples.
We briefly review the results of Rowley et al. (\cite{rowley-mezic-bagheri-schlatter-henningson}).
The KMD has two remarkable properties (\cite[Rowley, Mezi\'c, Bagheri, Schlatter, Henningson]{rowley-mezic-bagheri-schlatter-henningson}. 
First, 
it could remove redundant features and noise. 
Second, we could extract an amount of samples from snapshots,
even if a sample comes from a nonlinear dynamical system.

We derive a dynamical system for time series of random vectors.
Let $X_t=(X_{t,1},\ldots,X_{t,n})$ 
be a random vector with continuous time $t$ from $p_0$ and
$F(x_{t,j})$  the  cumulative distribution function of $X_{t,j}$. Then    
\begin{equation*}
	\partial_{x_{t,j}}F(x_{t,j}) = p_{0}(x_{t,j}), \quad
 	\frac{d}{dt}F(x_{t,j}) = p_{0}(x_{t,j})\dot{x}_{t,j},
\end{equation*}
where $p_0(x_{t,j})$ is the marginal probability distribution of $X_{t,j}$.

\begin{theorem} \label{dynamical system for stochastic process}
Suppose that $F(x_{t,j})$ is continuously differentiable with respect to $t$ and $x_{t,j}$ for $1\le j\le n$. 
Then the dynamical system for continuous time series of $x_{t,j}$ is obtained by
\begin{equation} \label{cds}
	\dot{x}_{t,j} 
	= \psi_j(x_{t,j}),
\end{equation}
where $\psi_j(x_{t,j})=\frac{\frac{d}{dt} F(x_{t,j} )} {\partial_{x_{t,j}} F(x_{t,j})}$.
\end{theorem}

For a discrete time series of random vectors, which is our concern, 
let $X_r=(X_{r,1},\ldots,X_{r,n})$  $(r\in\mathbb{Z}^\ast)$ be a discretization of $(X_t)$.
From (\ref{cds}), 
the dynamical system for $x_{r,j}$ $(r\in\mathbb{Z}^\ast)$ is approximated
by
\begin{equation*}
\begin{aligned}
	x_{r+1,j} 
	&= x_{r,j} + \frac{\frac{d}{dt} F(x_{r,j})} {\partial_{x_{r,j}}F(x_{r,j})}h_{r,j} \\
	&\equiv \tilde\psi_j(x_{r,j})
\end{aligned}	
\end{equation*}
and put
\begin{equation} \label{ds}
	x_{r+1} = \tilde\psi (x_r),
\end{equation}
where $h_{r,j}$ is time step 
and $\tilde\psi(x_r)\!=\!(\tilde\psi_1(x_{r,1}),\ldots,\tilde\psi_n(x_{r,n}))^\top$ is a mapping on
$\mathpzc{M}\subset\mathbb{R}^n$ of an invariant compact manifold.

The space  $L^2(\mu)$ consists of  complex valued $L^2$-functions on $\mathpzc{M}$  with $d\mu\equiv p_0(x)\,dx$ and 
the $L^2$-norm.
For a subspace $\mathpzc{H}$ of $L^2(\mu)$, the Koopman operator $\mathpzc{K}:\mathpzc{H}\to\mathpzc{H}$ is defined by
\begin{equation} \label{koopman operator}
	\mathpzc{K}f=f\circ\tilde\psi.
\end{equation}
We call 
$\lambda\in\mathbb{C}$ an eigenvalue of $\mathpzc{K}$ associated with eigenfunction
$\phi\in\mathpzc{H}$   
if $\mathpzc{K}\phi = \lambda\phi$.

\begin{remark}
\begin{enumerate}[$(i)$]
\item In the dynamical system of (\ref{ds}), it is hard  to calculate $\tilde\psi$, because  the cumulative distribution function $F$ does not reveal itself.  
\item The Koopman operator is a linear, i.e., $\mathpzc{K}(af+bg)=a\mathpzc{K}f+b\mathpzc{K}g$ $(f,g\in\mathpzc{H},\,\,a,b\in\mathbb{C})$.
\end{enumerate}
\end{remark}

To avoid mathematical ambiguity, we need two fundamental  assumptions on $\mathpzc{K}$. First, $\mathpzc{K}$ is bounded on $\mathpzc{H}$ which is dense in $L^2(\mu)$.
Second, the spectrum of $\mathpzc{K}$ consists of only discrete spectrum.  
If $f$ is an infinite linear combination of  Koopman eigenfunctions,  then 
\begin{equation} \label{koopman expansion}
	\mathpzc{K}^r f(x) = \sum_{k=1}^\infty \lambda_k^r v_k \phi_k(x) 
\end{equation} 
for $r\in\mathbb{Z}^\ast$, where
$(\lambda_k, \phi_k)$ is a pair of the Koopman eigenvalue and corresponding eigenfunction,  
$v_k$ a coordinate sequence  of $f$ relative to Koopman eigenfunctions (\cite{rowley-mezic-bagheri-schlatter-henningson}).

Putting $\mathpzc{E}(x)\equiv [f_1(x),\ldots,f_{n'}(x)]^\top$, namely, an ensemble of observables in $\mathpzc{H}$, where
$n'$ is the number of $\mathpzc{E}$,
we have 
\begin{equation} \label{koopman expansion2}
	\mathpzc{K}^r\mathpzc{E}
		= \mathpzc{E}(\tilde\Psi^r)
		= \sum_{k=1}^\infty \lambda_k^r V_k \phi_k
\end{equation} 
for $r\in\mathbb{Z}^\ast$, where a column vector of $V_k\in\mathbb{C}^n$  is the $k$th coordinates of $\mathpzc{E}$ relative to Koopman eigenfunctions, which is called the $k$th Koopman mode with respect to $(\lambda_k,\phi_k)$. 
%

We say that a measure preserving mapping $\tilde\psi$ is ergodic with respect to $\mu$ 
when either $\mu(E)=0$ or $\mu(E)=1$ for any measurable set $E\subset \mathpzc{M}$ with $\tilde\psi^{-1}(E)=E$. Birkhoff ergodic theorem affirms that if $f\in L^2(\mu)\subset L^1(\mu)$, then
\begin{equation*}
	\lim_{N\to\infty} \frac1N \sum_{r=0}^{N-1} f\circ \tilde\psi^r(x) = \int_{\mathpzc{M}}f\,d\mu\quad\mbox{a.e. }\; x\in\mathpzc{M}.
\end{equation*}

Write the measurement of values of $\mathpzc{E}$ along a single trajectory of $\tilde\Psi$ starting at  an initial vector $x\in\mathpzc{M}$ as
\begin{equation} \label{snapshot}
\begin{aligned}
	Y
		&=\left(\begin{matrix*}[c]
		\mid & \mid & & \mid \\ 
		\mathpzc{E}(x) & \mathpzc{E}(\tilde\Psi(x)) & \cdots & \mathpzc{E}(\tilde\Psi^{m-1}(x)) \\
		\mid & \mid & & \mid \\ 
		\end{matrix*}\right), \\[0.5em]		
	Y'
		&=\left(\begin{matrix*}[c]
		\mid & \mid & & \mid \\ 
		\mathpzc{E}(\tilde\Psi(x)) & \mathpzc{E}(\tilde\Psi^2(x)) & \cdots & \mathpzc{E}(\tilde\Psi^{m}(x)) \\
		\mid & \mid & & \mid \\ 
		\end{matrix*}\right)	
\end{aligned}				
\end{equation}
which consist of rows of data snapshots. Put $\mathpzc{A}=Y'Y^{+}$,
where $Y^+$ is the Moore-Penrose pseudo-inverse of $Y$.
According to Remark \ref{denseness}, we have the following relation between the KMD and dynamic mode decomposition (DMD) (\cite[Rowley, Mezi\'c, Bagheri, Schlatter, Henningson]{rowley-mezic-bagheri-schlatter-henningson}).

\begin{theorem} \label{koopman theorem}
Suppose that (\ref{ds}) is ergodic
and  $\mathpzc{E}$ spans an $n_0$-dimensional invariant subspace. 
Let $(\lambda_k,\phi_k)$ be a pair of eigenvalue and corresponding eigenfunction of $\mathpzc{K}$.
If $\phi_k$ belongs to the span of $f_1,\ldots,f_{n'}$, i.e.,
\begin{equation*}
	\phi_k = w_1f_1+\cdots +w_{n'} f_{n'} = w\cdot \mathpzc{E}, 
\end{equation*}
then for almost every $x$,
$w\in\mathpzc{E}\big(\overline{\mathrm{orbit}(x)}\big)$, which is a left eigenvector of $\mathpzc{A}$ with eigenvalue $\lambda_k$.
\end{theorem}

On a finite orbit $\big(\tilde\Phi^k(x)\big)_{k=0}^{m-1}$, eigenvalues and eigenvectors of  $\mathpzc{A}$ approximate Koopman eigenvalues, eigenfunctions, and modes. Those can be derived by the dynamic mode decomposition (DMD) that was defined by Schmid and Sesterhenn in \cite[Schmidt, Sesterhenn]{schmidt-sesterhenn} and \cite[Schmidt]{schmidt} to extract the spatial flow structures that evolve linearly with time.

The DMD algorithm can be used to compute the augmented modes which approximate the Koopman modes. From \cite[Tu, Rowley, Luchtenburg, Brunton, Kutz]{tu-rowley-luchtenburg-brunton-kutz}, 
by taking the singular value decomposition put $Y=U\Sigma V^\ast$. Then, $\mathpzc{A}=Y'V\Sigma^{-1}U^\ast$.
For computational efficiency, take projection $\mathpzc{\tilde A}$ of $\mathpzc{A}$ onto the mode of  proper orthogonal decomposition (POD), 
by $\mathpzc{\tilde A}=U^\ast \mathpzc{A} U=U^\ast Y' V\Sigma^{-1}$ of an $n_0\!\times\! n_0$-matrix.

The matrix $\mathpzc{\tilde A}$ defines a low-dimensional linear model of the dynamical system on the POD coordinates $\tilde{y}_{r+1}=\mathpzc{\tilde A} \tilde{y}_r$ and the reconstruction to the high-dimensional state follows as $y = U \tilde{y}$. By Eigen-decomposition, $\mathpzc{\tilde A}W = W\Lambda$, where columns of $W$ are eigenvectors and $\Lambda$ is a diagonal matrix containing the corresponding eigenvalues $\lambda_k$.
Thus, the eigenvectors of $\mathpzc{A}$ (DMD modes) are given by columns of $\Phi$, 
\begin{equation*} 
	\Phi = Y'V\Sigma^{-1} W.
\end{equation*}

With the low-rank decomposition in hand, the projected future solution can be constructed for all time. With time step $h_t$, the solutions at all future times are approximated by
\begin{equation} \label{KMD approximation}
	y_r \approx \Phi \exp(\Omega r)B,
\end{equation}
where 
$\Omega=\diag(\ln(\lambda_k)/h_t)$ is a diagonal matrix, and
$B$ consists of the initial amplitudes of each mode, precisely, $B=\Phi^+ y$ at time $t=0$.

As a consequence of Theorem \ref{koopman theorem} and (\ref{KMD approximation}),  the inversion $\mathpzc{E}^{-1}$ from observables back to state-space yields $x_r$.

\begin{corollary} \label{data generation} 
If there exists the inverse of $\mathpzc{E}$ such that $\mathpzc{E}(x) = y$, then
\begin{equation*} 
	x_r \approx \mathpzc{E}^{-1}( \Phi \exp(\Omega r)B).
\end{equation*}
\end{corollary}

If $\mathpzc{E}$ is identified, e.g., $f_k(x)=x_k$, then Corollary \ref{data generation} generates data which follows (\ref{ds})
and 
we obtain sufficiently large size of data.
For more information for the KMD and DMD, 
refer to \cite[Korda, Mezi\'c]{korda-mezic}, \cite[Arbabi, Mezi\'c]{arbabi-mezic}. 

\begin{remark}
\begin{enumerate}[$(i)$]
\item The ergodic property in Theorem \ref{koopman theorem} can be replaced with the \textsc{i.i.d.} property.
\item The ergodicity allows us to reduce in computational cost in contrast to the \textsc{i.i.d.} property 
and to raise accuracy up for learning.
\end{enumerate}
\end{remark}

\smallskip

\section{$L^1$-stability of empirical distribution functions}

In this section, it will be discussed the convergence of empirical  distribution functions for  sample vectors.
Suppose that  a stochastic vector process of $(X_r)$ 
from a governing probability $p_0$ of $X$
with values in $\mathpzc{M}$.
Let $F$ be the (multivariate) cumulative distribution function of $X$. 
Since the set of discontinuities is an $F_\sigma$ set,
by Cavalieri's principle, the componentwise monotonicity of $F$ permits continuity except possibly measure zero set. 

The multivariate empirical distribution function of $(X_r)$ is given by
\begin{equation} \label{multi-empirical}
	\hat F_m(x)=\frac1m\sum_{r=1}^m\mathbbm{1}_{X_r\le x},
\end{equation}
where 
$X_r\le x$ means $X_{r,k}\le x_k$ componentwise.  
If $(X_r)$ is a sequence of \textsc{i.i.d.} random vectors, then $\hat F_m(x)$ converges to $F$ in the supremum norm  
by Vapnik-Chervonenkis theorem  \cite[p.823, Example 1 of p.833]{shorack-wellner} 
a generalization of Glivenko-Cantelli theorem.


For  a sequence of locally integrable functions, 
we say that the sequence is $L^1$-stable if the sequence 
is a Cauchy sequence in  
$L^1(\mathbb{R}^n)$. 
We prove the $L^1$-stability of multivariate empirical distribution functions of (\ref{multi-empirical}) 
on $\mathbb{R}^n$,  
even if those may not be integrable.

\begin{theorem} \label{norm convergence}
If $(X_r)$ is a stationary ergodic process from $p_0$ on $\mathbb{R}^n$ such that
$\mathbb{E}(|X_{1,j}|)<\infty$ $(1\le j\le  n)$, then 
$\hat F_m$ converges pointwise to $F$  almost everywhere and is $L^1$-stable.
\end{theorem}

\begin{proof} 
For each $x$ put  $f_x(\,\cdot\,)=\mathbbm{1}_{\,\cdot\,\le x}\in L^1(\mu)$, where $d\mu(x)=p_0(x)\,dx$.
By Birkhoff ergodic theorem, 
there exists a set of almost all initial samples such that 
\begin{equation} \label{ergodic identity}
	\hat F_m(x)=\frac1m \sum_{r=1}^m f_x(X_r)\to \mathbb{E}(f_x)=F(x)
\end{equation}
at almost every $x$ as $m\to\infty$.  
In fact, the componentwise monotonicity of $\hat F_m$ implies the componentwise monotonicity of the limit of $\hat F_m$.
By Cavalieri's principle, the limit is continuous except at most measure zero set.

As $\hat F_m-\hat F_{m'}$ is a linear combination of indicator functions, it vanishes if a vector variable is larger than every $x_k$. 
It is supported compactly and belongs to $L^1(\mathbb{R}^n)$,
accordingly the integral of $|\hat F_m-\hat F_{m'}|$ is well defined for $m$ and $m'$.
Let $R$ be a positive integer and split $\mathbb{R}^n$ into two parts of
\begin{equation*}
	\mathbb{R}^n = [-R,R]^n \cup \mathbb{R}^n\setminus[-R,R]^n.
\end{equation*}
Then
\begin{equation} \label{split}
\begin{aligned}
	\int_{\mathbb{R}^n}|\hat F_m-\hat F_{m'}|\,dx
		&= \int_{[-R,R]^n} |\hat F_m-\hat F_{m'}|\,dx \\
		&\qquad + \int_{\mathbb{R}^n\setminus[-R,R]^n} |\hat F_m-\hat F_{m'}|\,dx \\
		&\equiv I + II.
\end{aligned}		
\end{equation}

Estimate of $I$: 
By the compactness of $[-R,R]^n$,
the Lebesgue dominated convergence theorem and Birkhoff ergodic theorem imply 
\begin{equation} \label{estimate of I}
	I \to 0
\end{equation}	
as $m,m'\to\infty$.

Estimate of $II$: 
The region can be written the union of two part separations, precisely, 
\begin{equation*}
	\mathbb{R}^n\setminus[-R,R]^n
	= \bigcup_{k=0}^{n-1}\mathbb{R}^{k}\times\left(\mathbb{R}\setminus [-R,R] \right)\times \mathbb{R}^{n-1-k}.
\end{equation*}

Let $b=\{0,1\}^n$ be a binary multi-index and $b'$ the negation of $b$.
The notation of $x^b$ means 
the selection of corresponding components $x_{j}$ when $b_j=1$.
By Fubini's theorem,
\begin{align}
	II &\le \sum_{|b|=1}
			\int_{\mathbb{R}^{n-1}} 
			\int_{\mathbb{R}\setminus[-R,R]} 
			|\hat F_m-\hat F_{m'}|
			(dx)^b(dx)^{b'} \notag \\
	       &\equiv \sum_{|b|=1}
	       		\int_{\mathbb{R}^{n-1}} V_{m,m'}(x^{b'})(dx)^{b'}, \label{estimate of II}
\end{align}	
where  $(dx)^b= dx^b$ and
\begin{equation*}
	V_{m,m'}(x^{b'})
		= \int_{\mathbb{R}\setminus[-R,R]} 
			|\hat F_m-\hat F_{m'}|(dx)^b.
\end{equation*}

Estimate of $V_{m,m'}(x^{b'})$: By the change of variables,
\begin{equation*}
\begin{aligned}
	V_{m,m'}(x^{b'})
		&= \int_{(-\infty,-R]} |\hat F_m-\hat F_{m'}|(dx)^b   \\
		&\qquad\quad\qquad	+ \int_{[R,\infty)} |\hat F_m-\hat F_{m'}|(dx)^b  \\
		&= \int_{[R,\infty)} |\hat F_m(-1\odot^b x)
				-\hat F_{m'}(-1\odot^b x)|(dx)^b   \\
		&\qquad\quad\qquad + \int_{[R,\infty)} |\hat F_m-\hat F_{m'}|(dx)^b, 
\end{aligned}
\end{equation*}
where $-1\odot^b x$ is the componentwise multiplication only for non-zero components of $b$. 
By the triangle inequality, the last sum of integrals is bounded by
\begin{align}
		\int_{[R,\infty)} & \hat F_m(-1\odot^b x) + \hat F_{m'}(-1\odot^b x)(dx)^b  \notag \\
		&\qquad + \int_{[R,\infty)} |\hat F_m-1|+|1-\hat F_{m'}|(dx)^b \notag \\
		&= \int_{[R,\infty)} \hat v_m(x)(dx)^b  + \int_{[R,\infty)} \hat v_{m'}(x)(dx)^b \notag \\
		&\equiv A_b+B_b,	\label{estimate of V_m}		
\end{align}
where 
\begin{equation*}
\begin{aligned}
	\hat v_m(x) &= \hat F_m(-1\odot^b x) + 1 - \hat F_m(x), \\
	            \hat v_{m'}(x) &= \hat F_{m'}(-1\odot^b x) + 1 - \hat F_{m'}(x).
\end{aligned}
\end{equation*}

Estimate of $A_b$ and $B_b$: Since $\hat F_m(x)$ and $\hat F_{m'}(x)$ are sums of products of Heaviside functions, 
by Birkhoff ergodic theorem,
\begin{align}
	A_b 
	   &= \frac1m \sum_{r=1}^m  (|X_{r}^b|-R) \mathbbm{1}_{|X_{r}^b|\ge R} \notag \\
	   &\quad \longrightarrow
	   \mathbb{E}\left( (|X_{1}^b|-R) \mathbbm{1}_{|X_{1}^b|\ge R} \right)_{p_0}  \label{estimate of A}
\end{align}
as $m'\to\infty$.
Similarly,
\begin{equation}  \label{estimate of B}
	B_b  \longrightarrow
	   \mathbb{E}\left( (|X_{1}^b|-R) \mathbbm{1}_{|X_{1}^b|\ge R} \right)_{p_0} 
\end{equation}
as $m'\to\infty$.

Finally, according to (\ref{split}) $\sim$ (\ref{estimate of B}),
\begin{align}
	&\limsup_{m.m'\to\infty}  \int_{\mathbb{R}^n}   |\hat F_m-\hat F_{m'}|\,dx \notag \\
		&\le \limsup_{m,m'\to\infty} 
			\sum_{|b|=1} \int_{\mathbb{R}^{n-1}} V_{m,m'}(x^{b'})(dx)^{b'} \notag \\
		&\le  \sum_{|b|=1} \int_{\mathbb{R}^{n-1}} 
			\limsup_{m,m'\to\infty}
			V_{m,m'}(x^{b'})(dx)^{b'} \notag \\
		&= 2 \sum_{|b|=1} \int_{\mathbb{R}^{n-1}}
			\E\big(|X_{1}^b|-R)\mathbbm{1}_{|X_{1}^b| \ge R}\big)_{p_0}(dx)^{b'} \notag \\
		&= 2  \int_{\mathbb{R}^{n-1}}
			\E\left(\sum_{|b|=1}(|X_{1}^b|-R)\mathbbm{1}_{|X_{1}^b| \ge R}\right)_{p_0}(dx)^{b'}. \label{ergodic vanishing}
\end{align}
By the Lebesgue dominated convergence theorem, we can interchange the limit of $R\to\infty$ and two integrals of (\ref{ergodic vanishing}).
Then, the integrand of  (\ref{ergodic vanishing})  goes to 0. 
Therefore, $\hat{F}_m$ is $L^1$-stable for
almost every initial sample.
\end{proof}


\section{Estimations}
We describe how the designed and learned neurons respond to a signal.
For convenience, we define a lexicographical ordering to classify neurons. 
For $\alpha,\,\beta\in\mathbb{Z}^n$, define the ordering by
\begin{equation*}
	\alpha < \beta
\end{equation*} 
if either 
\begin{equation*}
	|\alpha| = \sum_i^n |\alpha_i| < \sum_i^n |\beta_i| = |\beta|
\end{equation*}
or 
\begin{equation*}
	\sum_i^n |\alpha_i| = \sum_i^n |\beta_i|\;\mbox{ and }\;
	\alpha_i<\beta_i
\end{equation*}
for the largest $i$ such that  $\alpha_i\ne \beta_i$.  
If $\alpha_i=\beta_i$ for all $i$, then we define $\alpha=\beta$.  
For example, let $n=2$. Putting $\bar m=-m$ for $m>0$, we have
%

\tikzstyle{materia}=[draw, fill=blue!20, text width=6.0em, text centered,
  minimum height=1.5em,drop shadow]
\tikzstyle{practica} = [materia, text width=9em, minimum width=10em,
  minimum height=3em, rounded corners, drop shadow]
\tikzstyle{decision} = [diamond, materia, text width=9em, minimum width=10em,
  minimum height=3em, rounded corners, drop shadow]
\tikzstyle{texto} = [above, text width=6.0em, text centered]
\tikzstyle{line} = [draw, thick, color=black!50, -latex']
\tikzstyle{ur}=[draw, text centered, minimum height=0.01em]
 
\newcommand{\blockdist}{1.3}
\newcommand{\edgedist}{1.5}

\newcommand{\practica}[2]{node (p#1) [practica]
  {\scriptsize{#2}}}
\newcommand{\decision}[2]{node (p#1) [decision]
  {\scriptsize{#2}}}

\newcommand{\background}[5]{%
  \begin{pgfonlayer}{background}
    \path (#1.west |- #2.north)+(-0.5,0.35) node (a1) {};
    \path (#3.east |- #4.south)+(+0.5,-0.35) node (a2) {};
    \path[fill=yellow!20,rounded corners, draw=black!50, dashed]
      (a1) rectangle (a2);
    \path (a1.east |- a1.south)+(0.1,-0.3) node (u1)[texto]
      {\scriptsize{#5}};
  \end{pgfonlayer}}

\newcommand{\transreceptor}[3]{%
  \path [linepart] (#1.east) -- node [above]
    {\scriptsize Transreceptor #2} (#3);}

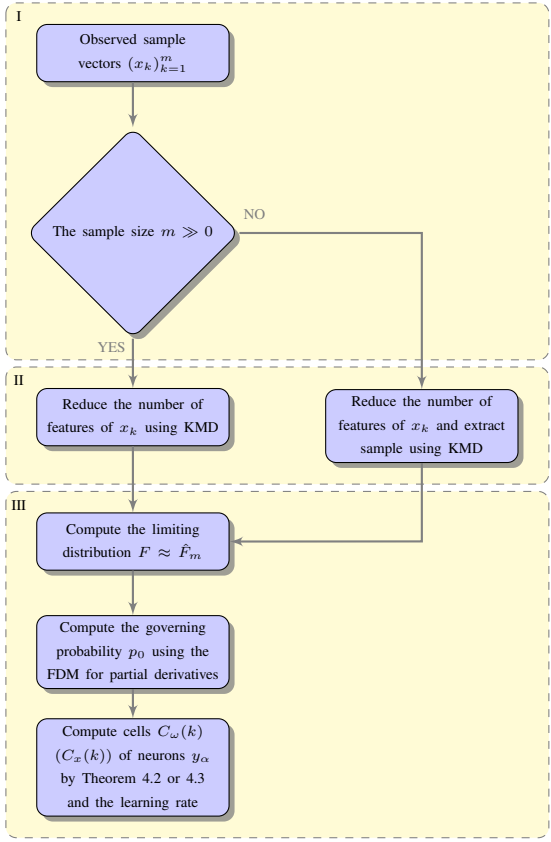
\begin{figure}[H]
\centering
\begin{tikzpicture}[scale=0.8,transform shape]
 
  \path \practica {1}{Observed sample vectors $(x_k)_{k=1}^m$};
  \path (p1.south)+(0.0,-2.5) \decision{2}{The sample size  $m\gg0$};
  \path (p2.south)+(0.0,-1.3) \practica{3}{Reduce the number of features of $x_k$ using KMD};
  \path (p2.east)+(3.0,-3.2) \practica{4}{Reduce the number of features of $x_k$ and extract sample using KMD};
  
  \path (p3.south)+(0.0,-1.58) \practica{5}{Compute the limiting distribution $F\approx \hat{F}_m$};
  \path (p5.south)+(0.0,-1.35) \practica{6}{Compute the governing probability $p_0$ using the FDM for partial derivatives};
  \path (p6.south)+(0.0,-1.3) \practica{7}{Compute cells $C_\omega(k)$ $(C_x(k))$ of neurons $y_\alpha$ by Theorem \ref{PNN by frequency} or  \ref{PNN by moment} and the learning rate};
     
  \path [line] (p1.south) -- node [above] {} (p2);

  \path [line] (p2.south) -- +(0.0,-0.475) 
    -- node [pos=-1.0, left] {\scriptsize YES} (p3);
  \path [line] (p2.east) -- +(3.0, 0.0) -- +(3.0,-2.4)
    -- node [above,xshift=-2.75cm,yshift=2.6cm] {\scriptsize NO} (p4);

  \path [line] (p3.south) -- +(0.0,-0.30) -- node [above] {} (p5) ;
  \path [line] (p4.south) -- +(0.0,-1.3) -- node [above] {} (p5) ;
  \path [line] (p5.south) -- +(0.0,-0.1) -- node [above] {} (p6) ;
  \path [line] (p6.south) -- +(0.0,-0.1) -- node [above] {} (p7) ;

  \background{p3}{p1}{p4}{p2}{I}
  \background{p3}{p3}{p4}{p4}{II}
  \background{p3}{p5}{p4}{p7}{III}

\end{tikzpicture}  
\caption{Block diagram for learning: Part I is the area to generate more samples. Part II controls the KMD and the learning process is achieved in Part III.}
\label{block diagram for learning}
\vspace{-0.0cm}
\end{figure}

\vspace{-0.7cm}

\hfsetfillcolor{green!10}
\hfsetbordercolor{green!20!black!30}
\begin{equation*}
	\underbrace{\scriptstyle 00}_{\text{Index of DC}} \!\! <
	\underbrace{\scriptstyle 0\bar1<\bar10<10 <01}_{\text{Indexes of $1$-cell}} <\;
	\underbrace{\scriptstyle 0\bar2<\bar1\bar1<1\bar1<\bar11<\bar20<20<11<02}_{\text{Indexes of $2$-cell}} \;< \cdots.	 
\end{equation*}

Neurons in Theorems \ref{PNN by frequency} and \ref{PNN by moment} are grouped into
the $k$th ordered cells $C_{\omega}(k)$ and $C_x(k)$, respectively. These are vector-valued functions from $\mathbb{Z}^\ast$ defined by
\begin{equation*}
	C_\omega(k)=
	\begin{pmatrix}
	\mid  \\
	y_\alpha  \\
	\mid  \\
	\end{pmatrix}_{|\alpha|=k}\kern -1em \in\mathbb{C}^{{}_{n}\mathsf{\kern -1pt H}_k}, \quad
	C_x(k)=
	\begin{pmatrix}
	\mid  \\
	y_\alpha  \\
	\mid  \\
	\end{pmatrix}_{|\alpha|=k} \kern -1em\in \mathbb{R}^{{}_{n}\mathsf{\kern -1pt H}_k}
\end{equation*}
in the lexicographical ordering, where ${}_{n}\mathsf{\kern -1pt H}_k$ is a combination of $n$ things taken $k$ at a time with repetition.
The $k$th  evaluation vectors $V_\omega(k)$ and $V_x(k)$  of a signal $x$ are ordered vector-valued functions from $\mathbb{Z}^\ast$ defined by
\begin{equation*}
	V_\omega(k)=
	\begin{pmatrix}
	\mid  \\
	\omega_\alpha(x) \\
	\mid  \\
	\end{pmatrix}_{|\alpha|=k} \kern -1.5em\in\mathbb{C}^{{}_{n}\mathsf{\kern -1pt H}_k},\quad
	V_x(k)=
	\begin{pmatrix}
	\mid  \\
	x^\alpha \\
	\mid  \\
	\end{pmatrix}_{|\alpha|=k} \kern -1.5em \in\mathbb{R}^{{}_{n}\mathsf{\kern -1pt H}_k}.
\end{equation*}
For convenience, set $C_\omega(0)=C_x(0)=(D_c)$ and $V_\omega(0)=V_x(0)=(1)$ of a singleton. 

\smallskip

\noindent\textsc{Active path.}
We will select the best matching neurons for a signal, namely, the active path. Those are chosen 
so that the likelihood of $x$ is maximized in the following way.
From $|e^{-E}|=e^{-\Re{E}}$, 
collect indexes $\alpha$ such that  $\Re (C_w(k)\odot V_w(k))_\alpha < 0$ $\big((C_x(k)\odot V_x(k))_\alpha < 0, \mbox{resp.}\big)$ 
and put it by
\begin{equation*}
\begin{aligned}
	\Gamma_k &= \begin{pmatrix}
	\mid \\
	\alpha \\
	\mid\\
	\end{pmatrix},
\end{aligned}	
\end{equation*}
where $(A\odot B)_\alpha$ is the $\alpha$th component of the Hadamard product of $A$ and $B$.
The active path for $x$ is defined by
\begin{equation*}
	\Gamma=\bigcup_{k=0}^\infty\big\{\Gamma_k\big\}.
\end{equation*}

We say that any subset $\mathpzc{D}$  of indexes is a dictionary if an index size is at most $N$ so that 
$\|\hat F_m - \hat F_{m'}\|_{L^1}<\epsilon$ for $m,m'\ge N$. 
Usually, an active path is considered in the dictionary.  
 So, we call
$\Gamma_{\mathpzc{D}}\equiv\Gamma\cap \mathpzc{D}$  the active path for $x$ with respect to $\mathpzc{D}$.

\newcommand{\backgroundlong}[5]{%
\begin{pgfonlayer}{background}
    \path (#1.west |- #2.north)+(-0.5,0.35) node (a1) {};
    \path (#3.east |- #4.south)+(+0.5,-1.8) node (a2) {};
    \path[fill=yellow!20,rounded corners, draw=black!50, dashed]
      (a1) rectangle (a2);
    \path (a1.east |- a1.south)+(0.1,-0.3) node (u1)[texto]
      {\scriptsize{#5}};
\end{pgfonlayer}}

\begin{figure}[H]
\centering
\begin{tikzpicture}[scale=0.8,transform shape]
 
  \path \practica {1}{A new signal $x$ and a dictionary $\mathpzc{D}$};
  \path (p1.south)+(0.0,-1.2) \practica{2}{Compute the evaluation vector $V_\omega(k)$ $(V_{x}(k))$};
  \path (p2.east)+(2.92,-7.6) \practica{2-1}{Compute the likelihood $p_0(x)$ of $x$ in the network};
  \path (p2.south)+(0.0,-2.6) \decision{3}{$\displaystyle \Re\!\big(C_\omega(k)\!\odot\! V_\omega(k)\big) <0$};
  \path (p2.east)+(2.92,-3.1) \practica{4}{Discard $\alpha$};
  \path (p3.south)+(0.0,-1.45) \practica{5}{Compute the active path $\Gamma_k\ni\alpha$};
  \path (p5.south)+(0.0,-1.0) \practica{6}{Active path on the dictionary of $\Gamma\cap \mathpzc{D}$};
     
  \path [line] (p1.south) -- node [above] {} (p2);

  \path [line] (p3.south) -- +(0.0,-0.475) 
    -- node [xshift=-0.4cm,yshift=0.6cm] {\scriptsize YES} (p5);
    
  \path [line] (p3.east) 
    -- node [above,xshift=-0.2cm,yshift=0.0cm] {\scriptsize NO} (p4);

  \path [line] (p2.south) -- +(0.0,-0.30) -- node [above] {} (p3) ;
  \path [line,] (p2.east) -- +(4.7,0.0) --+(4.7,-6.5) --+(2.92,-6.5)  --node [above] {} (p2-1) ;
  \path [line] (p5.south) -- +(0.0,-0.1) -- node [above] {} (p6) ;

  \background{p3}{p1}{p4}{p2}{I}
  \backgroundlong{p3}{p3}{p4}{p4}{II}
  \background{p3}{p5}{p2-1}{p6}{III}
    
\end{tikzpicture}  
\caption{Block diagram for estimation: Part I consists of a signal and a dictionary of indexes, in which the evaluation vector and likelihood are computed.
Part II chooses the indexes at which neurons have a negative projection to maximize a network probability. In Part III, an active path for the signal is estimated.}
\label{block diagram for estimation}
\vspace{-0.0cm}
\end{figure}
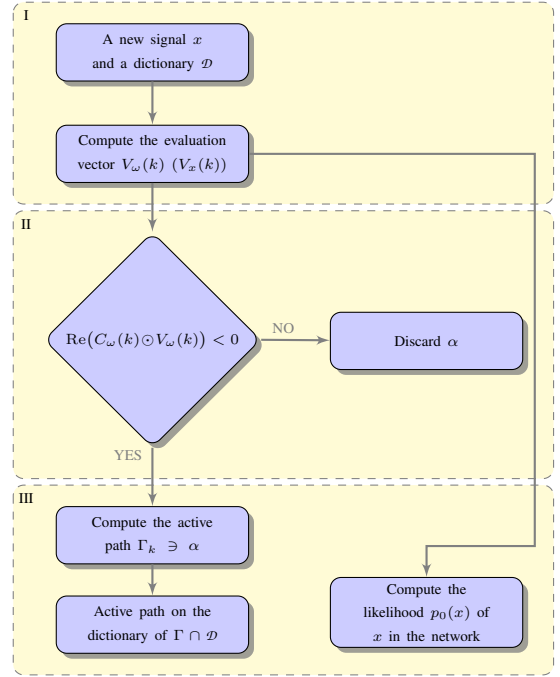

\smallskip
\noindent\textsc{Probability of active neuron.}
For a frequency learning, we define the probability of active neurons (POAN) on $\Gamma_{\mathpzc{D}}$ by
\begin{equation*}
	\bigcup_{k=0}^N
	\left\{ 
	\begin{pmatrix}
	\mid  \\
	-\Re\frac{(C_w(k)\odot V_w(k))_\alpha}{|(C_w(k)\odot V_w(k))_\alpha|} \\
	\mid  \\
	\end{pmatrix}\Big\lvert\, \alpha\in \Gamma_k\cap\mathpzc{D}
	\right\}.
\end{equation*}
For a moment learning, in extension to definition, each probability of active neurons must be $1$.

\smallskip
\noindent\textsc{Likelihood.}
The likelihood of $x$ is a non-negative real number induced by 
$p(x\mid y)$. The value means the possibility that  $x$  belongs to the network. 
The value is relatively large if and only if $x$ may be likely to happen in the network. 
See Examples \ref{Frequency learning ex}, \ref{Moment learning ex}, and Examples of section \ref{application to a standing wave}.

With energy functions assumed in advance, in the following examples we will explain how to calculate estimations of the network, i.e., an active path, probability of active neuron, and likelihood.
Figures \ref{block diagram for learning} and \ref{block diagram for estimation} show the learning and estimation procedures, respectively.

\begin{example}[Frequency learning] \label{Frequency learning ex}
Suppose that $E$ is given by
\begin{equation*}
	E(x_1,x_2;y ) 
	= -3\sin \pi x_1 +\cos \pi x_2  + 2\cos 2\pi x_2 
\end{equation*}
and  $(\nicefrac{1}{12},\nicefrac{1}{6})$ is a signal to be estimated by the network.

The energy function has the indexed form of 
\begin{equation*}
\begin{aligned}
	&-3\sin \pi x_1 +\cos \pi x_2  + 2\cos 2\pi x_2 \\
	&= \frac{3i}{2}(\omega_{10}-\omega_{\bar10}) 
		+ \frac{1}{2}(\omega_{01}+\omega_{0\bar1})  
		+ \omega_{02}+\omega_{0\bar2},
\end{aligned}	
\end{equation*}
its value of the partition function
\begin{equation*}
	Z = \int_{\mathbb{T}^2} e^{-E}dx = 47.9883,
\end{equation*}
and the DC value $D_c=\ln Z=3.8710$.
Thus, the network probability is given by
\begin{equation*}
	p(x_1,x_2\mid y) = e^{-3.8710 + 3 \sin \pi x_1 - \cos \pi x_2 -2 \cos 2\pi x_2}
\end{equation*}
(see Figure \ref{governing probability1}). 
Non-trivial cells are
\begin{equation*}
	C_\omega(0)=({\scriptstyle 3.8710}), \;
	C_\omega(1)=
	\begin{pmatrix*}[r]
	\nicefrac{1}{2}  \\ 
	\nicefrac{-3i}{2}  \\ 
	\nicefrac{3i}{2}  \\ 
	\nicefrac{1}{2}  
	\end{pmatrix*}, \;
	C_\omega(2)=
	\begin{pmatrix}
	\scriptstyle 1  \\
	\scriptstyle 0 \\
	\scriptstyle \vdots  \\
	\scriptstyle 0 \\
	\scriptstyle 1 
	\end{pmatrix}
\end{equation*}
and evaluation vectors that we need for $(\nicefrac{1}{12},\nicefrac{1}{6})$  are
\begin{equation*}
	V_\omega(0)=\big( {\scriptstyle 1 }\big),\;
	V_\omega(1)=
	\begin{pmatrix*}[r]
	\nicefrac{(1-\sqrt3i)}{2} \\ 
	\nicefrac{(\sqrt3-i)}{2} \\ 
	\nicefrac{(\sqrt3+i)}{2} \\ 
	\nicefrac{(1+\sqrt3i)}{2}  
	\end{pmatrix*}, \;
	V_\omega(2)=
	\begin{pmatrix}
	\nicefrac{(-1-\sqrt3i)}{2}  \\
	\vdots  \\[5pt]
	\nicefrac{(-1+\sqrt3i)}{2}   
	\end{pmatrix}.
\end{equation*}

Hence, for $(\nicefrac{1}{12},\nicefrac{1}{6})$, the active path,
POAN of $\Gamma$,
and  likelihood are as follows.
\begin{equation*}
	\Gamma =
	\left\{
	\begin{pmatrix}
	\scriptstyle \bar10 \\
	\scriptstyle 10 
	\end{pmatrix},\;
	\begin{pmatrix}
	\scriptstyle \bar20 \\
	\scriptstyle 20 
	\end{pmatrix}	
	\right\},\quad
	\mathrm{POAN}(\Gamma) =
	\left\{ 
	\begin{pmatrix}
	\scriptstyle \nicefrac{1}{2} \\
	\scriptstyle \nicefrac{1}{2} 
	\end{pmatrix},\;
	\begin{pmatrix}
	\scriptstyle \nicefrac{1}{2} \\
	\scriptstyle \nicefrac{1}{2} 
	\end{pmatrix}	
	\right\},
\end{equation*}
and    
\begin{equation*}
	p(\nicefrac{1}{12},\nicefrac{1}{6} \mid y) = 0.007,
\end{equation*}
respectively.

Since the likelihood is relatively small (Figure \ref{governing probability1}),  we may guess that $(\nicefrac{1}{12},\nicefrac{1}{6})$ may happen with small possibility in the network.
\end{example}

\begin{example}[Moment learning] \label{Moment learning ex}
Suppose that $E$ is defined by
\begin{equation*}
	E(x_1,x_2;y) 
	= -0.5x_1 - 2x_2 + 4x_1x_2 + 3x_2^2
\end{equation*}
and  $(-\nicefrac{4}{5},\nicefrac{4}{5})$ is a signal to be estimated by the network.

The indexed form of  $E$ is  
\begin{equation*}
\begin{aligned}
	& -0.5x_1 - 2x_2 + 4x_1x_2 + 3x_2^2 \\
	&= -0.5x^{10} - 2x^{01} + 4x^{11} + 3x^{02},
\end{aligned}	
\end{equation*}
its value of the partition function
\begin{equation*}
	Z = \int_{\mathbb{T}^2} e^{-E}dx = 4.2883
\end{equation*}
and the DC value $\ln Z=1.4559=D_c$.
So, the network probability is given by
\begin{equation*}
	p(x_1,x_2\mid y)
		= e^{-1.4559+0.5x_1 + 2x_2 - 4x_1x_2 - 3x_2^2}
\end{equation*}
(refer to Figure \ref{governing probability2}).
All non-trivial cells are
\begin{equation*}
	C_x(0)=({\scriptstyle 1.4559}), \;
	C_x(1)=
	\begin{pmatrix*}[r]
	\scriptstyle 0.5 \\
	\scriptstyle 2  \\
	\end{pmatrix*}, \;
	C_x(2)=
	\begin{pmatrix*}[r]
	\scriptstyle 0  \\
	\scriptstyle -4  \\
	\scriptstyle 0  \\
	\scriptstyle -3  
	\end{pmatrix*}
\end{equation*}
 and 
some evaluation vectors that we need for $(-\nicefrac{4}{5},\nicefrac{4}{5})$ are
\begin{equation*}
	V_x(0)=\big({\scriptstyle 1 }\big),\;
	V_x(1)=
	\begin{pmatrix*}[r]
	\scriptstyle -\nicefrac{4}{5} \\
	\scriptstyle \nicefrac{4}{5}  
	\end{pmatrix*}, \;
	V_x(2)=
	\begin{pmatrix}[r]
	\scriptstyle \nicefrac{16}{25}  \\
	\scriptstyle -\nicefrac{16}{25}  \\
	\scriptstyle -\nicefrac{16}{25} \\
	\scriptstyle \nicefrac{16}{25}   
	\end{pmatrix}.
\end{equation*}

Hence, for  $(-\nicefrac{4}{5},\nicefrac{4}{5})$, 
the active path, likelihood  are  
\begin{equation*}
	\Gamma =
	\left\{
	\begin{pmatrix}
	\scriptstyle 10 
	\end{pmatrix},\;
	\begin{pmatrix}
	\scriptstyle 02 
	\end{pmatrix}	
	\right\},
\end{equation*}
\begin{equation*}
	p(-\nicefrac{4}{5},\nicefrac{4}{5} \mid y) = 1.4683,
\end{equation*}
respectively.
The likelihood is relatively large (Figure \ref{governing probability1}), and so, we may guess that $(-\nicefrac{4}{5},\nicefrac{4}{5})$ may happen with large possibility in the network..
\end{example}

\begin{figure}[H]
\vspace{-0.1cm}
\centering
    \includegraphics[width=0.82\columnwidth,trim={1.5cm 1cm 1.25cm 0cm},clip]{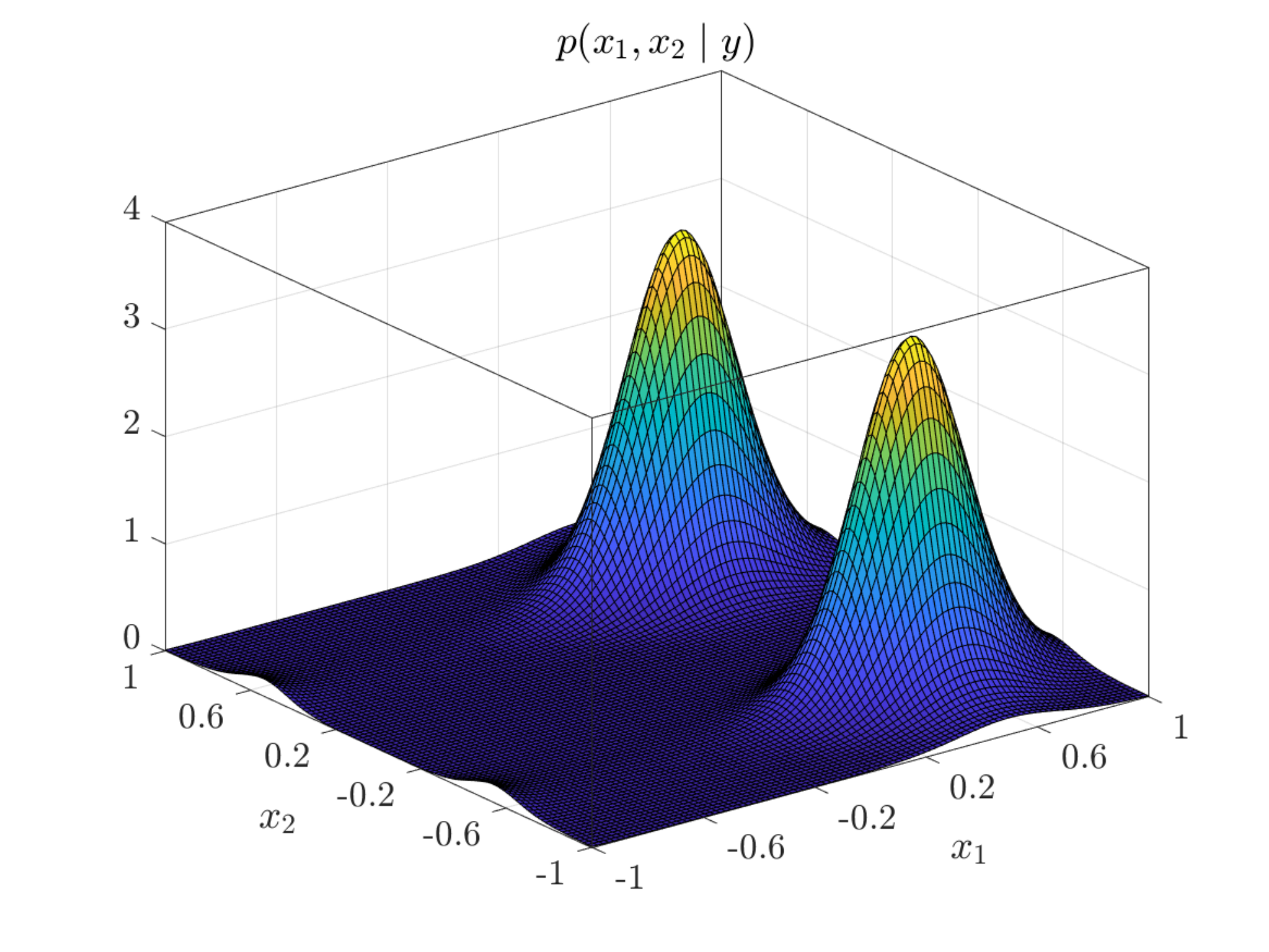}
\caption{The network  probability of Example \ref{Frequency learning ex}.}
\label{governing probability1}
\end{figure}

\begin{figure}[H]
\vspace{-0.5cm}
\centering
    \includegraphics[width=0.82\columnwidth,trim={1.5cm 1cm 1.25cm 0cm},clip]{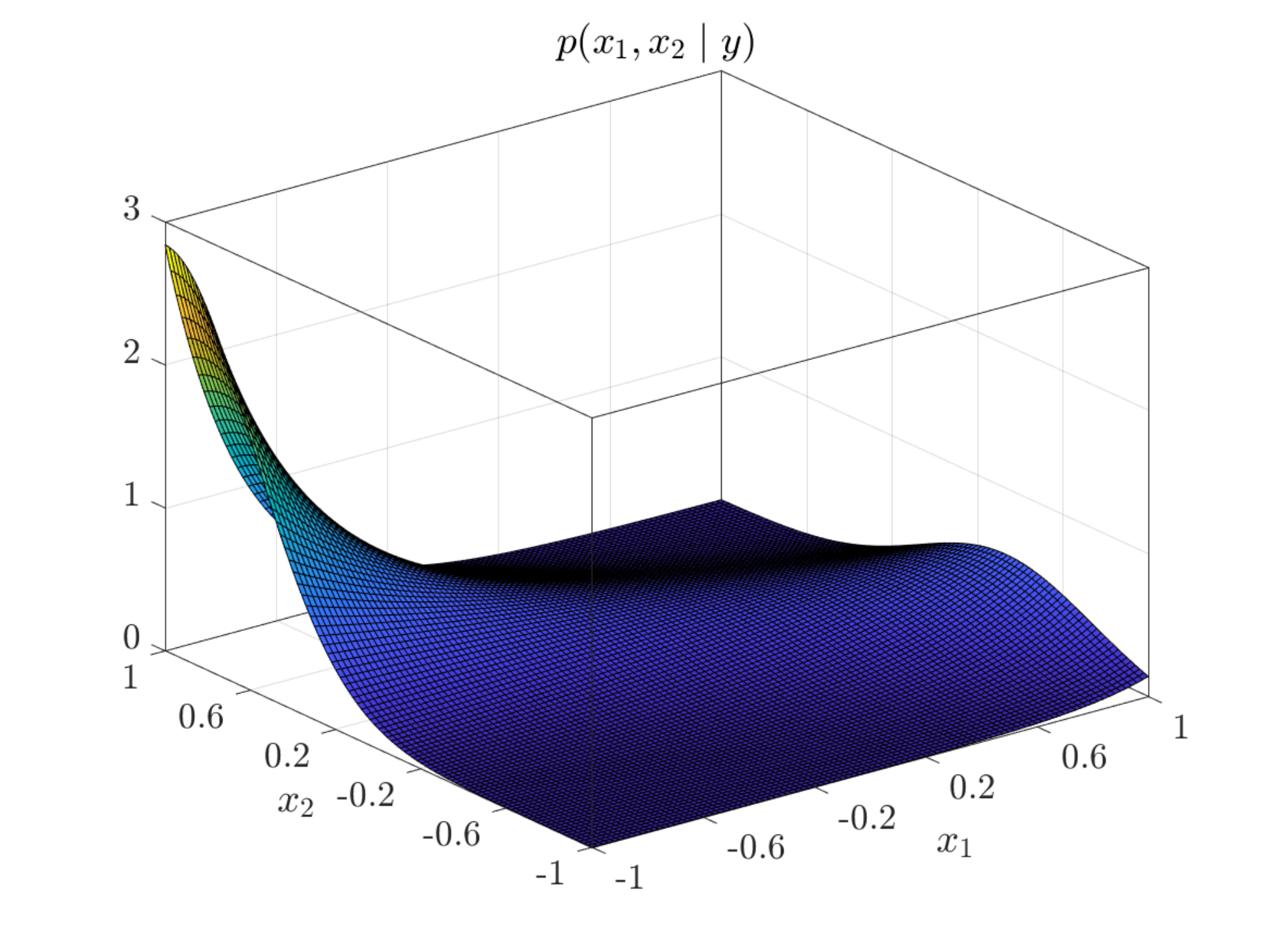}
\caption{The network probability of Example \ref{Moment learning ex}.}
\label{governing probability2}
\end{figure}

\smallskip
\noindent\textsc{Interpretation of estimations.}
To simplify what the active path says we 
give a topological structure on $\Gamma_{\mathpzc{D}}$.
Setting $d(y_\alpha,y_\beta)=|\alpha-\beta|$ on $\Gamma_{\mathpzc{D}}$, we define the number of basic open balls $N:\mathbb{Z}^\ast \to\bar{\mathbb{N}}$ by 
\begin{equation} \label{neighborhood}
	N(k)=\frac{1}{2}\big|k\mbox{-topology}\big|,
\end{equation}
where 
\begin{equation*}
	k\mbox{-topology}=\big\{(\alpha,\beta)\mid d(y_\alpha,y_\beta)=k\;\mbox{on}\;\Gamma\big\}
\end{equation*}
and $\bar{\mathbb{N}}$ is the extended natural numbers.

The definition (\ref{neighborhood})  could be changed 
according to learning models. 
For more methods, refer to \cite[Geirhos, Janssen, Schutt, Rauber, Bethge, Wichmann]{geirhos-janssen-schutt-rauber-bethge-wichmann}.  

\smallskip

\section{Application to a standing wave} \label{application to a standing wave}

A standing wave (or a stationary wave)  appears on the surface of a liquid in a vibrating container or on  vibrating strings,
is oscillation in time but whose peak amplitude outline does not move in space.
The locations at which the amplitude is minimum are defined as  nodes, and the locations where the amplitude is maximum are defined as antinodes.

For  the realization $x$ of the amplitude random variable $X$,
 the differential equation of the normalized wave is represented by
	$\ddot{x} + x = 0$,
where $-1\le x\le 1$.
Putting $x_1 = x$ and $x_2 = \dot{x}$, we have the dynamical system of
\begin{equation} \label{governing equation}
	\frac{d}{dt} 
	\left(\begin{matrix}
	x_1 \\
	x_2
	\end{matrix}
	\right)
	=
	\left(
	\begin{matrix}
	0 & 1\\
	-1 & 0
	\end{matrix}
	\right)
	\left(
	\begin{matrix}
	x_1 \\
	x_2
	\end{matrix}
	\right)
\end{equation}
which contains the standing wave.

\paragraph{Frequency learning.}
Suppose that there are $32$ samples with time step $h_t = 0.2$ for $0\le t \le 2 \pi$  
in Figure \ref{amplitude sample of standing wave}.

\begin{figure}[H]
\vspace{-0.2cm}
\centering
\includegraphics[width=0.82\columnwidth,trim={0.5cm 0cm 0.5cm 0cm},clip]{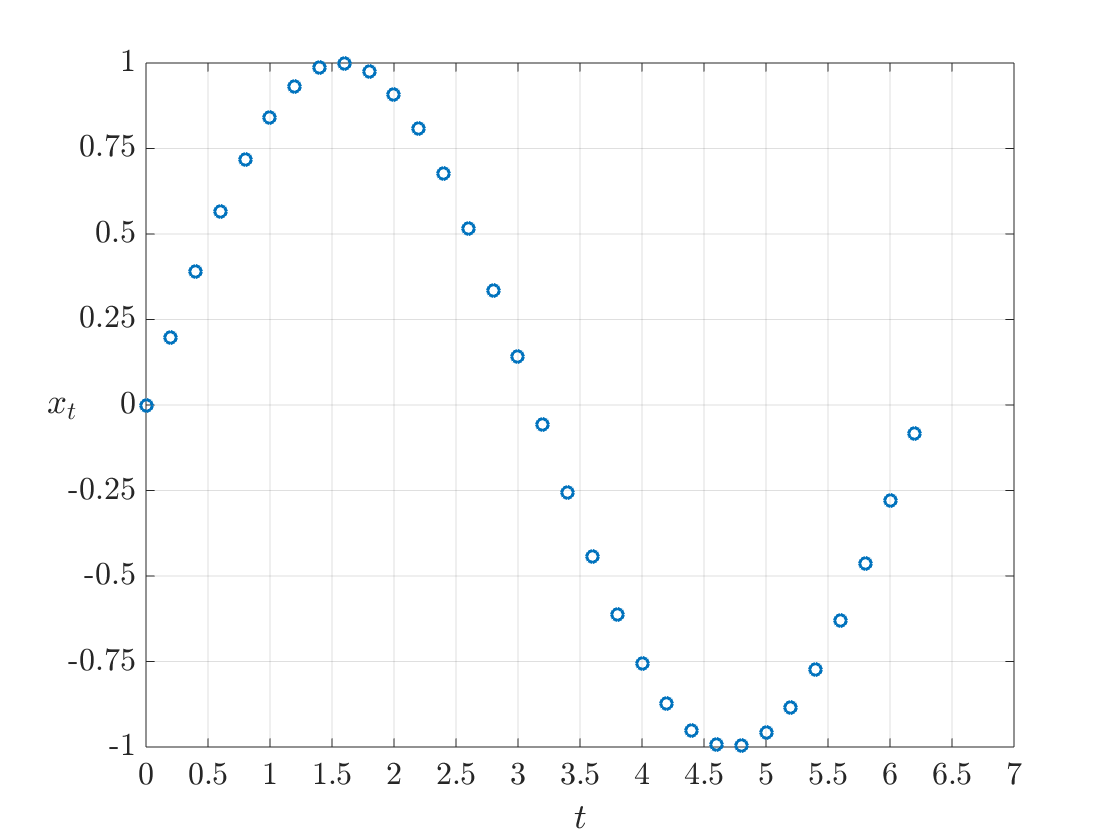}
\caption{Amplitude samples of size $32$.}
\vspace{-0.2cm}
\label{amplitude sample of standing wave}
\end{figure}

\begin{figure}[H]
\vspace{-0.5cm}
\centering
\includegraphics[width=0.82\columnwidth,trim={0.5cm 0cm 0.5cm 0cm},clip]{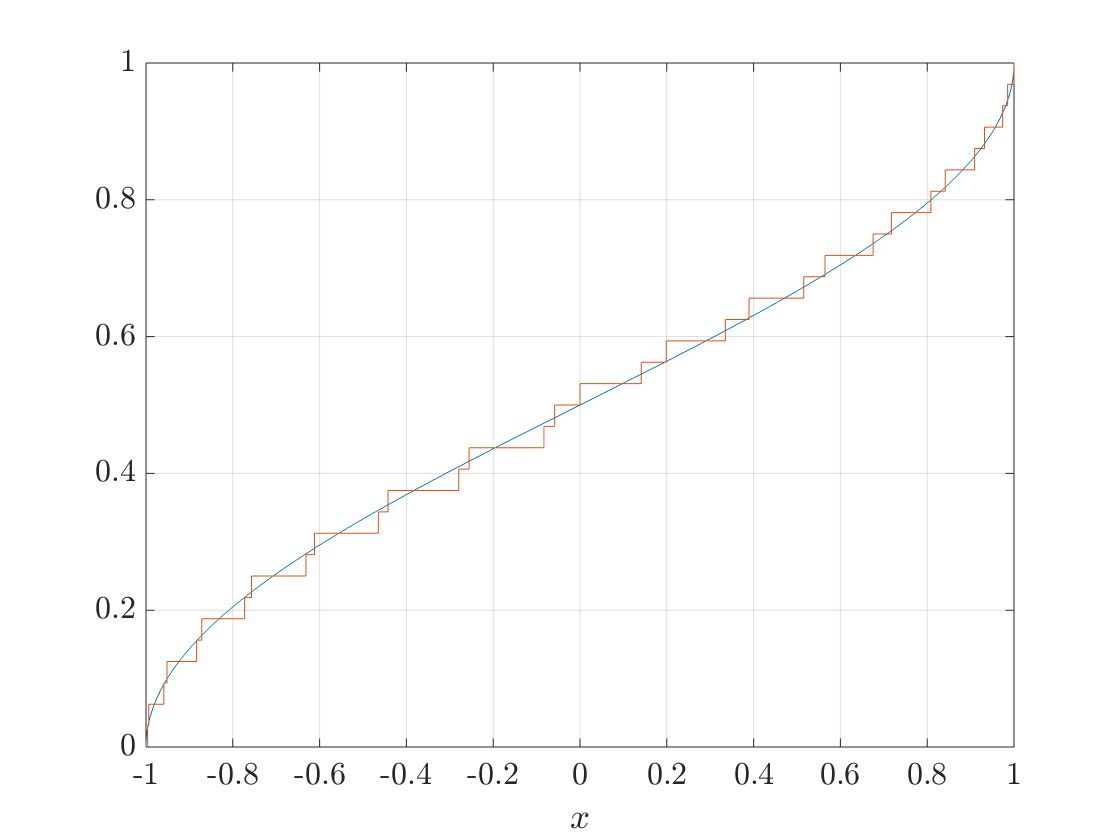}
\caption{The empirical distribution function $\hat{F}_m$ (red) and the limiting distribution $F$ (blue).}
\vspace{-0.2cm}
\label{amplitude empirical distributions}
\end{figure}

From snapshots of Figure \ref{amplitude sample of standing wave}, the KMD enables us to generate amplitude samples of size 6,284 ($h_t = 0.001$) as Figure \ref{generated sample2}. 
We regard $\hat F_{6,284}$ as the limiting distribution in Figure \ref{limiting distribution} by cubic spline interpolation.
With the assumption of differentiability of $p$ with respect to $x$, by taking partial derivatives, we approximate the governing probability $p_0$  as Figure \ref{f network probability}.

\begin{figure}[H]
\centering
\vspace{-0.0cm}
\includegraphics[width=0.82\columnwidth,trim={0.5cm 0cm 0.5cm 0cm},clip]{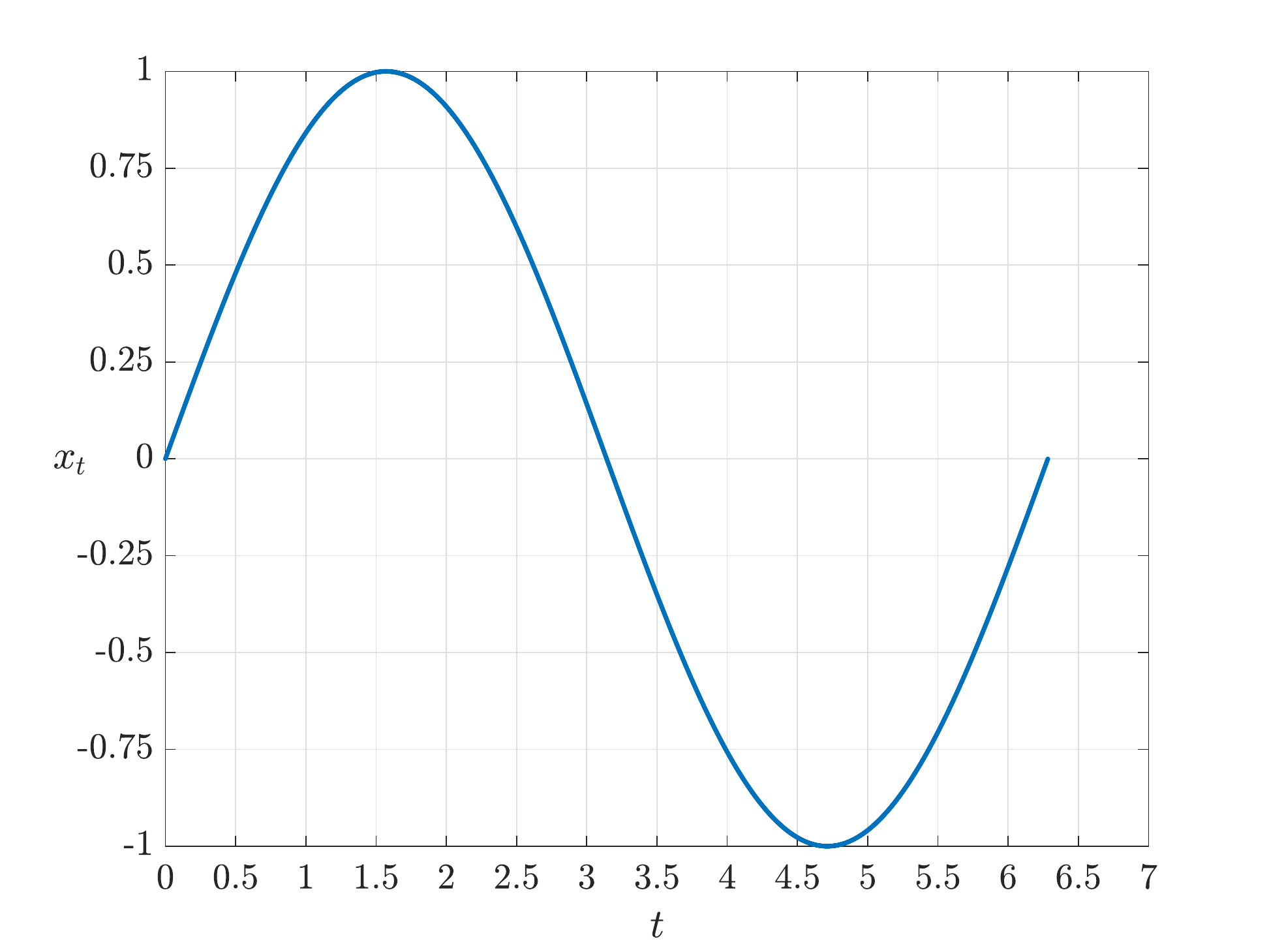}
\caption{Generated samples of size 6,284 with the KMD.}
\label{generated sample2}
\vspace{-0.2cm}
\end{figure}

\begin{figure}[H]
\centering
\vspace{-0.2cm}
\includegraphics[width=0.82\columnwidth,trim={0.5cm 0cm 0.5cm 0cm},clip]{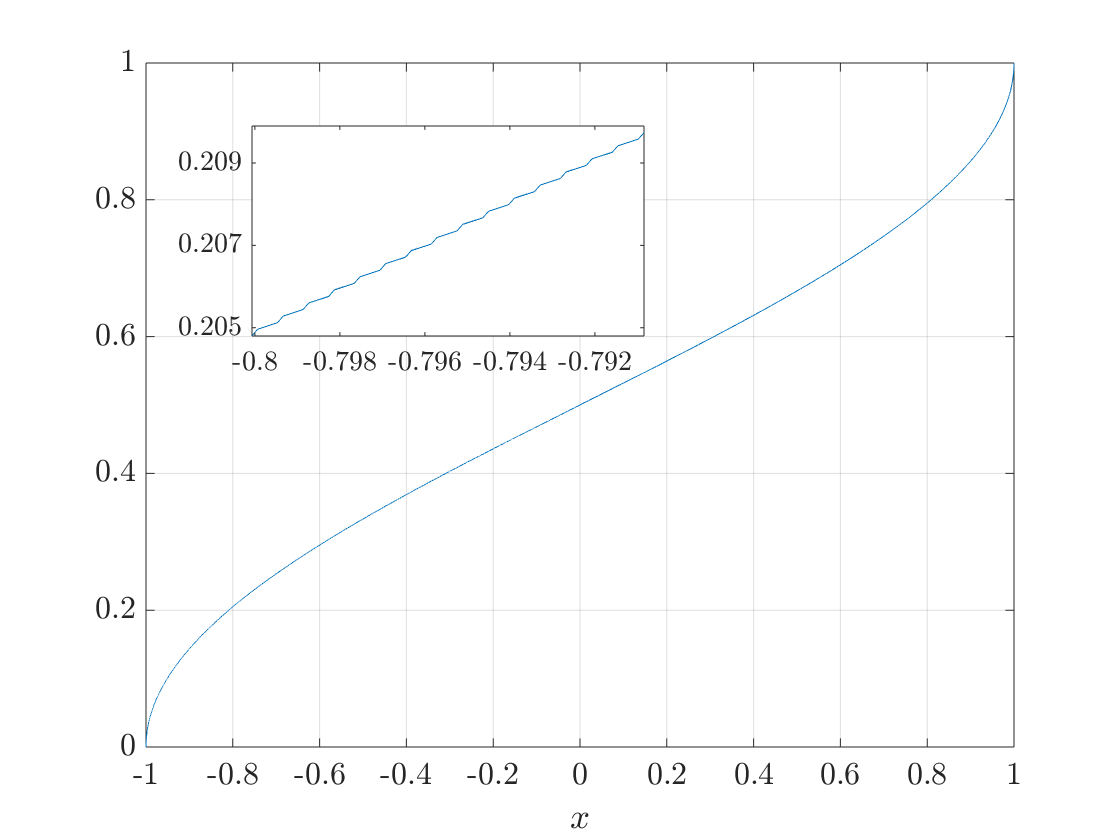}
\caption{The limiting distribution $F$.} 
\label{limiting distribution}
\vspace{-0.2cm}
\end{figure}

\begin{figure}[H]
\centering
\vspace{-0.2cm}
\includegraphics[width=0.82\columnwidth,trim={0.5cm 0cm 0.5cm 0cm},clip]{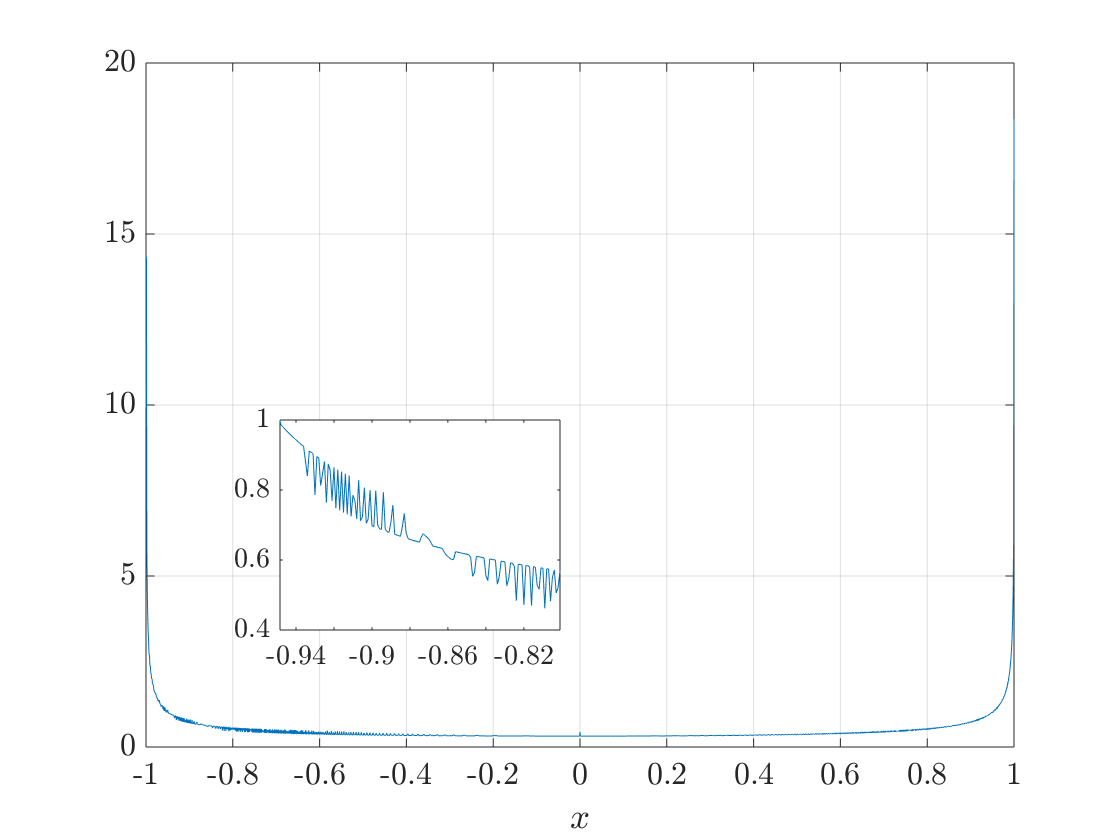}
\caption{The governing probability $p_0$.}
\label{f network probability}
\vspace{-0.2cm}
\end{figure}

\begin{figure}[H]
\centering
\vspace{-0.2cm}
\includegraphics[width=0.82\columnwidth,trim={1.25cm 0 1.25cm 1cm},clip]{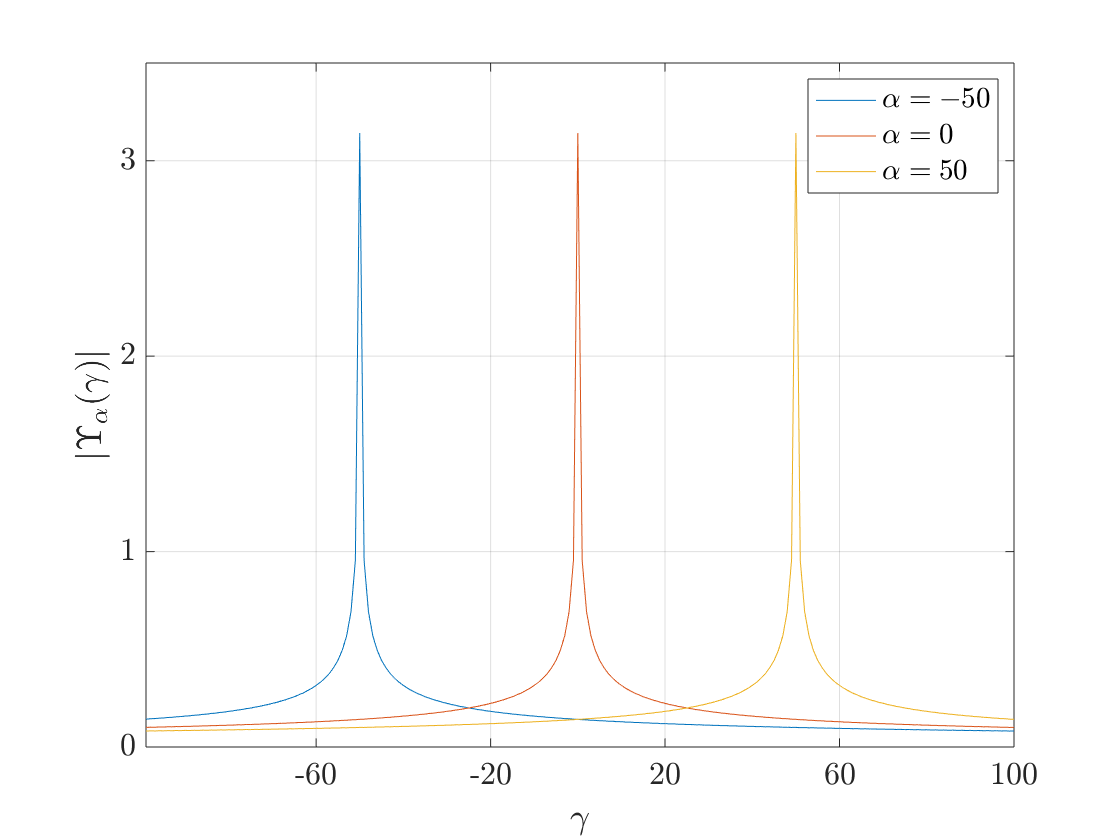}
\caption{$|\Upsilon_\alpha(\gamma)|$ for $\alpha=-50$, $0$, and $50$.}
\vspace{-0.2cm}
\label{frequency LR}
\end{figure}

According to (\ref{lr for frequency}),
every connection function of 
\begin{equation*}
	\Upsilon_\alpha(\gamma)  
		= \frac12\sum_{k=1}^{6,284}  \exp(-\pi i (\gamma-\alpha)x_k) \cdot 0.001 
\end{equation*}
is a translation of $\Upsilon_0$.
Figure \ref{frequency LR} shows three kinds of absolute values of $\Upsilon_\alpha$. By  (\ref{learning-rate1}),
the learning rate of frequency with the sample regenerated is as follows,
\begin{equation*}
	\frac1{\max_\alpha \mathrm{Var}(\Upsilon_\alpha)} 
	=\frac1{\mathrm{Var}(\Upsilon_0)} =   221.62. 
\end{equation*}

\smallskip

\noindent\textsc{Interpretation of estimations.}
Let $\mathpzc{D}$ be a dictionary of $|\alpha|\le100$. Red dots of Figures \ref{active neurons2} denote the active neurons for $x=0.75$, $1$, $1.5$, and $3$, respectively, 

Active neurons have specific patterns formed by bundles of neighboring neurons. 
In pictures  of Figure \ref{active neurons2} 
we draw neighborhood groups of red areas of rectangles with rounded corners only on positive index set because of symmetry.
As $x$ increases to $1$, the adjacent groups of active neurons are distributed between frequencies $10$ and $30$ in $(b)$ of Figure \ref{active neurons2}. On the other hand, If $x$ is far from $1$, the number of groups decreases, whereas the ball size increases. Moreover, if $x=3$, then there is no such groups between $10$ and $30$ in $(d)$ of Figure \ref{active neurons2}.

In addition, we examine topological statistics on $\Gamma_{\mathpzc{D}}$ 
for $x=0.75$, $1$, $1.5$, $3$, and $0$.
In the union of $k$-balls of each active path,
the number of balls, 
mean, and variance of indexes  
are compared in Figures \ref{frequency topology}, \ref{frequency means}, and \ref{frequency variances}, respectively.

Finally, from the network probability,
\begin{equation} \label{frequency network probability}
	\argmax_{x\in\mathbb{R}}p(x\mid y) = \argmax_{-1 \le x \le 1}p(x\mid y) = \pm 1,
\end{equation}
i.e., the likelihood has the maximum at the node and antinode.
Hence,  a random variable of amplitude for a standing wave has the maximum likelihood at peaks.
With all $y_\alpha$, (\ref{frequency network probability}) and $p_0(x)$ are identical. 
We can calculate likelihoods $p(x\mid \mathpzc{D})$ if the dictionary is main focus of interest.

\paragraph{Moment learning.}
We derive the network probability for $(x,\dot{x})$ as a moment learning and analyze estimations of the network through the topological interpretation.  
Suppose that there are $32$ samples of $(x,\dot{x})$ drawn from (\ref{governing equation})  with time step $h_t = 0.2$ for $0\le t \le 2 \pi$ in Figure \ref{sample of standing wave}. 
By (\ref{multi-empirical}), the empirical distribution function $\hat{F}_{32}(x,\dot{x})$ is shown in Figure \ref{empirical cdf for moment}.
 
Using the KMD with $h_t = 0.001$, we gain 6,284 samples (Figure \ref{m generated sample1}) with which
we draw 
the limiting distribution $F(x,\dot{x})$ and governing probability $p_0(x,\dot{x})$
in Figures \ref{limiting distribution for moment} and \ref{network probability for moment}, respectively.
Here, $p_0$ is calculated by partial derivatives of a finite-difference method at centers of $300\times300$ equally spaced bins.

The sample  has a uniform time step $h_t= 0.001$.
By (\ref{learning-rate2})
the learning rate for moment is given by
\begin{equation*}
	\frac1{\max_{k} \|h_k\|} = \frac{1}{0.001}=\num[group-separator={,}]{1000}.
\end{equation*}

According to topological statistics of Figures \ref{frequency topology} $\sim$ \ref{frequency variances}, 
mean- and variation-distributions have unstable behaviors on small balls when $x=1$. 
Note that $x=1$ is the magnitude of an antinode.
A loss of stem in figures is caused by 
absence of $k$-ball in $\Gamma_{\mathpzc{D}}$, and so, there does not exist any topological statistic at the position.

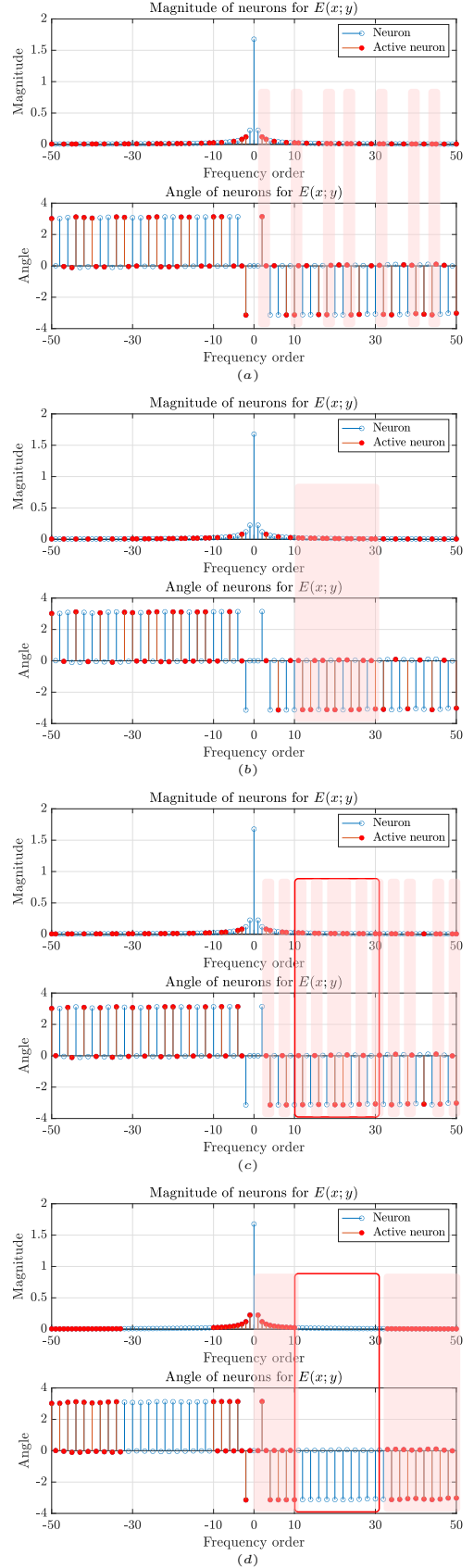
\begin{figure}[H]
   \centering
\begin{tikzpicture}[scale=0.533]
\node[inner sep=0pt] (active) at (-0.44,0)
   {\includegraphics[width=0.85\columnwidth,trim={0cm 0.25cm 1cm 0.25cm},clip]{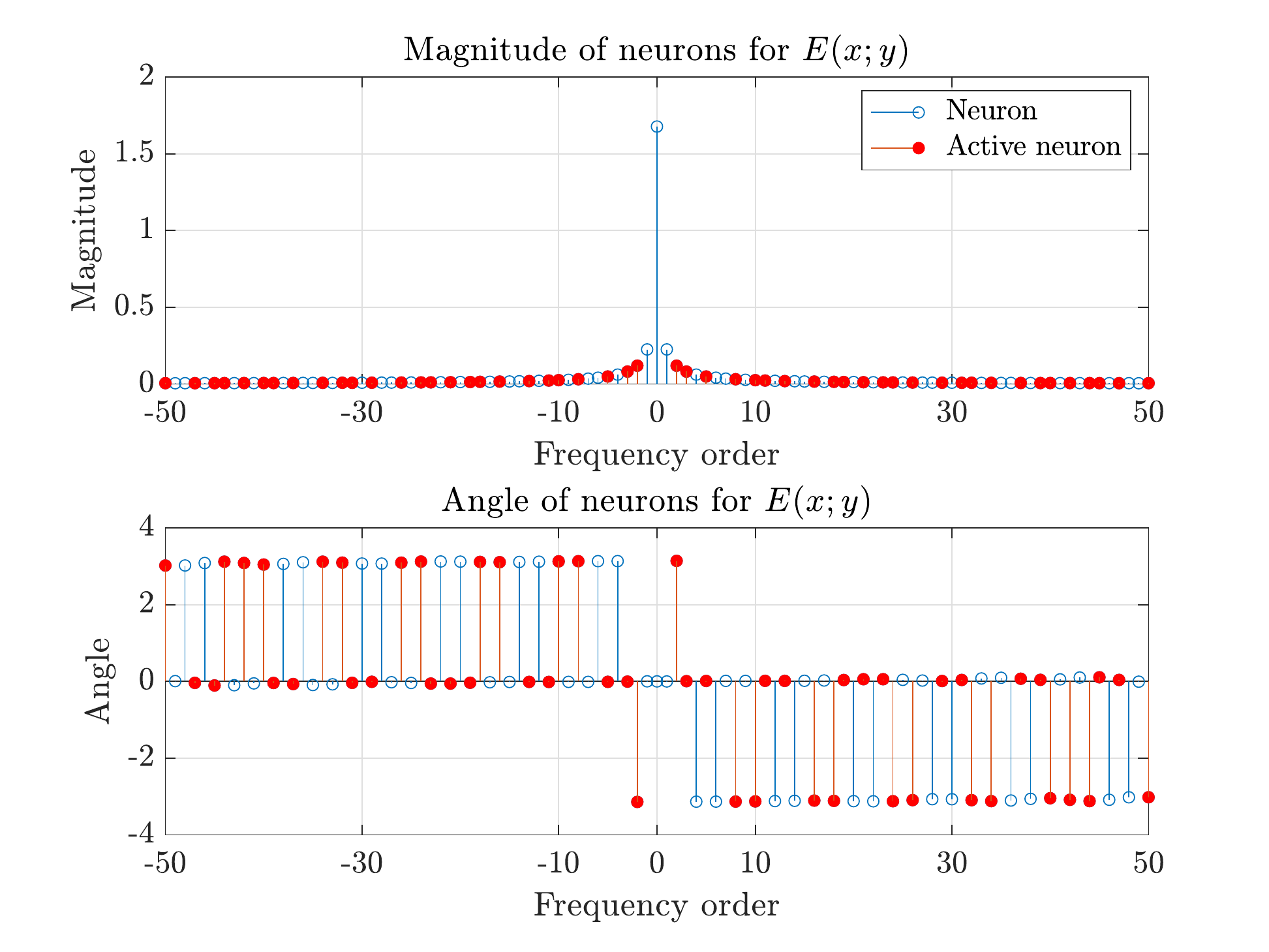}};
   \fill[fill=red!20, rounded corners=0.5mm, fill opacity=0.4] (0.27, -3.9) rectangle (0.59,2.5) {};   
   \fill[fill=red!20, rounded corners=0.5mm, fill opacity=0.4] (1.14, -3.9) rectangle (1.45,2.5) {};   
   \fill[fill=red!20, rounded corners=0.5mm, fill opacity=0.4] (2.0, -3.9) rectangle (2.32,2.5) {};   
   \fill[fill=red!20, rounded corners=0.5mm, fill opacity=0.4] (2.54, -3.9) rectangle (2.86,2.5) {};   
   \fill[fill=red!20, rounded corners=0.5mm, fill opacity=0.4] (3.41, -3.9) rectangle (3.72,2.5) {};   
   \fill[fill=red!20, rounded corners=0.5mm, fill opacity=0.4] (4.27, -3.9) rectangle (4.58,2.5) {};   
   \fill[fill=red!20, rounded corners=0.5mm, fill opacity=0.4] (4.81, -3.9) rectangle (5.125,2.5) {};   
	\node [text centered, text width=3.4cm] at (0,-5.2)   {\tiny $(a)$};
\end{tikzpicture}

\begin{tikzpicture}[scale=0.533]
\node[inner sep=0pt] (active) at (-0.44,0)
   {\includegraphics[width=0.85\columnwidth,trim={0cm 0.25cm 1cm 0.25cm},clip]{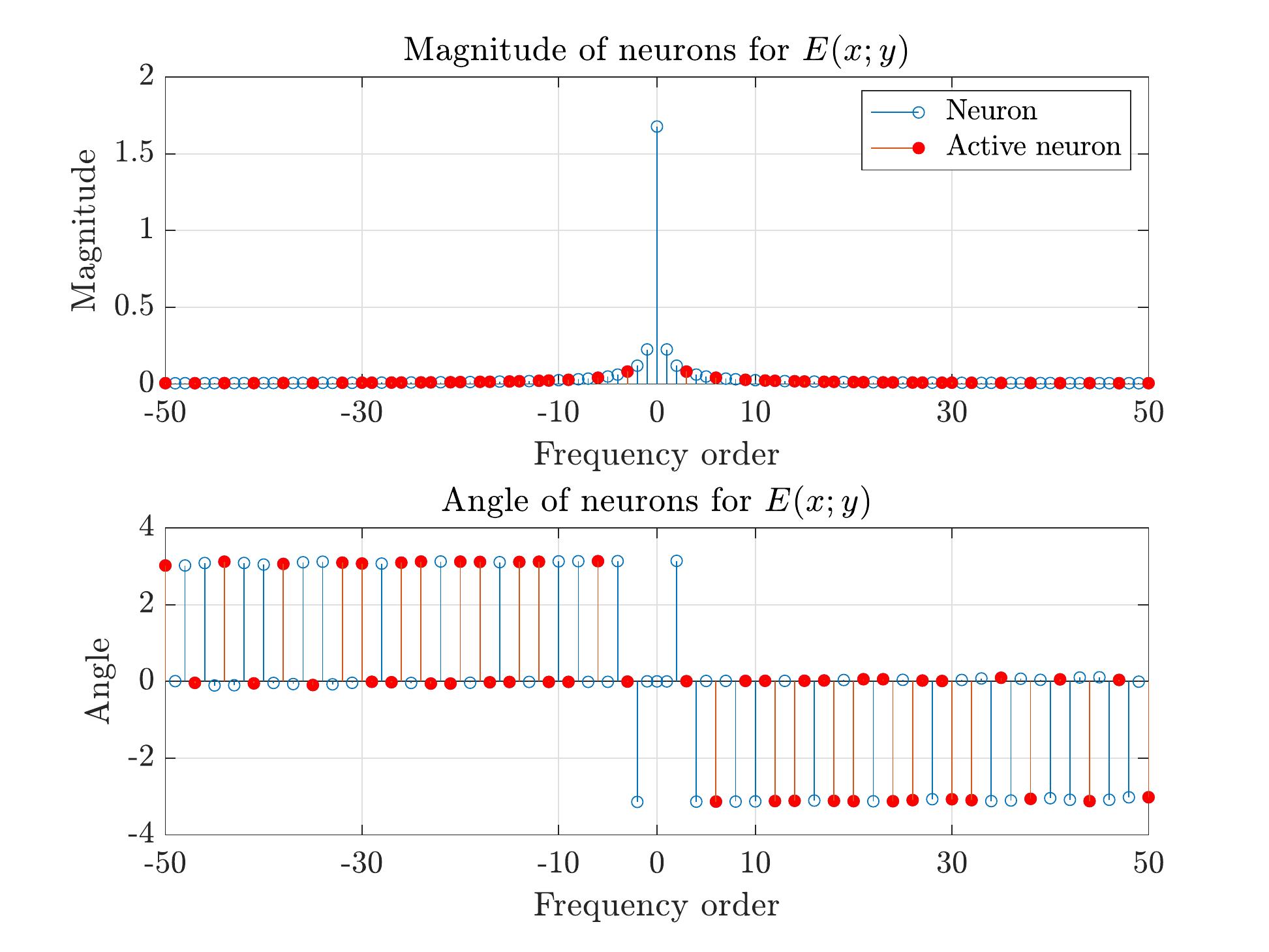}};
   \fill[fill=red!20, rounded corners=0.5mm, fill opacity=0.4] (1.25, -3.9) rectangle (3.5,2.5) {};   
	\node [text centered, text width=3.4cm] at (0,-5.2)   {\tiny $(b)$};
\end{tikzpicture}

\begin{tikzpicture}[scale=0.533]
\node[inner sep=0pt] (active) at (-0.44,0)
   {\includegraphics[width=0.85\columnwidth,trim={0cm 0.25cm 1cm 0.25cm},clip]{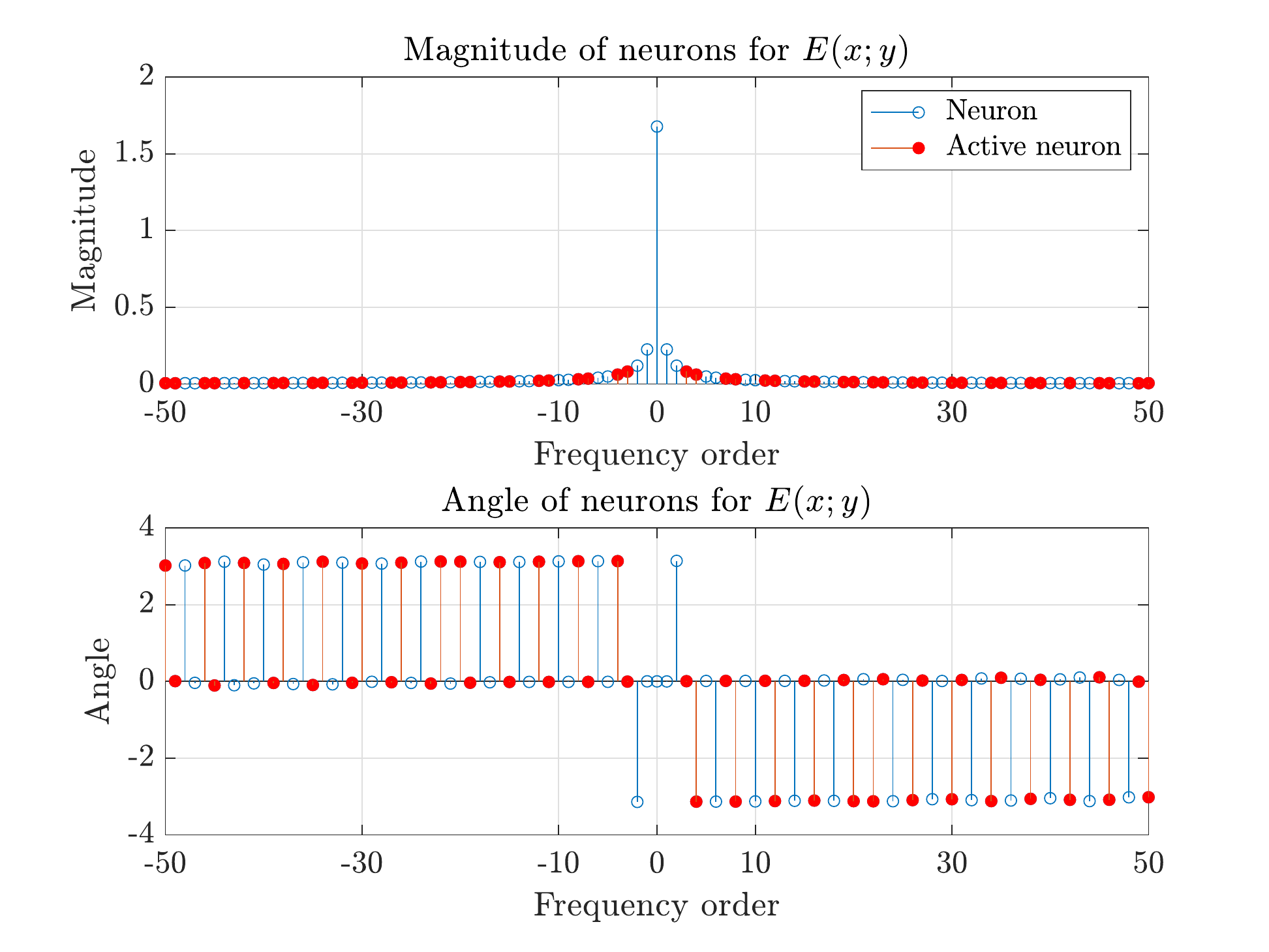}};
   \fill[fill=red!20, rounded corners=0.5mm, fill opacity=0.4] (0.38, -3.9) rectangle (0.7,2.5) {};   
   \fill[fill=red!20, rounded corners=0.5mm, fill opacity=0.4] (0.816, -3.9) rectangle (1.125,2.5) {};   
   \fill[fill=red!20, rounded corners=0.5mm, fill opacity=0.4] (1.245, -3.9) rectangle (1.56,2.5) {};   
   \draw[line width=0.2mm,red, rounded corners=0.5mm, fill opacity=0] (1.25, -3.9) rectangle (3.5,2.5) {};   
   \fill[fill=red!20, rounded corners=0.5mm, fill opacity=0.4] (1.68, -3.9) rectangle (1.99,2.5) {};   
   \fill[fill=red!20, rounded corners=0.5mm, fill opacity=0.4] (2.11, -3.9) rectangle (2.75,2.5) {};   
   \fill[fill=red!20, rounded corners=0.5mm, fill opacity=0.4] (2.868, -3.9) rectangle (3.183,2.5) {};   
   \fill[fill=red!20, rounded corners=0.5mm, fill opacity=0.4] (3.3, -3.9) rectangle (3.614,2.5) {};   
   \fill[fill=red!20, rounded corners=0.5mm, fill opacity=0.4] (3.73, -3.9) rectangle (4.045,2.5) {};   
   \fill[fill=red!20, rounded corners=0.5mm, fill opacity=0.4] (4.16, -3.9) rectangle (4.475,2.5) {};   
   \fill[fill=red!20, rounded corners=0.5mm, fill opacity=0.4] (4.92, -3.9) rectangle (5.233,2.5) {};   
   \fill[fill=red!20, rounded corners=0.5mm, fill opacity=0.4] (5.353, -3.9) rectangle (5.662,2.5) {};   
	\node [text centered, text width=3.4cm] at (0,-5.2)   {\tiny $(c)$};
\end{tikzpicture}

\begin{tikzpicture}[scale=0.533]
\node[inner sep=0pt] (active) at (-0.44,0)
   {\includegraphics[width=0.85\columnwidth,trim={0cm 0.25cm 1cm 0.25cm},clip]{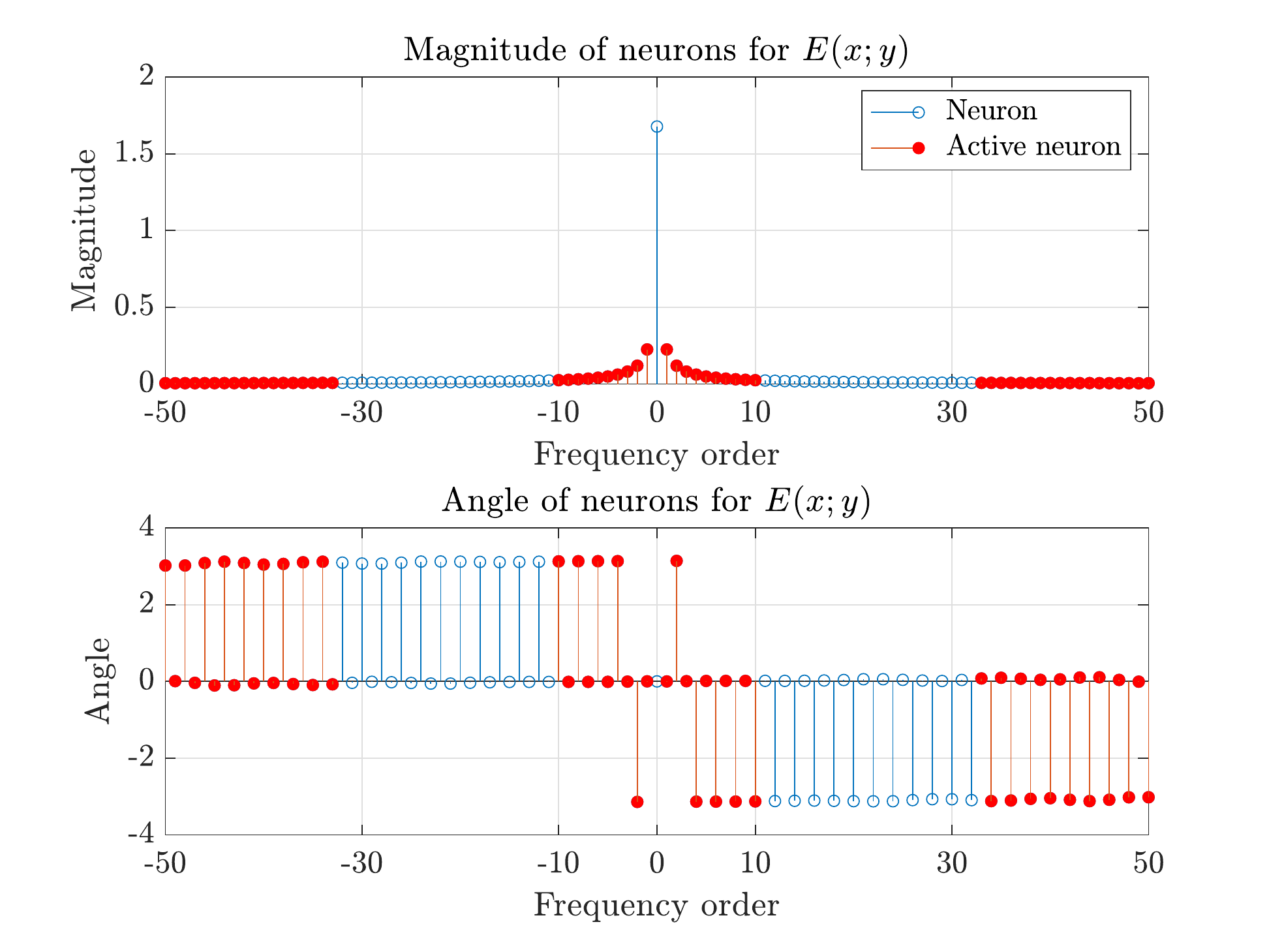}};
   \draw[line width=0.2mm,red, rounded corners=0.5mm, fill opacity=0] (1.25, -3.9) rectangle (3.5,2.5) {};   
      \fill[fill=red!20, rounded corners=0.5mm, fill opacity=0.4] (0.15, -3.9) rectangle (1.35,2.5) {};   
      \fill[fill=red!20, rounded corners=0.5mm, fill opacity=0.4] (3.62, -3.9) rectangle (5.662,2.5) {};   
	\node [text centered, text width=3.4cm] at (0,-5.2)   {\tiny $(d)$};
\end{tikzpicture}
   \captionof{figure}{$(a)$ $x=0.75$, $(b)$ $x=1$, $(c)$ $x=1.5$, $(d)$ $x=3$. Red rectangular areas with rounded corners denote groups of adjacent active neurons.  Red rounded rectangles of $(c)$ and $(d)$ are copies from $(b)$.}
   \label{active neurons2}
\end{figure}

On the other hand, 
the induced $p_0$  is almost  cylindrical and every partial derivative of $p_0$ vanishes at the origin. 
More precisely,  $p_0$  is an approximation of singular measure which is supported in  
the unit circle $S^1$  and 
\begin{equation*}
	\iint_{\mathbb{T}^2} p_0(x,\dot{x})\,dx\,d\dot{x} = 1.
\end{equation*}
Thus, $\ln 1/p_0$ does not belong to $A_0$  on $\mathbb{T}^2$. 
For this reason, the energy function cannot be expanded as a power series near the origin.

To avoid a dead end, we adopt an auxiliary function as a pullback limit of $p_0$. Since $p_0$ could be regarded as the limit of suitable 
analytic functions, take an approximation of  $p_0$, for example, 
\begin{equation} \label{aux}
	p_a(x,\dot{x}) = 
		\left\{
 		\begin{array}{cl}  
		C(1-x^2-\dot{x}^2)^{-\nicefrac{1}{2}}\quad & \mbox{if }\; x^2+\dot{x}^2\le 1 \\
		0\quad & \mbox{otherwise},
 		\end{array}
		\right.
\end{equation}
where $C$ is chosen such that $\|p_a\|_{L^1}=1$ (see Figure \ref{auxiliary network probability}).
(In fact, the bowl-shaped functions such as $p_a$ have a similar distribution of neurons 
to Figures \ref{active neurons for nonzero-zero} $\sim$ \ref{active neurons for nonzero-nonzero1}.)

Table \ref{three signals} presents some examples of likelihoods of $p_a(x,\dot{x}\mid y)$ and $p_0(x,\dot{x})$ with zero velocity. 
All express that $(1,0)$ has the maximum likelihood when $\dot{x}=0$.

\begin{table}[H]%
  \begin{center}
    \caption{Likelihoods.}
    \label{three signals}
    \begin{tabular}{c|c|c} 
      \toprule 
{\small $(x,\,0)$}  &   {\small $p_a(x,0\mid y)$}  &  {\small $p_0(x,0)$} \\ 
      \midrule 
{\small $0$} & {\small $0.1592$}  &  {\small $0$} \\ 
{\small $\nicefrac{1}{\sqrt{2}}$} & {\small $0.2251$}  & {\small $0$} \\
{\small $1$} & {\small $\infty$}  & {\small $23.9926$} \\
{\small $|x|>1$} & {\small $0$}  & {\small $0$} \\
      \bottomrule 
    \end{tabular}
  \end{center}
  \vspace{-0.3cm}
\end{table}

\begin{figure}[H]
\vspace{-0.3cm}    
\centering
    \includegraphics[width=0.82\columnwidth,trim={0.5cm 0cm 0cm 1.5cm},clip]{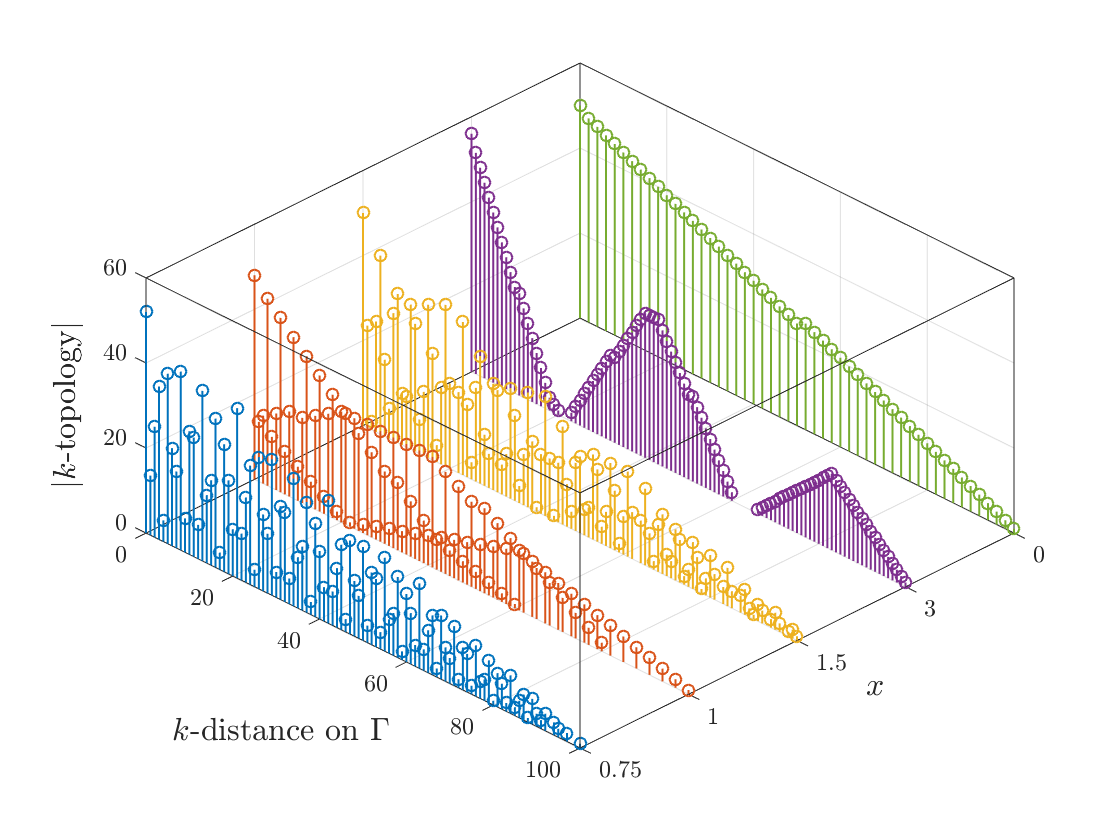}
\vspace{-0.5cm}    
\caption{The number of $k$-distance open balls on each active path for $x=0.75$, $1$, $1.5$, $3$, and $0$. Distribution of the number of $k$-topology for $x=3$, $0$ are more regular and simpler than others.}
\label{frequency topology}
\end{figure}

\begin{figure}[H]
\vspace{-0.3cm}    
\centering
    \includegraphics[width=0.82\columnwidth,trim={0.5cm 0cm 0cm 1.5cm},clip]{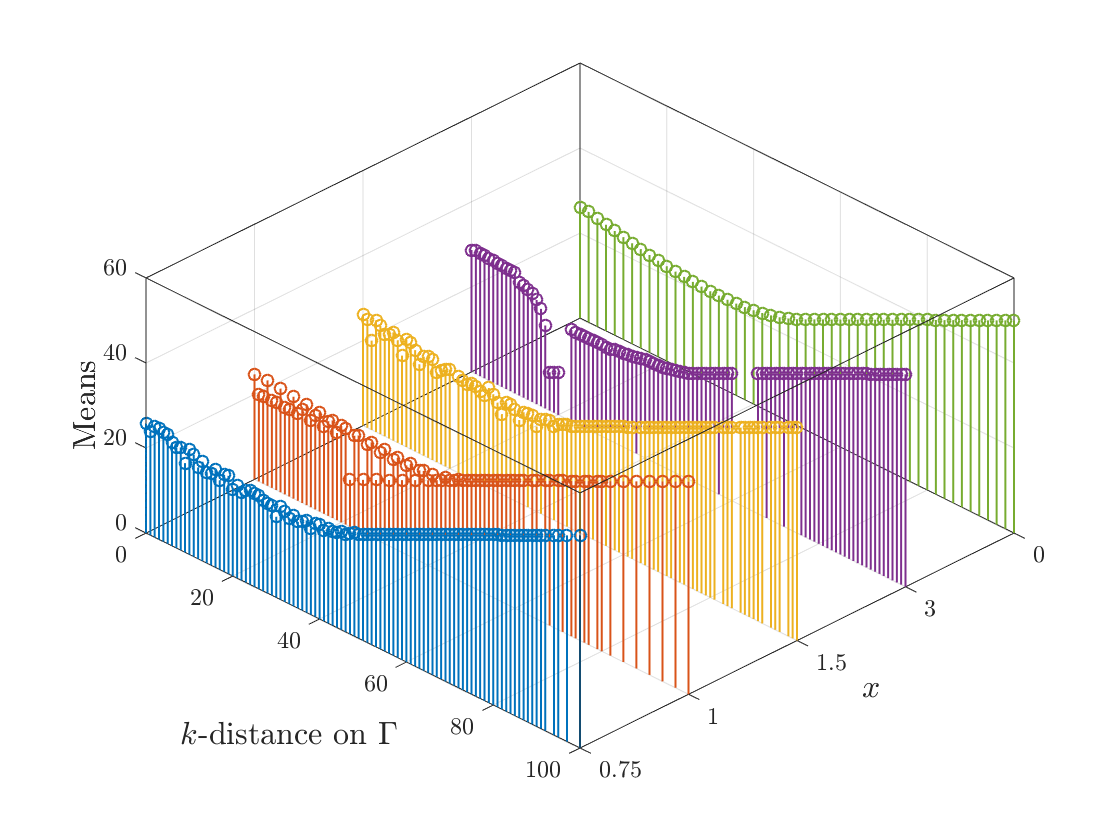}
\vspace{-0.5cm}    
\caption{The means of $k$-distance neurons on active paths for $x=0.75$, $1$, $1.5$, $3$, and $0$. 
The mean-distributions for small balls for $x=0.75$, $1$, $1.5$ are more unstable than others.}
\label{frequency means}
\end{figure}

\begin{figure}[H]
\centering
    \includegraphics[width=0.82\columnwidth,trim={0.5cm 0cm 0cm 1.5cm},clip]{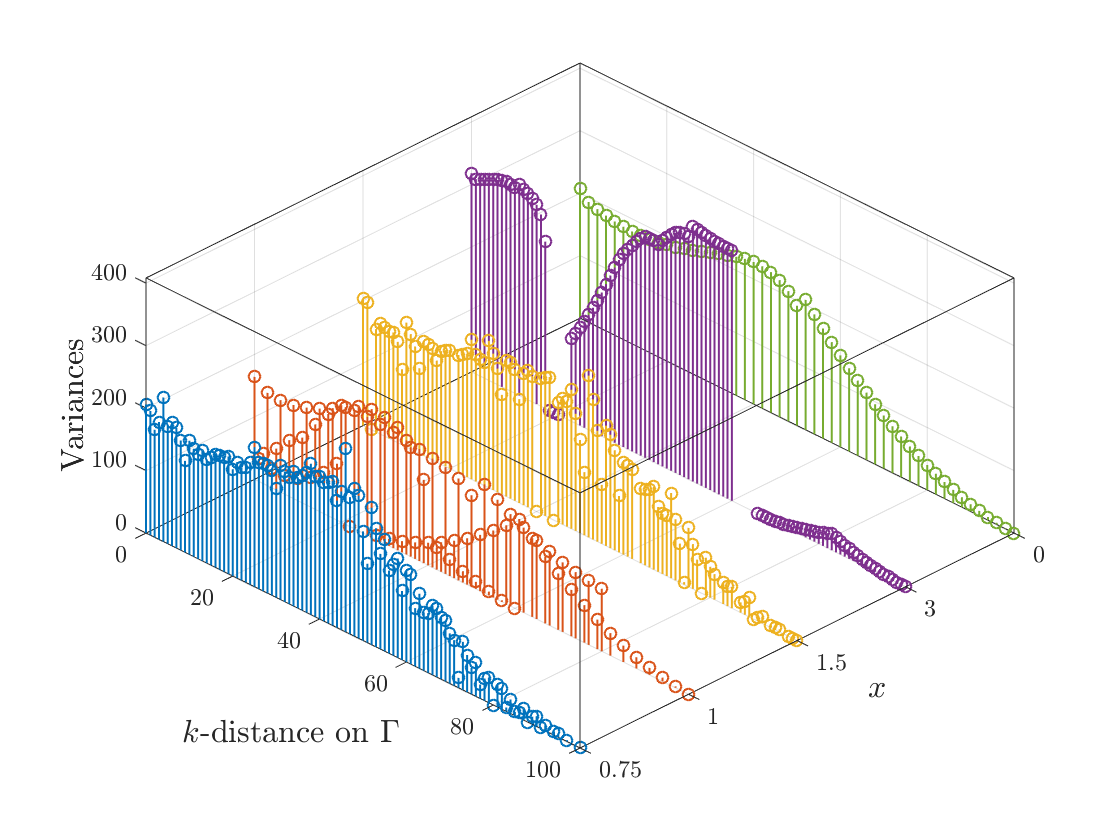}
\vspace{-0.5cm}    
\caption{The variances of $k$-distance neurons on active paths for $x=0.75$, $1$, $1.5$, $3$, $0$. 
The most unstable variance-distribution is of $x=1$.}
\label{frequency variances}
\end{figure}

\vspace{-0.2cm}

More precisely, 
\begin{equation*}
	\argmax_{x\in\mathbb{R}}p_a(x,0\mid y)
	=\argmax_{x\in\mathbb{R}}p_0(x,0) 
	=\pm1,
\end{equation*}
and thus, we conclude that  the likelihood of $(\pm1,0)$ has the maximum when $\dot{x}=0$. 
In general, for $-1\le \dot{x}\le 1$, we
obtain $x=\pm\sqrt{1-\dot{x}^2}$ which have the maximum likelihood of $p_0(x,\dot{x})$.
Therefore, $(x,\pm\sqrt{1-x^2})$  is most likely in the network (refer to Figures \ref{network probability for moment} and \ref{auxiliary network probability}).
In fact, the relation between $x$ and $\dot{x}$ is equal to $x^2 + \dot{x}^2 = 1$ for $0\le x\le 2 \pi$, e.g., $\dot{x}=\cos t$ when $x=\sin t$.

\begin{figure}[htp]
\vspace{-0.2cm}    
\centering
    \includegraphics[width=0.82\columnwidth,trim={0.5cm 0cm 1cm 0.75cm},clip]{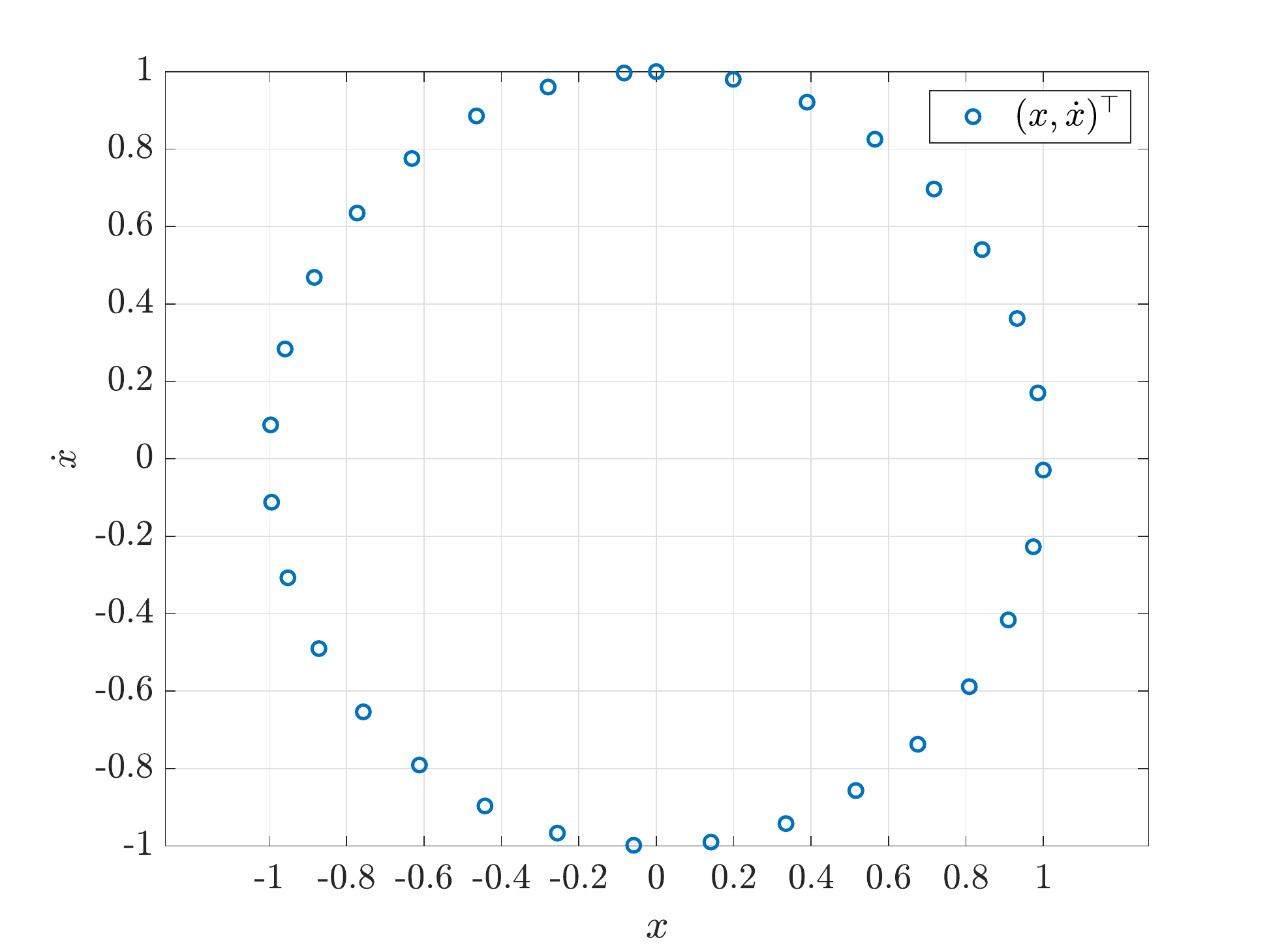}
\caption{Samples $(x,\dot{x})$ of size $32$.}
\label{sample of standing wave}
\end{figure}

\begin{figure}[htp]
\vspace{-0.3cm}    
\centering
    \includegraphics[width=0.82\columnwidth,trim={0.5cm 0cm 1cm 1.05cm},clip]{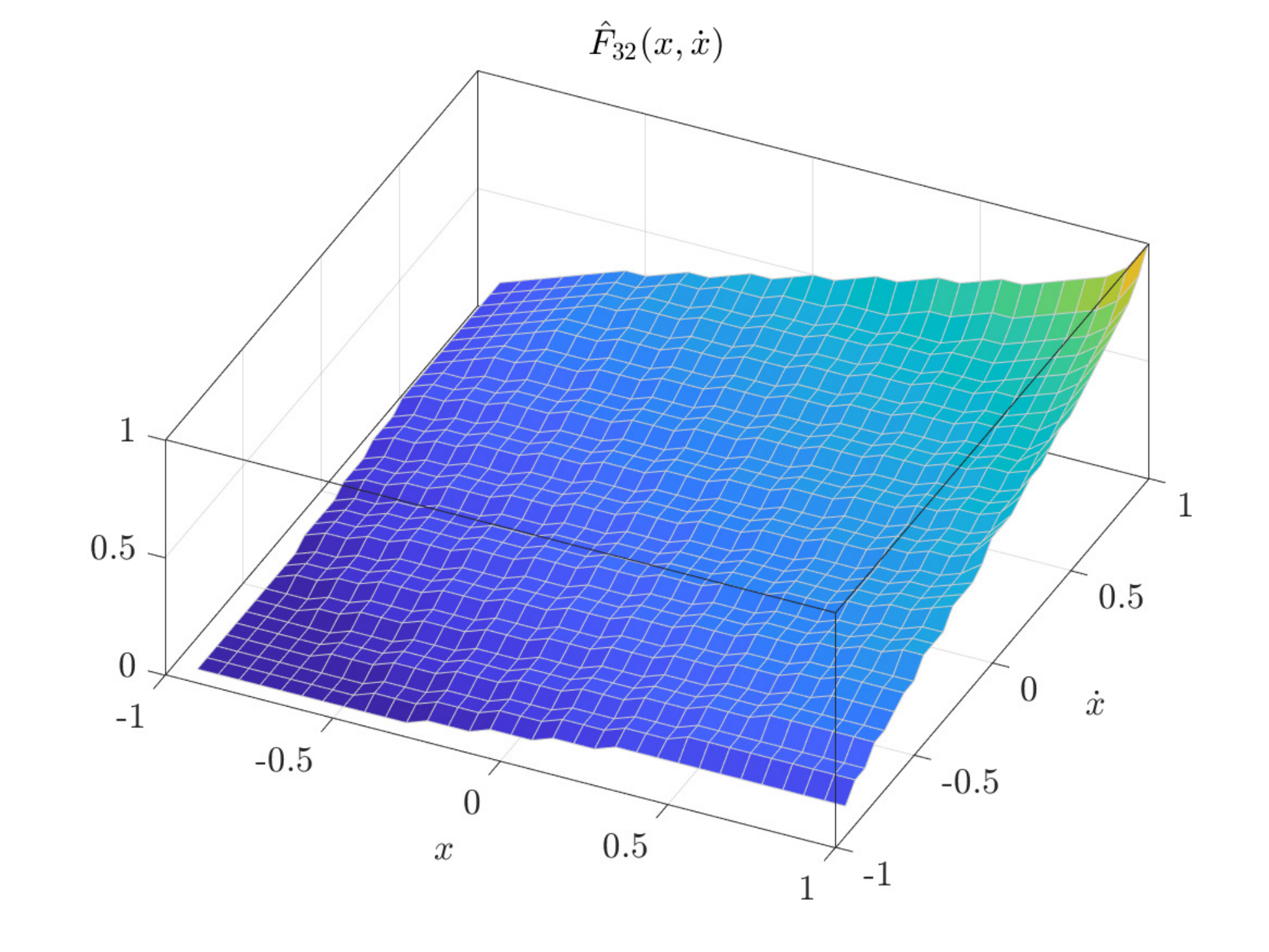}
\caption{The empirical distribution function $\hat{F}_{32}(x,\dot{x})$.}
\label{empirical cdf for moment}
\end{figure}

\begin{figure}[H]
\vspace{-0.0cm}    
\centering
    \includegraphics[width=0.82\columnwidth,trim={0.5cm 0cm 1cm 0.75cm},clip]{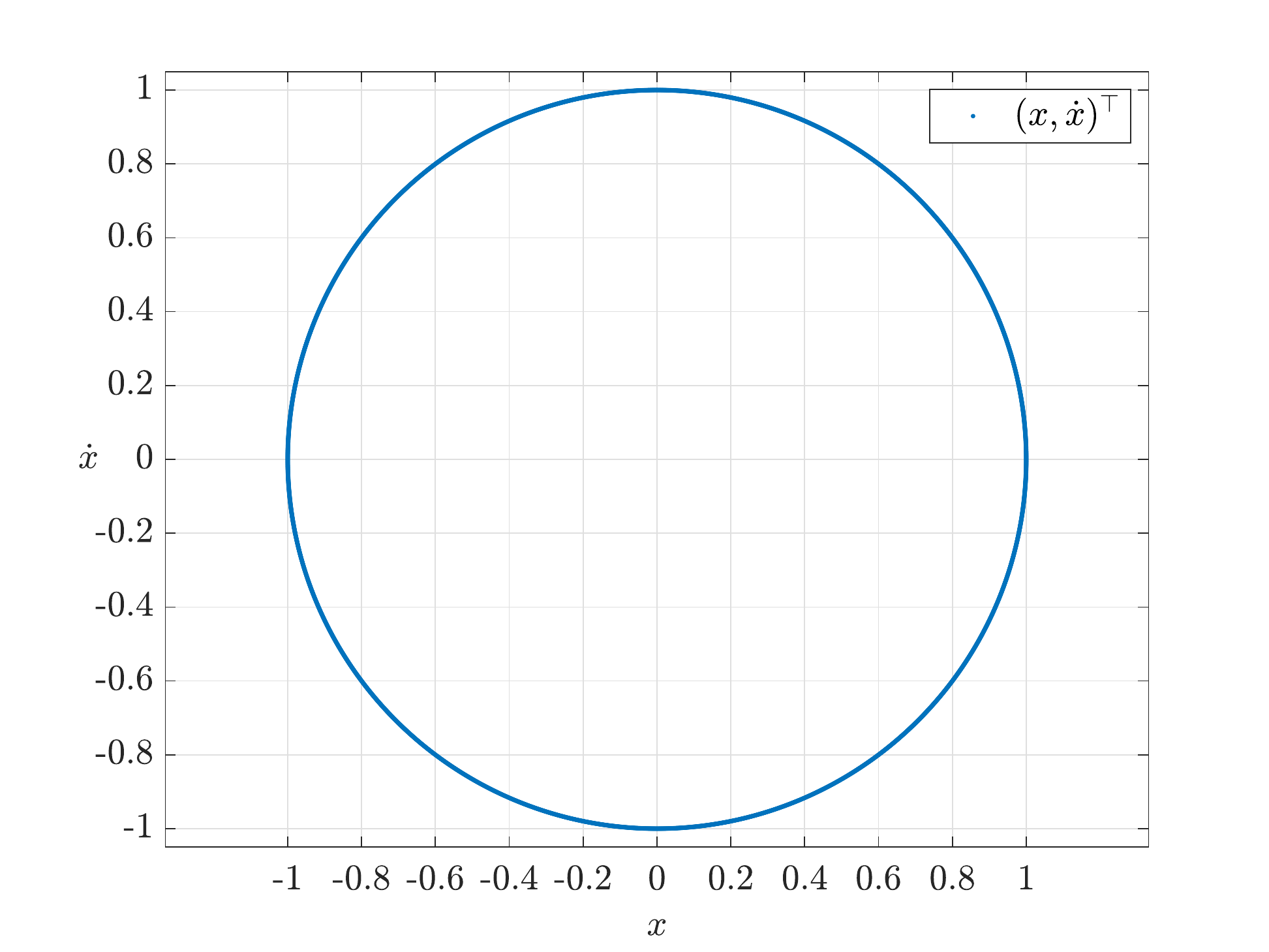}
\caption{Samples $(x,\dot{x})$ of size 6,284 with the KMD.}
   \vspace{-0.27cm}
\label{m generated sample1}
\end{figure}

\begin{figure}[H]
\vspace{-0.2cm}    
\centering
    \includegraphics[width=0.82\columnwidth,trim={0.5cm 0cm 1cm 0.5cm},clip]{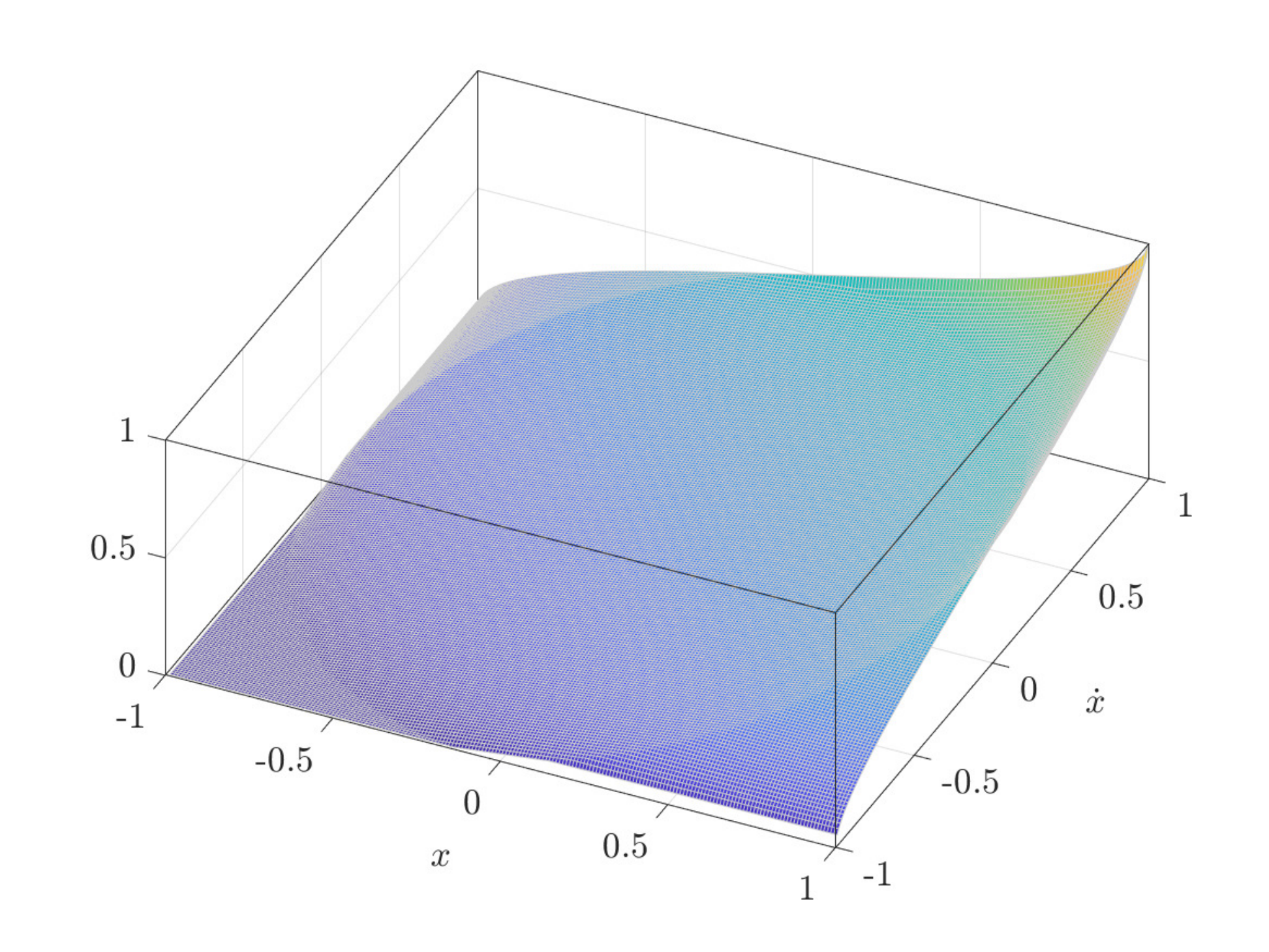}
\caption{The limiting distribution $F(x,\dot{x})$.}
   \vspace{-0.27cm}
\label{limiting distribution for moment}
\end{figure}

\begin{figure}[H]
\vspace{-0.2cm}    
\centering
    \includegraphics[width=0.82\columnwidth,trim={0.5cm 0cm 1cm 0.5cm},clip]{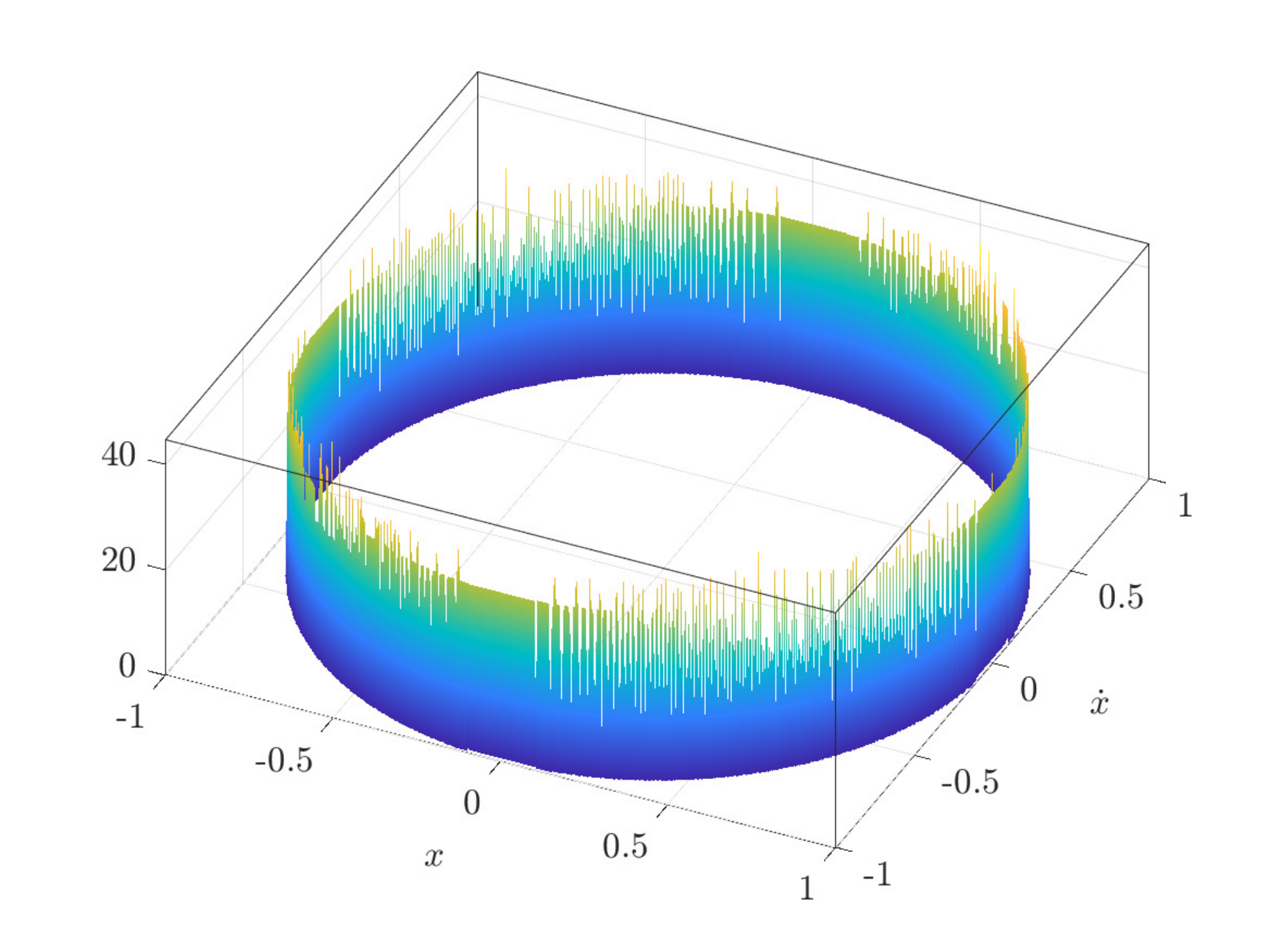}
\caption{The governing probability $p_0$.}
   \vspace{-0.25cm}
\label{network probability for moment}
\end{figure}

\begin{figure}[H]
\vspace{-0.2cm}    
\centering
    \includegraphics[width=0.82\columnwidth,trim={0.5cm 0cm 1cm 1.0cm},clip]{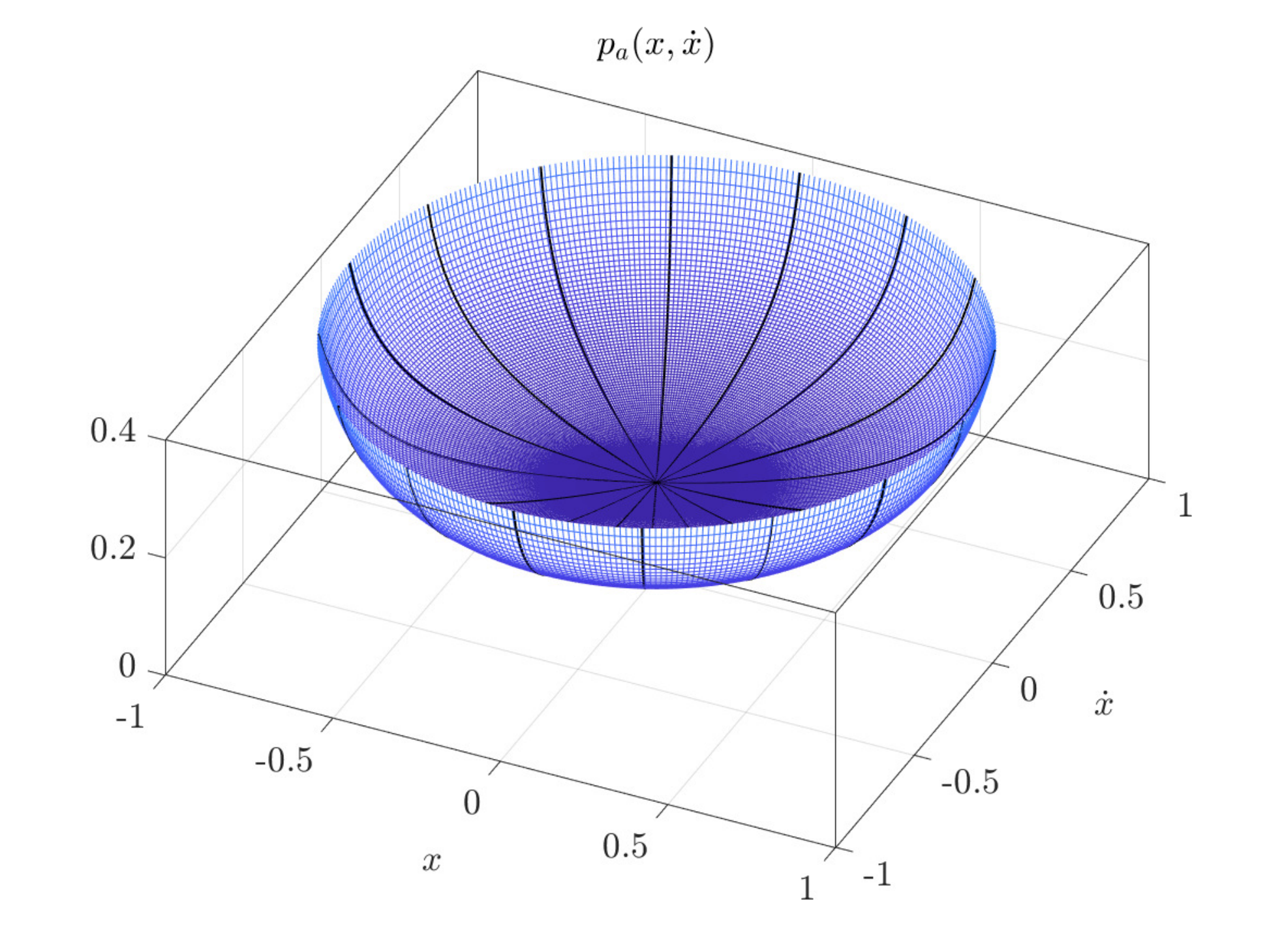}
\caption{Auxiliary governing probability $p_a(x,\dot{x})$.}
   \vspace{-0.35cm}
\label{auxiliary network probability}
\end{figure}

\smallskip
\noindent\textsc{Interpretation of estimations.}
Let $\mathpzc{D}$ be a dictionary  of indexes $\alpha$ such that $0\le \alpha_j\le20$ $(j=1,2)$.
The partial derivatives of $p_a$ at $(0,0)$ yield non-trivial neurons.
For estimation and interpretation, we consider $8$ signals of 
$(x,\dot{x})=(\pm1,0)$, $(0,\pm1)$, $(\pm1,\pm1)$, and $(\pm1,\mp1)$,
where double signs are in same order.

In Figures \ref{active neurons for nonzero-zero} and  \ref{active neurons for nonzero-nonzero1},
we draw neurons and active paths induced from $p_a$, in which red dots denote active paths for signals.
For simplicity, we adopt $\|\cdot\|_1$-norm to generate the topology, with which we
examine topological statistics on $\Gamma_{\mathpzc{D}}$. 
For each $k$, on the collection of all  $k$-balls,   
the number of balls, 
mean, and variance of indexes  
are compared in Figures \ref{moment topology}, \ref{moment means}, and \ref{moment variances}, respectively.

Note that all active paths do not contain any $k$-ball for $k\ge4$. 
In Figures \ref{moment means} and \ref{moment variances}, each position has values of two components.  
The right and left sides toward the distance increasing are 
the quantities from $\alpha_1$ and $\alpha_2$, respectively.

There are  plenty of mathematical research results for classification (\cite{bishop, goodfellow-bengio-courville,haykin, murphy}).  
Those could be useful tools for interpreting unsupervised estimates.

\begin{remark}
Without reproducing process of data, 
we could approximate the neurons of a probabilistic neural network. 
Indeed, the drawn sample would give the similar distribution of neurons to the probabilistic neural network.
In many cases we do not need thousands or more of data to calculate neurons. 
However, the accuracy of the derived governing probability may be poor and the learning rate is low, because the accuracy depends on the sample size and quality.    
\end{remark}

\begin{figure}[H]
\vspace{-0.5cm}
\centering
\subfloat[$x=1$, $\dot{x}=0$]{
\kern-0.77em
\includegraphics[width=0.82\columnwidth,trim={0.5cm 0cm 0.5cm 1cm},clip]{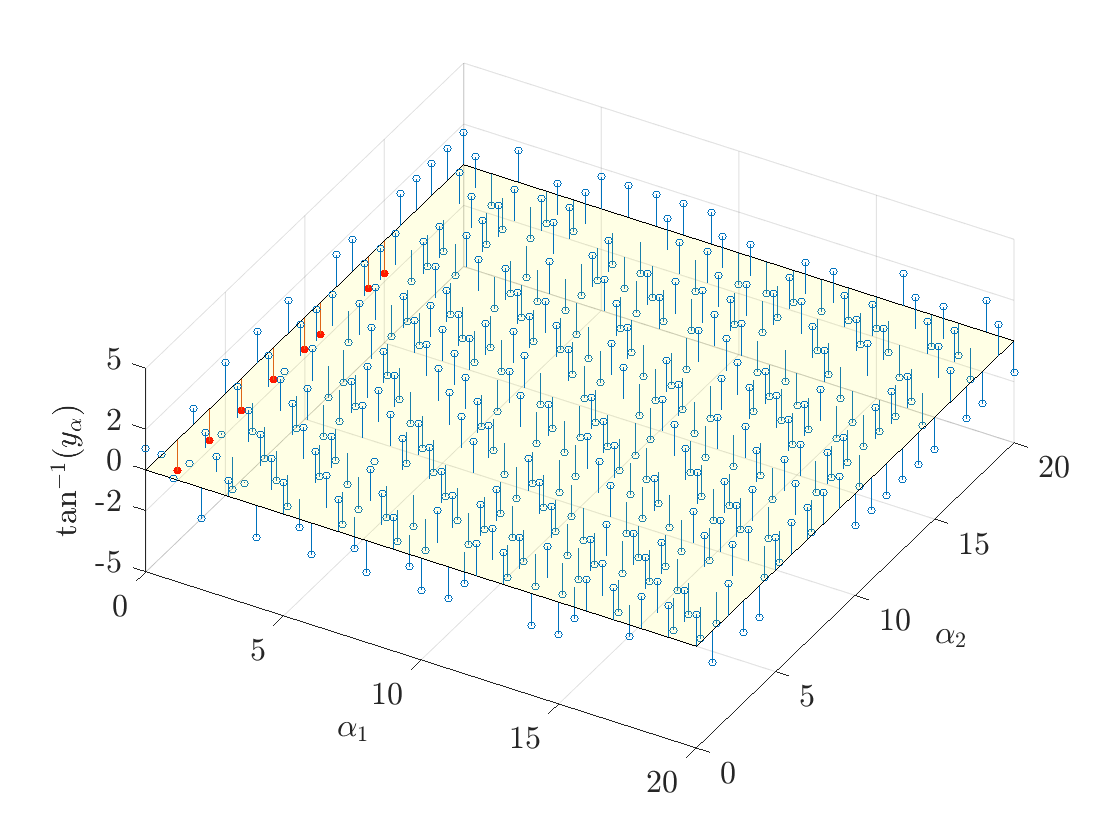}}
\vspace{-0.2cm}
\subfloat[$x=-1$, $\dot{x}=0$]{
\includegraphics[width=0.82\columnwidth,trim={0.5cm 0cm 0.5cm 1cm},clip]{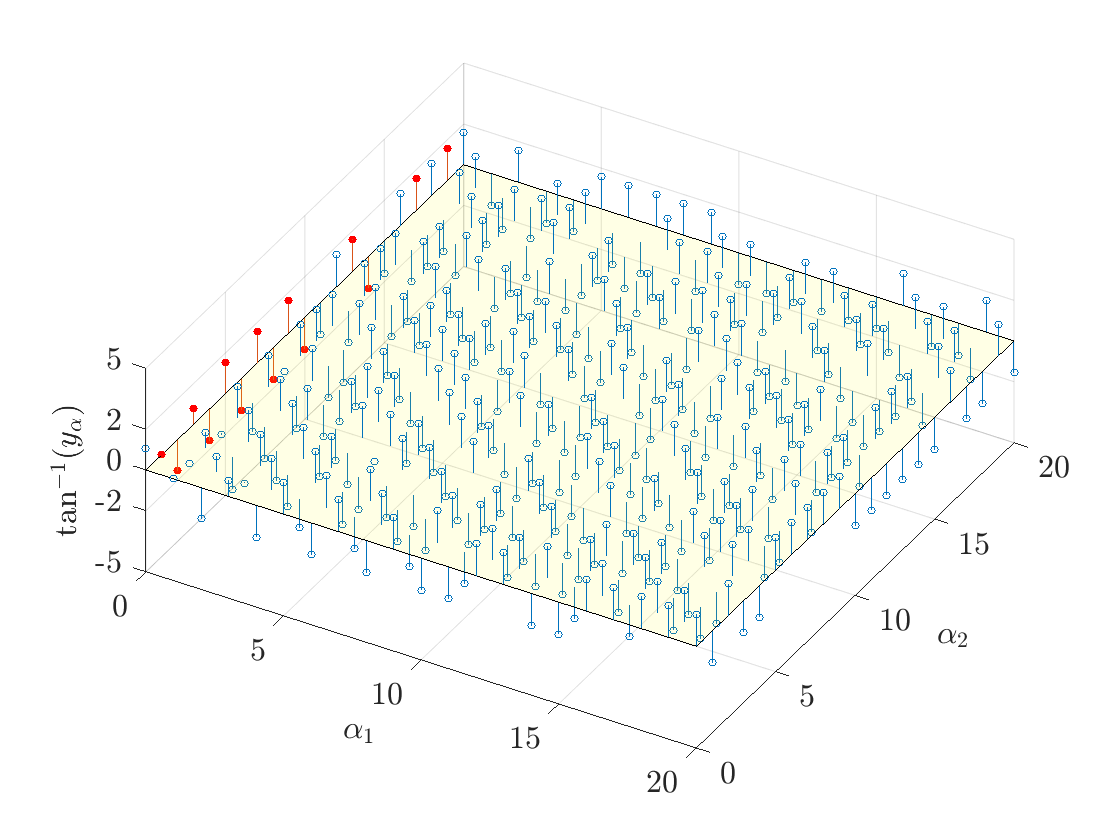}}
   \vspace{-0.2cm}
 \caption{Two active paths are parts of $\alpha_2$-axis. All active neurons of $(a)$ are negative but not $(b)$ and both topological statistics are very similar (Figures \ref{moment topology} $\sim$ \ref{moment variances}).}
\label{active neurons for nonzero-zero}
\end{figure}

\begin{figure}[H]
\vspace{-0.0cm}
\centering
\subfloat[$x=0$, $\dot{x}=1$]{
\kern-0.77em
\includegraphics[width=0.82\columnwidth,trim={0.5cm 0cm 0.5cm 1cm},clip]{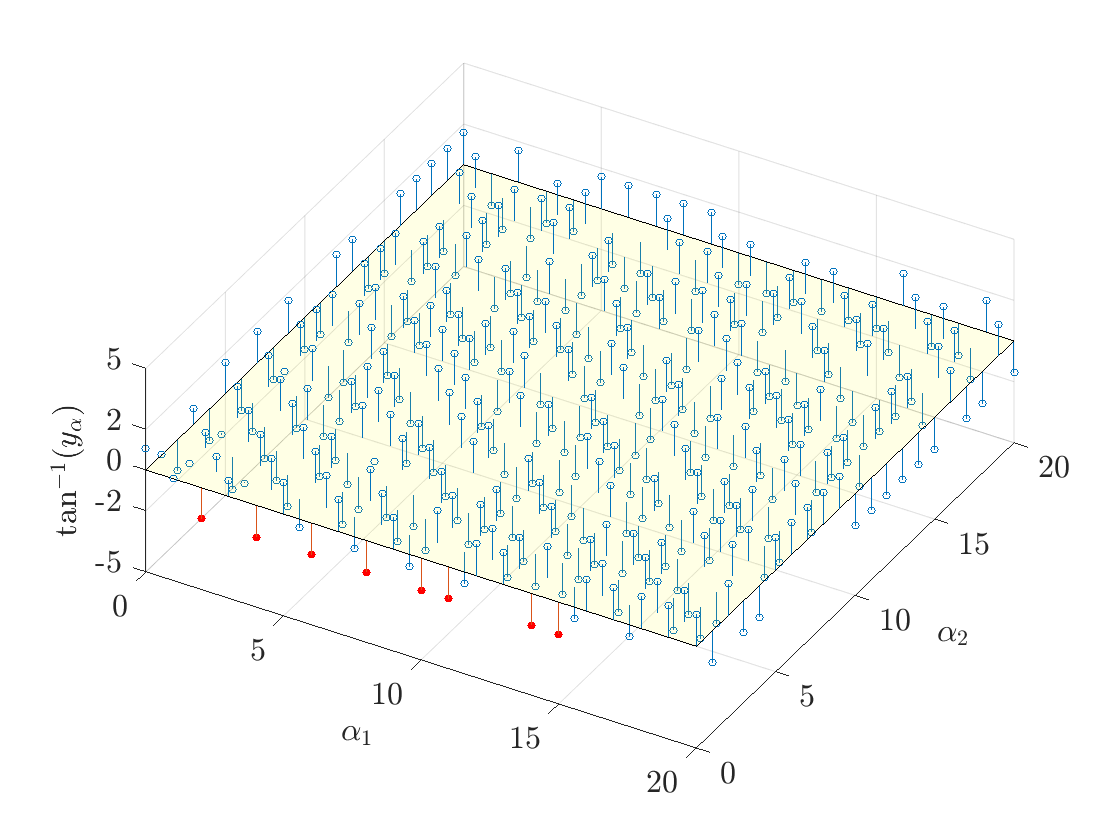}}

\vspace{-0.2cm}
\subfloat[$x=0$, $\dot{x}=-1$]{
\includegraphics[width=0.82\columnwidth,trim={0.5cm 0cm 0.5cm 1cm},clip]{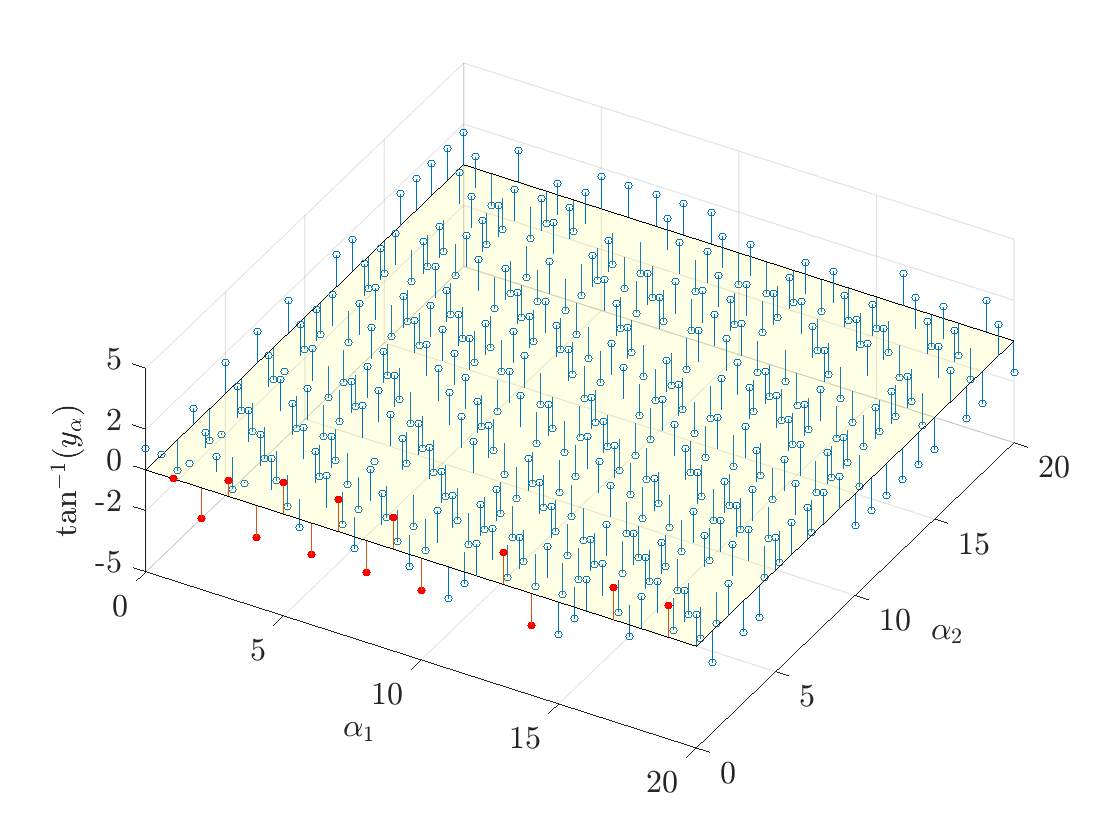}}
   \vspace{-0.2cm}
 \caption{Both have the symmetric transposition of axes compared to Figure \ref{active neurons for nonzero-zero}.}
\label{active neurons for zero-nonzero}
\end{figure}

\begin{figure}[H]
\vspace{-0.5cm}
\centering
\subfloat[$x=1$, $\dot{x}=1$]{
\kern-0.8em
\includegraphics[width=0.82\columnwidth,trim={0.5cm 0cm 0.5cm 1cm},clip]{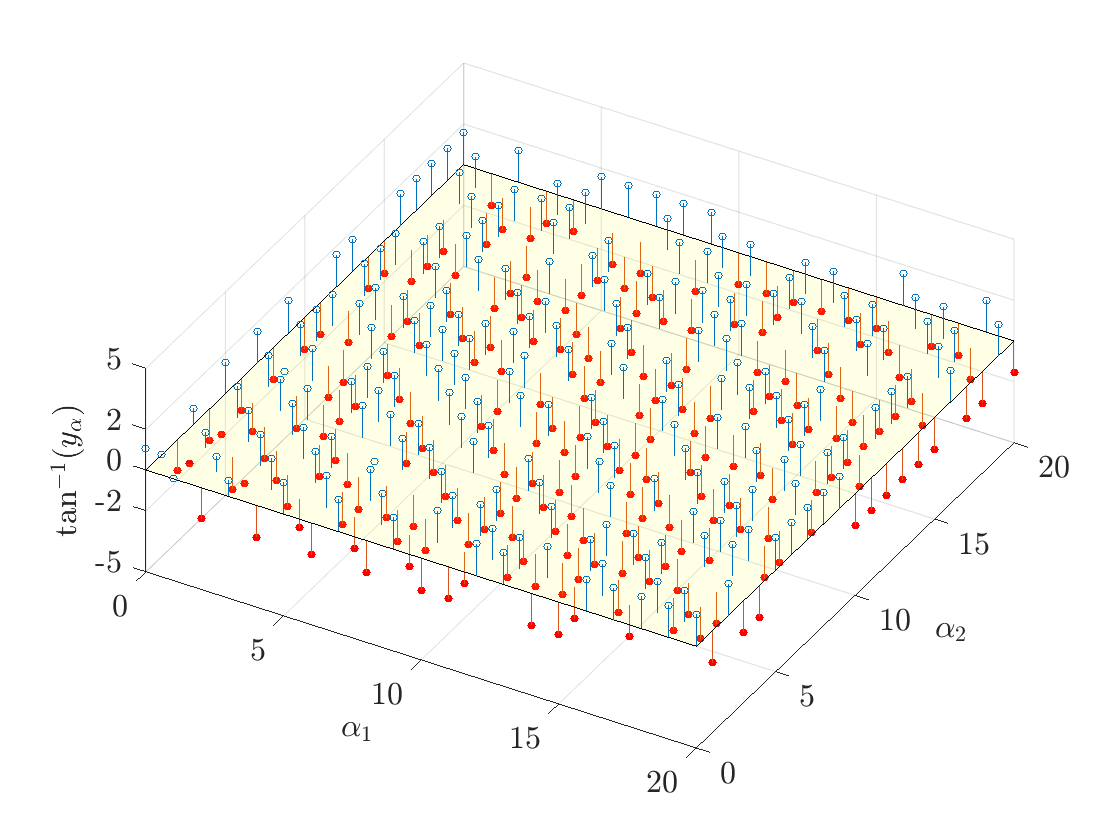}}

\vspace{-0.2cm}
\subfloat[$x=-1$, $\dot{x}=-1$]{
\includegraphics[width=0.82\columnwidth,trim={0.5cm 0cm 0.5cm 1cm},clip]{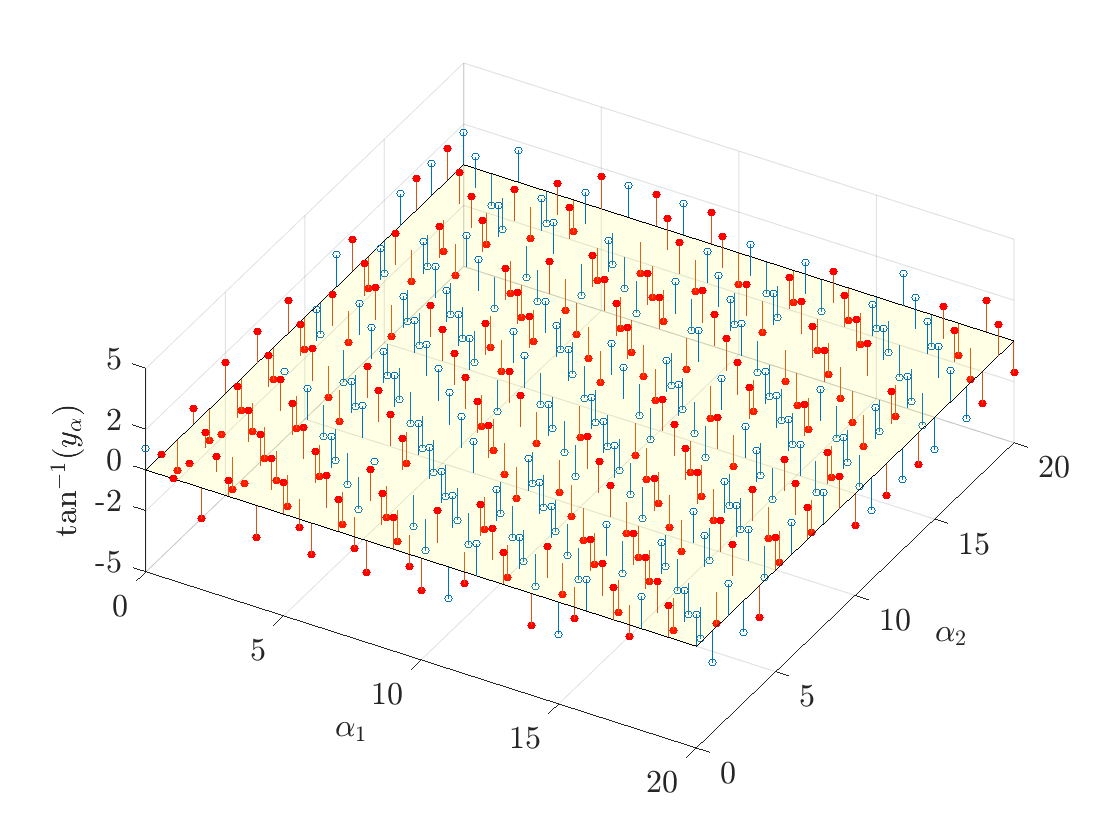}}
   \vspace{-0.2cm}
 \caption{Both have non-trivial active neurons in contrast to Figure \ref{active neurons for zero-nonzero}.  
 Active neurons of $(a)$  are negative but not $(b)$. Topological statistics are different (Figures \ref{moment topology} $\sim$ \ref{moment variances}).}
\label{active neurons for nonzero-nonzero}
\end{figure}

\begin{figure}[H]
\vspace{-0.0cm}
\centering
\subfloat[$x=1$, $\dot{x}=-1$]{
\kern-0.77em
\includegraphics[width=0.82\columnwidth,trim={0.5cm 0cm 0.5cm 1cm},clip]{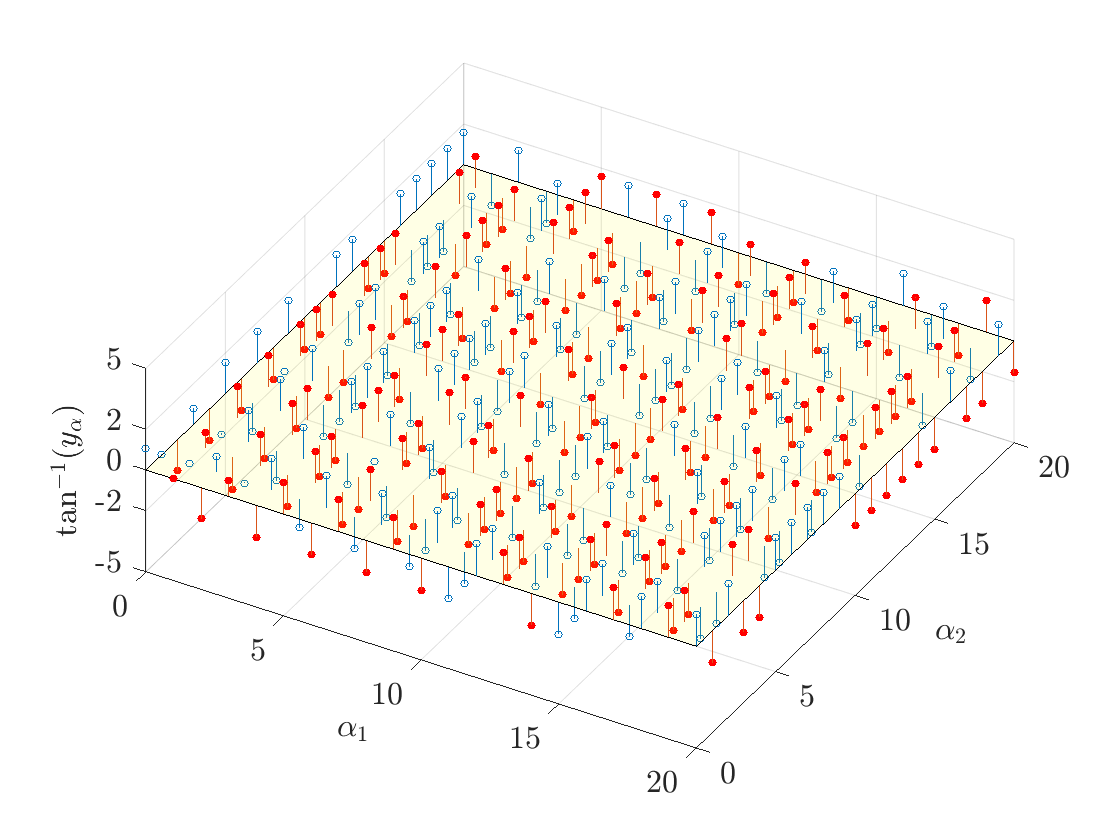}}

\vspace{-0.2cm}
\subfloat[$x=-1$, $\dot{x}=1$]{
\includegraphics[width=0.82\columnwidth,trim={0.5cm 0cm 0.5cm 1cm},clip]{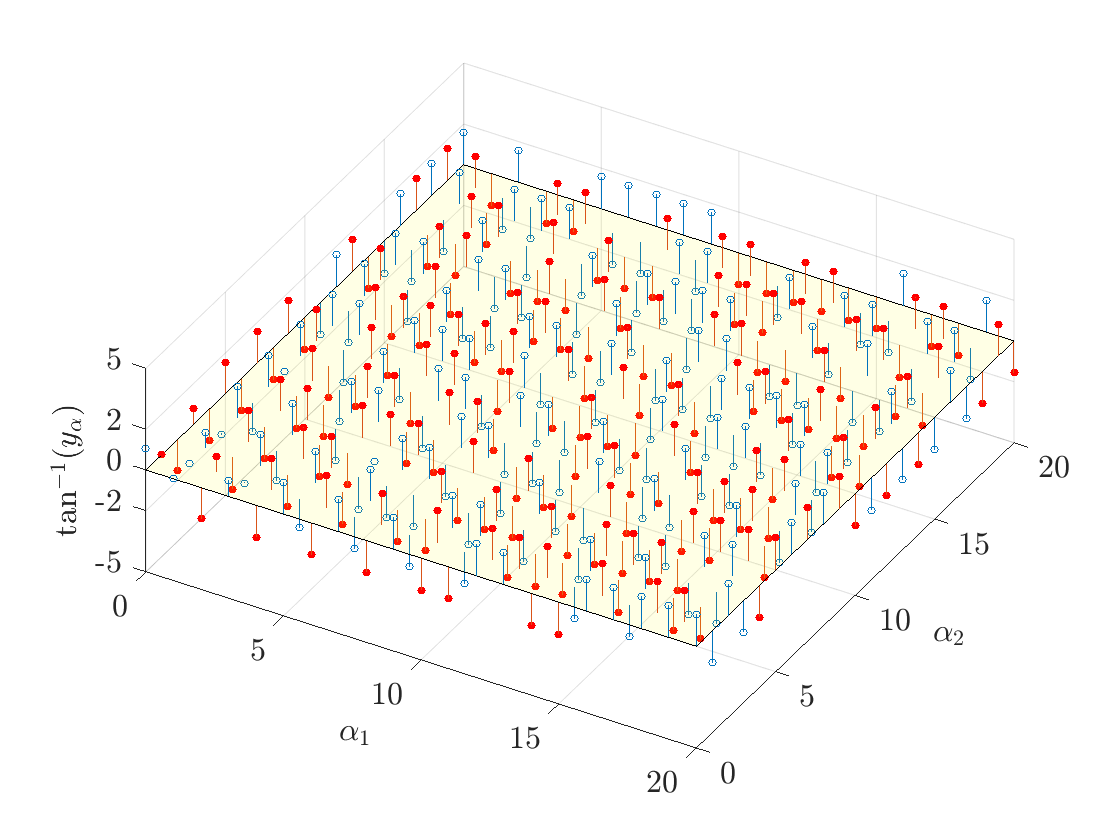}}
   \vspace{-0.2cm}
 \caption{Both are similar to each other. Topological statistics are also different (Figures \ref{moment topology} $\sim$ \ref{moment variances}).}
\label{active neurons for nonzero-nonzero1}
\end{figure}

\begin{figure}[H]
\vspace{-0.2cm}
\centering
    \includegraphics[width=0.82\columnwidth,trim={0.5cm 0cm 0cm 1.5cm},clip]{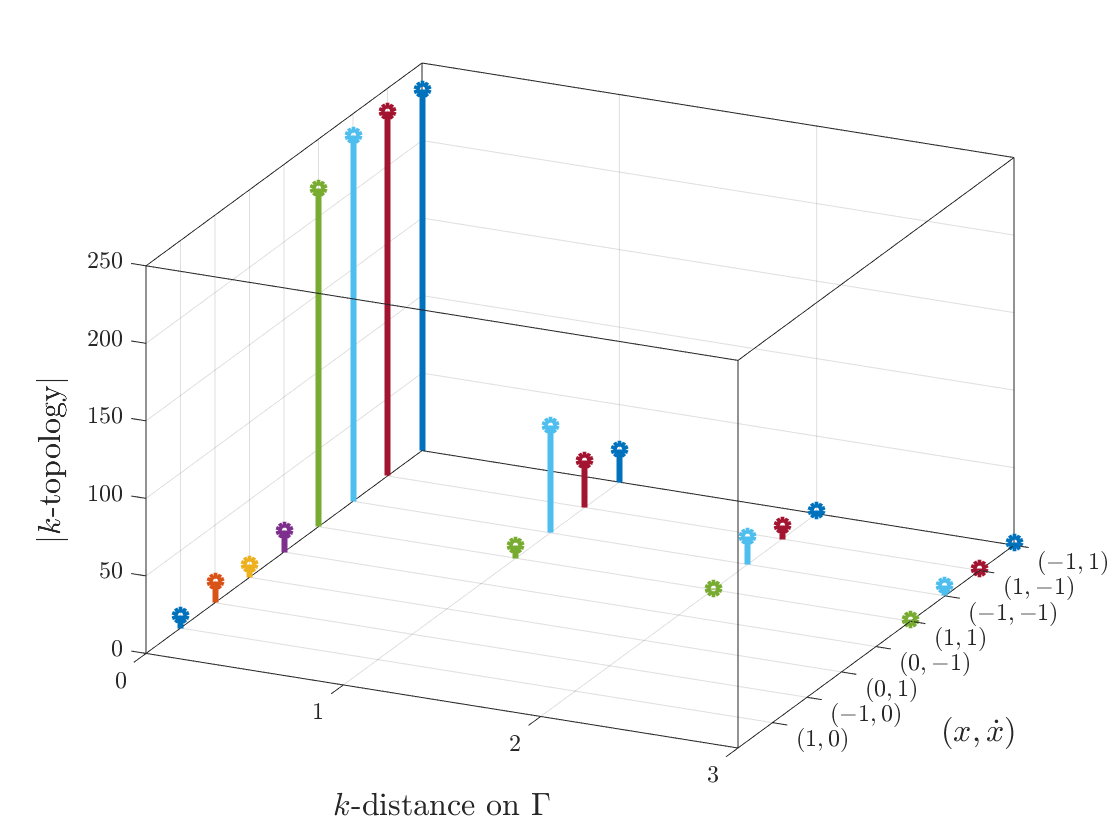}
\vspace{-0.2cm}    
\caption{The number of $k$-balls on each active path.  The numbers of $k$-balls for $(1,0)$, $(0,1)$ are smaller.}
\label{moment topology}
\end{figure}

\begin{figure}[H]
\vspace{-0.2cm}
\centering
    \includegraphics[width=0.82\columnwidth,trim={0.5cm 0cm 0cm 1.5cm},clip]{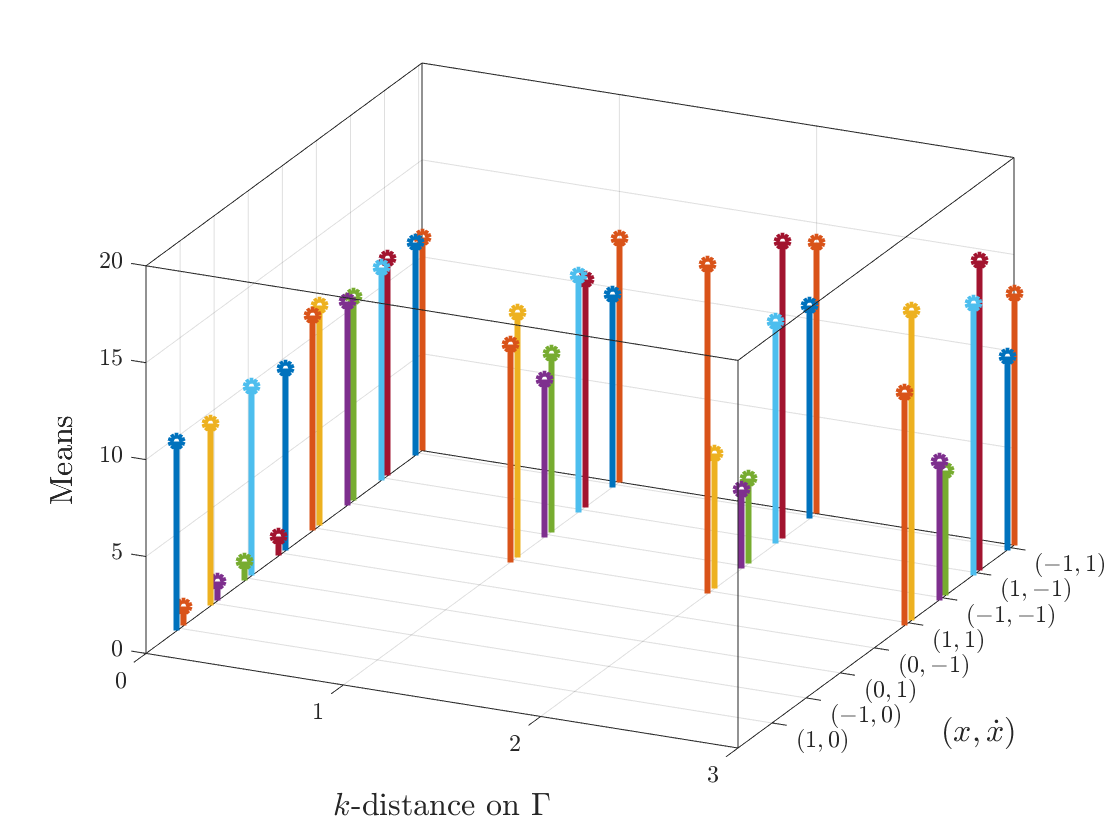}
\vspace{-0.2cm}    
\caption{The mean-distributions of $k$-balls on each active paths. Means for $(1,0)$, $(-1,0)$ are similar to each other and so do $(0,1)$, $(0,-1)$. In addition, Two groups 
change positions with each other.}
\label{moment means}
\end{figure}

\begin{figure}[H]
\vspace{-0.0cm}
\centering
    \includegraphics[width=0.82\columnwidth,trim={0.5cm 0cm 0cm 1.5cm},clip]{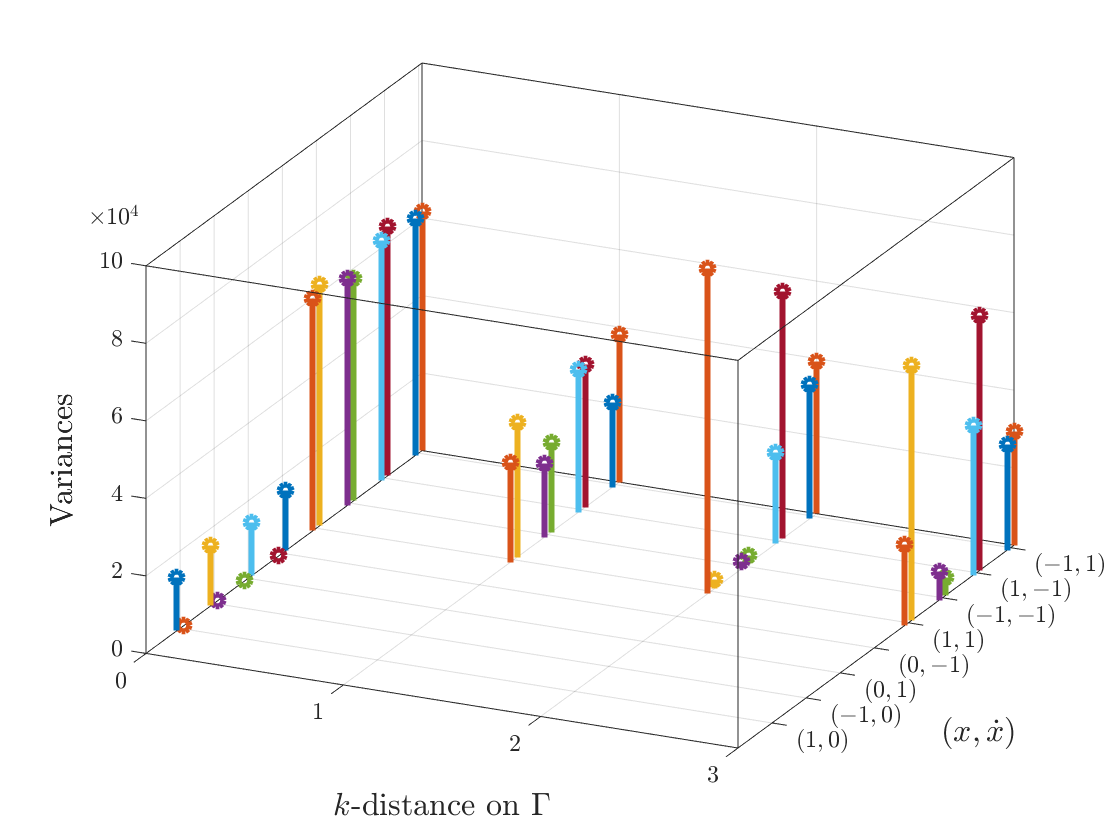}
\vspace{-0.2cm}    
\caption{The variance-distributions of $k$-balls on each active paths. Variances for $(1,0)$, $(-1,0)$ are similar to each other and so do $(0,1)$, $(0,-1)$. In addition, Two groups 
change positions with each other.}
\label{moment variances}
\end{figure}

\smallskip

\section{Conclusion and Further study}

The probabilistic neural network as an unsupervised learning model provides globally unique learning model as an optimal solution for observed samples in probability. The networks are essentially data driven models that do not depend on the
back-propagation because our proposed networks are not trained through iterative computation.
Hereby, these network models would be applied 
in a direction where quality is more important than amount of data in industry, e.g., such as financial service and bioengineering industries. 
These networks provide the learning rates which depend on the drawn samples, and thus,
a probabilistic neural network maximizes the utilization of data. 
The application of  probabilistic neural network technologies, especially for medical image classification, disease prediction, and object recognition, could provide accurate and reliable results.

To appreciate the estimation of artificial intelligence we may need certain physical interpretations
and analysis of active neurons.
For example, a suitable topology on the collection of all active paths or a dictionary containing a meaningful concept, e.g., the standard hearing range 20 to \SI{20,000}{\hertz} for humans,  the $\alpha$-brainwave range  8 to \SI{12.99}{\hertz}, the first-order moments for (conditional) expectations, etc.



\begin{IEEEbiography}[{\includegraphics[width=1in,height=1.25in,clip,keepaspectratio]{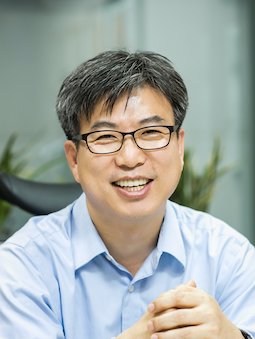}}]{Kyung Soo Rim}
received the B.\,S. degree from Korea University (KU), Korea in 1990, the MS and Ph.\,D. degree in harmonic analysis of mathematics from Korea Advanced Institute of Science and Technology (KAIST), Korea in 1992 and 1996, respectively. 
He is a full professor in the department of mathematics at Sogang University in Seoul, Korea
and served as the Dean of Academic and Admissions Affairs from 2015-18.
His recent research interests focus on the probabilistic machine intelligence related to the Koopman operator of dynamical systems 
including the continuity of convolution and compositions between function spaces.
\end{IEEEbiography}


\begin{IEEEbiography}[{\includegraphics[width=1in,height=1.25in,clip,keepaspectratio]{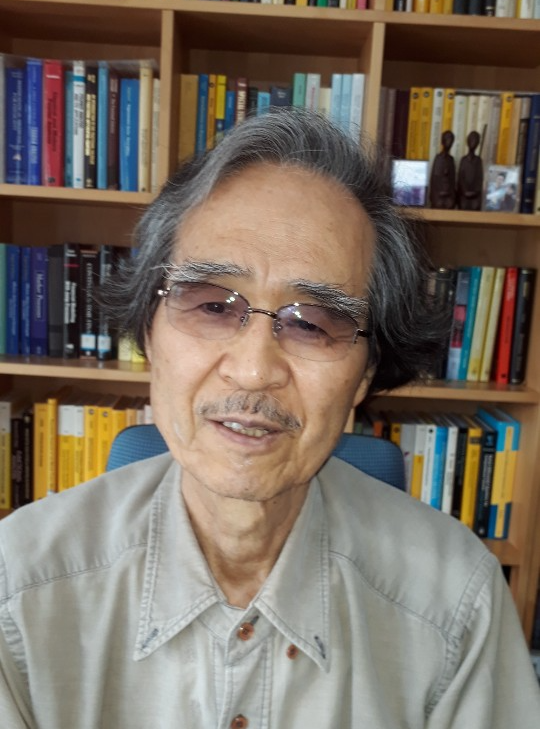}}]{U Jin Choi}
received the B.\,S. degree from Seoul National University (SNU), Korea in 1974 and Ph.\,D. degree from Carnegie Mellon University in 1987, U.\,S.\,A., both in mathematics. He is currently emeritus professor at Korea Advanced Institute of Science and Technology (KAIST) and a member of Research Advisory Committee at National Institute of Mathematical Sciences (NIMS), Korea and  
served as a full professor in the department of mathematical sciences at KAIST until 2011.  
His current research interests are statistics, dynamical systems theory, and machine intelligence in viewpoint of mathematics including stochastic calculus.
\end{IEEEbiography}

\vfill


\begin{thebibliography}{77}
%
\bibitem{arbabi-mezic}
H. Arbabi and I. Mezi\'c.
\textit{Ergodic theory, dynamic mode decomposition, and computation of spectral properties of the Koopman operator}.
SIAM J. Appl. Dyn. Syst. 16(4), 2096--2126, 2017.

\bibitem{bengio-courville-vincent}
Y. Bengio, A. Courville, and P. Vincent.
\textit{Representation learning: A review and new perspectives.}
IEEE Trans.Pattern Anal. Machine Intelli. (35), pp. 1798--1828, 2013.



\bibitem{bishop} 
C. M. Bishop. 
Pattern recognition and machine learning.
\textit{Springer}, 2006.




\bibitem{cybenko}
G. Cybenko.
\textit{Approximation by superpositions of a sigmoidal function}.
Mathematics of Control, Signals and Systems 2, pp. 303--314, 1989


\bibitem{furaoa-hasegawa}
S. Furaoa and O. Hasegawa.
\textit{An incremental network for on-line unsupervised classification and topology learning}.
Neural Networks 19, 90--106, 2006.


\bibitem{geirhos-janssen-schutt-rauber-bethge-wichmann}
R. Geirhos, D. H. J. Janssen,  H. H. Sch{\"u}tt, J. Rauber, M. Bethge, F. A. Wichmann.
\textit{Comparing deep neural networks against humans: object recognition when the signal gets weaker}.
ArXiv abs/1706.06969, 2017


\bibitem{goodfellow-bengio-courville} 
I. Goodfellow, Y. Bengio and A. Courville.
Deep learning.
\textit{MIT Press}, 2016.

\bibitem{grafakos}
L. Grafakos.
Classical and modern Fourier Analysis
\textit{Prentice Hall}, 2003.

\bibitem{haykin} 
S. Haykin. 
Neural networks and learning machines.
\textit{Pearson}, 2016.



\bibitem{hornik}
K. Hornik.
\textit{Approximation capabilities of multilayer feedforward networks}.
Neural Networks 4(2), pp. 251--257, 1991.

\bibitem{korda-mezic}
M. Korda and I.  Mezi\'c. 
\textit{On convergence of extended dynamic mode decomposition to the Koopman operator}.
Jour.  Nonlinear Science 28, pp. 687--710, 2018.



%
%



\bibitem{murphy}
K. P. Murphy.
Machine learning: A probabilistic perspective (Adaptive Computation and Machine Learning series).
\textit{The MIT Press}, 2012.


\bibitem{rowley-mezic-bagheri-schlatter-henningson}
C. W. Rowley, I. Mezi\'c, S. Bagheri, P. Schlatter, and D. S. Henningson.
\textit{Spectral analysis of nonlinear flows}.
Jour. Fluid Mech. 641, pp. 115--127, 2009.




\bibitem{schmidt}
P. Schmid.
\textit{Dynamic mode decomposition of numerical and experimental data}.
Jour. Fluid Mech. 656, pp. 5--28, 2010.

\bibitem{schmidt-sesterhenn}
P. Schmid and J. Sesterhenn.
\textit{Dynamic mode decomposition of numerical and experimental data}.
American Physical Society, 61st Annual Meeting of the APS Division of Fluid Dynamics, November 23--25, 2008, abstract id. MR.007.

\bibitem{shorack-wellner}
G. R. Shorack and J. A. Wellner.
Empirical processes with applications to statistics.
\textit{John Wiley \& Sons, Inc.}, 1986.


\bibitem{stein}
E. M. Stein.
Harmonic Analysis: Real-variable methods, orthogonality, and oscillatory Integrals,
\textit{Princeton University Press},
1993.

\bibitem{tu-rowley-luchtenburg-brunton-kutz}
J. H. Tu, C. W. Rowley, D. M. Luchtenburg, S. L. Brunton,  and  J. N. Kutz.
\textit{On dynamic mode decomposition: Theory and applications}.
Jour.  Comp. Dynamics 1(2), pp. 391--421, 2014.

\bibitem{watson}
D. Watson.
\textit{The rhetoric and reality of anthropomorphism in artificial intelligence}.
Minds and Machines 29, pp. 417--440, 2019.

\end{thebibliography}
\end{document}